\documentclass[psamsfonts,reqno]{amsart}
\usepackage{amssymb,eucal}
\usepackage{graphicx}

\newcommand{\Cuts}{\mathrm{Cuts}}
\newcommand{\tj}{\mathtt{j}}
\DeclareMathOperator{\coco}{CC}
\DeclareMathOperator{\Id}{Id}

\usepackage{hyperref}
\hypersetup{backref=true,  
colorlinks=true,
citecolor=green,
linkcolor=green,
anchorcolor=green,
} 

\theoremstyle{plain}

\newtheorem{lemma}{Lemma}[section]
\newtheorem{theorem}[lemma]{Theorem}
\newtheorem{proposition}[lemma]{Proposition}
\newtheorem{corollary}[lemma]{Corollary}
\newtheorem{examplepf}[lemma]{Example}

\newtheorem{claim}{Claim}
\newtheorem*{sclaim}{Claim}

\newtheorem*{stat}{\name}
\newcommand{\name}{testing}

\theoremstyle{definition}
\newtheorem{definition}[lemma]{Definition}
\newtheorem{example}[lemma]{Example}
\newtheorem{problem}{Problem}

\theoremstyle{remark}
\newtheorem{remark}[lemma]{Remark}

\newtheorem*{note}{Note}

\newcommand{\qedc}{{\qed}~{\rm Claim~{\theclaim}.}}
\newcommand{\qedsc}{{\qed}~{\rm Claim.}}

\newenvironment{cproof}
{\begin{proof}[Proof of Claim.]} {\qedc\renewcommand{\qed}{}\end{proof}}

\newenvironment{scproof}
{\begin{proof}[Proof of Claim.]} {\qedsc\renewcommand{\qed}{}\end{proof}}

\numberwithin{equation}{section}
\numberwithin{figure}{section}

\newcommand{\pup}[1]{\textup{(}{#1}\textup{)}}
\newcommand{\tvi}{\vrule height 12pt depth 6pt width 0pt}
\newcommand{\jirr}{join-ir\-re\-duc\-i\-ble}
\newcommand{\mirr}{meet-ir\-re\-duc\-i\-ble}
\newcommand{\Mirr}{Meet-ir\-re\-duc\-i\-ble}

\newcommand{\jsd}{join-sem\-i\-dis\-trib\-u\-tive}
\newcommand{\jsdy}{join-sem\-i\-dis\-trib\-u\-tiv\-i\-ty}
\newcommand{\msd}{meet-sem\-i\-dis\-trib\-u\-tive}
\newcommand{\msdy}{meet-sem\-i\-dis\-trib\-u\-tiv\-i\-ty}

\newcommand{\Msdy}{Meet-sem\-i\-dis\-trib\-u\-tiv\-i\-ty}
\newcommand{\jsubsemi}{join-sub\-sem\-i\-lat\-tice}
\newcommand{\RSD}[1]{\ensuremath{(\mathrm{RSD}_{#1})}}

\newcommand{\contr}{a contradiction}

\newcommand{\pI}[1]{\bigl({#1}\bigr)}
\newcommand{\pII}[1]{\Bigl({#1}\Bigr)}

\newcommand{\set}[1]{\{#1\}}
\newcommand{\setm}[2]{\set{#1\mid#2}}
\newcommand{\Set}[1]{\left\{#1\right\}}
\newcommand{\Setm}[2]{\Set{#1\mid#2}}
\newcommand{\seq}[1]{\langle{#1}\rangle}

\newcommand{\vecm}[2]{(#1\mid#2)}

\newcommand{\fS}{\mathfrak{S}}
\newcommand{\rZ}{\mathrm{Z}}

\newcommand{\co}[1]{\left[{#1}\right[}

\newcommand{\so}[1]{\boldsymbol{\delta}_{#1}}

\newcommand{\orth}[1]{{#1}^{\perp}}
\newcommand{\oorth}[1]{{#1}^{\perp\perp}}
\DeclareMathOperator{\conv}{conv}
\DeclareMathOperator{\card}{card}

\DeclareMathOperator{\Clop}{Clop}
\DeclareMathOperator{\Reg}{Reg}
\DeclareMathOperator{\Regop}{Reg_{op}}
\DeclareMathOperator{\tcl}{cl}
\DeclareMathOperator{\tin}{int}

\DeclareMathOperator{\Ji}{Ji}
\DeclareMathOperator{\Mi}{Mi}

\newcommand{\ga}{\alpha}
\newcommand{\gb}{\beta}

\newcommand{\gf}{\varphi}

\newcommand{\cgf}{\check{\gf}}

\newcommand{\go}{\omega}
\newcommand{\gO}{\Omega}

\newcommand{\op}{\mathrm{op}}

\newcommand{\cC}{\mathcal{C}}
\newcommand{\cD}{\mathcal{D}}

\newcommand{\cH}{\mathcal{H}}
\newcommand{\cK}{\mathcal{K}}
\newcommand{\cM}{\mathcal{M}}
\newcommand{\cP}{\mathcal{P}}
\newcommand{\cR}{\mathcal{R}}

\newcommand{\KK}{\mathbb{K}}
\newcommand{\RR}{\mathbb{R}}
\renewcommand{\SS}{\mathbb{S}}

\newcommand{\scp}[2]{\seq{{#1}\mid{#2}}}
\newcommand{\sep}[2]{\operatorname{sep}({#1},{#2})}
\DeclareMathOperator{\Pos}{Pos}
\newcommand{\Dr}{\mathbin{D}}

\newcommand{\dnw}{\mathbin{\downarrow}}
\newcommand{\ddnw}{\mathbin{\downdownarrows}}
\newcommand{\upw}{\mathbin{\uparrow}}

\newcommand{\utr}{\trianglelefteq}

\newcommand{\eps}{\varepsilon}
\DeclareMathOperator{\Pow}{Pow}
\newcommand{\cpl}{\mathsf{c}}
\newcommand{\ol}[1]{\overline{{#1}}}
\newcommand{\BX}[1]{\boxed{\textbf{#1}}}
\newcommand{\OUT}[1]{$\underset{{\notin\so{H}}}{\scriptstyle{#1}}$}
\newcommand{\id}{\mathrm{id}}

\newcommand{\js}{join-sem\-i\-lat\-tice}

\newcommand{\es}{\varnothing}

\newlength{\cuplength}
\newcommand{\prt}{\sqcup}
\newcommand{\bprt}{\bigsqcup}

\newcommand{\sB}{\mathsf{B}}

\newcommand{\sL}{\mathsf{L}}
\newcommand{\sM}{\mathsf{M}}

\newcommand{\sP}{\mathsf{P}}
\newcommand{\sR}{\mathsf{R}}
\newcommand{\sS}{\mathsf{S}}

\newcommand{\ba}{\boldsymbol{a}}
\newcommand{\bb}{\boldsymbol{b}}

\newcommand{\bc}{\boldsymbol{c}}
\newcommand{\bd}{\boldsymbol{d}}
\newcommand{\be}{\boldsymbol{e}}

\newcommand{\bp}{\boldsymbol{p}}
\newcommand{\bq}{\boldsymbol{q}}
\newcommand{\bu}{\boldsymbol{u}}
\newcommand{\bv}{\boldsymbol{v}}

\newcommand{\bx}{\boldsymbol{x}}
\newcommand{\by}{\boldsymbol{y}}
\newcommand{\bz}{\boldsymbol{z}}

\title[Lattices of regular closed sets]{Lattices of regular closed subsets of closure spaces}

\author[L. Santocanale]{Luigi Santocanale}
\address{Laboratoire d'Informatique Fondamentale de Marseille\\
Universit\'e de Provence\\
39 rue F. Joliot Curie\\
13453 Marseille Cedex 13\\
France}
\email{luigi.santocanale@lif.univ-mrs.fr}
\urladdr{http://www.lif.univ-mrs.fr/~lsantoca/}

\author[F. Wehrung]{Friedrich Wehrung}
\address{LMNO, CNRS UMR 6139\\
D\'epartement de Math\'ematiques\\
Universit\'e de Caen\\
14032 Caen Cedex\\
France}
\email{friedrich.wehrung01@unicaen.fr}
\urladdr{http://www.math.unicaen.fr/~wehrung}



\date{\today}

\begin{document}

\begin{abstract}
  For a closure space $(P,\gf)$ with $\gf(\es)=\es$,
  the closures of open subsets of~$P$, called the \emph{regular
    closed} subsets, form an ortholattice $\Reg(P,\gf)$, extending the
  poset $\Clop(P,\gf)$ of all clopen subsets. If $(P,\gf)$ is a finite
  convex geometry, then $\Reg(P,\gf)$ is pseudocomplemented. The
  Dedekind-MacNeille completion of the poset of regions of any central
  hyperplane arrangement can be obtained in this way, hence it is
  pseudocomplemented.  The lattice $\Reg(P,\gf)$ carries a
  particularly interesting structure for special types of convex
  geometries, that we call \emph{closure spaces of semilattice
    type}. For finite such closure spaces, 
\begin{itemize}
\item[---] $\Reg(P,\gf)$ satisfies an infinite collection of stronger and stronger quasi-identities, weaker than both meet- and \jsdy. Nevertheless it may fail sem\-i\-dis\-trib\-u\-tiv\-i\-ty.

\item[---] If $\Reg(P,\gf)$ is semidistributive, then it is a bounded homomorphic image of a free lattice.

\item[---] $\Clop(P,\gf)$ is a lattice if{f} every regular closed set is clopen.
\end{itemize}
The extended permutohedron~$\sR(G)$ on a graph~$G$, and the extended permutohedron~$\Reg S$ on a \js~$S$, are both defined as lattices of regular closed sets of suitable closure spaces. While the lattice of regular closed sets is, in the semilattice context, always the Dedekind Mac-Neille completion of the poset of clopen sets, this does not always hold in the graph context, although it always does so for finite block graphs and for cycles. Furthermore, both~$\sR(G)$ and~$\Reg S$ are bounded homomorphic images of free lattices.
\end{abstract}

\maketitle

\tableofcontents

\section{Introduction}\label{S:Intro}
The lattice of permutations $\sP(n)$, also known as the permutohedron,
even if well known and studied in combinatorics, is a relatively young
object of study from a pure lattice-theoretical perspective. Its
elements, the permutations of~$n$ elements, are endowed with the weak Bruhat order; this order turns out to be a lattice.

There are many possible generalization of this order, arising from the
theory of Coxeter groups (Bj\"orner~\cite{Bjo83}), from graph and
order theory (Pouzet \emph{et al.} \cite{PRRZ}, Santocanale and
Wehrung~\cite{SaWe12a}; see also Section \ref{S:PermGraph}), from
language theory (Flath~\cite{Fla93}, Bennett and
Birkhoff~\cite{BB94}), from geometry (Edelman~\cite{Edel84},
  Bj\"orner, Edelman, and Ziegler \cite{BEZ90}, Reading~\cite{Read03}).

While trying to understand those generalizations in a unified framework, we observed that the most noticeable property of permutohedra---at least from a lattice-theoretical perspective---is
that they arise as lattices of clopen (that is, closed and open) subsets for a closure operator. We started thus investigating this kind of
construction.

While closed subsets of a closure space naturally form a lattice when
ordered under subset inclusion, the same need not be true for clopen
subsets. Yet, we can tune our attention to a larger kind of subsets, the closures of open subsets, called here regular
closed subsets; they always form, under subset inclusion, a
lattice. Thus, for a closure space $(P,\gf)$, we denote by
$\Reg(P,\gf)$ the lattice of regular closed subsets of $P$;
$\Reg(P,\gf)$ is then an orthocomplemented lattice, which contains a
copy of $\Clop(P,\gf)$, the poset of all clopen subsets of $P$. There
are many important classes of closure spaces $(P,\gf)$ for which
$\Reg(P,\gf)$ is the Dedekind-MacNeille completion of
$\Clop(P,\gf)$. One of them is the closure space giving rise to
relatively convex subsets of real affine spaces
(cf. Corollary~\ref{C:SepExtAffDmcN}). As a particular case, we
describe the Dedekind-MacNeille completion~$L$ of the poset of regions
of any central hyperplane arrangement as the lattice of all regular
closed subsets of a convex geometry of the type above
(Theorem~\ref{T:DMcNPosHB}). This implies, in particular, that the
lattice~$L$ is always pseudocomplemented
(Corollary~\ref{C:DMcNPosHB1}).

After developing some basic properties of $\Reg(P,\gf)$, we restrict our focus to a class of closure spaces $(P,\gf)$ that arise in the concrete examples we have in mind---we call them \emph{closure spaces of semilattice type}. For such closure spaces, $P$ is a poset, and every minimal covering~$\bx$ of $p \in P$, with respect to the closure operator~$\gf$, joins to $p$ (i.e., $p = \bigvee \bx$). A closure space of semilattice type turns out to be
an atomistic convex geometry. For finite such closure spaces, we can prove the following facts:

\begin{itemize}
\item[---] $\Reg(P,\gf)$ satisfies an infinite collection of stronger and stronger quasi-identities, weaker than semidistributivity (cf. Theorem~\ref{T:AlmostSDReg2} and the discussion following). Nevertheless it may fail semidistributivity (cf. Example~\ref{Ex:ClopM3}).
 
\item[---] If $\Reg(P,\gf)$ is semidistributive, then it is a bounded homomorphic image of a free lattice (cf. Theorem~\ref{T:SDReg2bounded}).
  
\item[---] $\Clop(P,\gf)$ is a lattice if{f} $\Clop(P,\gf)=\Reg(P,\gf)$ (cf. Theorem~\ref{T:ClopWFLatt}).
\end{itemize}

While it is reasonable to conjecture that $\Reg(P,\gf)$ is the Dedekind-MacNeille completion of $\Clop(P,\gf)$---and this is actually the case for many examples---we disprove this conjecture in the general case, with various finite counterexamples (cf. Example~\ref{Ex:FinPosClop} and Corollary~\ref{C:NonClopjirr}). Yet we prove that, in the finite case, the inclusion map of $\Clop(P,\gf)$ into $\Reg(P,\gf)$ preserves all existing meets and joins (cf. Theorem~\ref{T:AbundClop}).

We focus then on two concrete examples of closure spaces of semilattice type. In the first case, $P$ is the collection~$\so{G}$ of all nonempty connected subsets of a \emph{graph}~$G$, endowed with set inclusion, while in the second case, $P$ is an arbitrary \emph{\js}, endowed with its natural ordering. In case~$P=\so{G}$, we define the closure operator in such a way that, if $G$ is a Dynkin diagram of type $A_{n}$, then we obtain $\Reg(P,\gf) = \Clop(P,\gf)$ isomorphic to the permutohedron~$\sP(n+1)$ (symmetric group on $n+1$ letters, with the weak Bruhat ordering). In case~$P$ is a \js, the closure operator associates to a subset~$\bx$ of~$P$ the \jsubsemi\ of~$P$ generated by~$\bx$, and then we write~$\Reg{P}$ instead of $\Reg(P,\gf)$.

In the finite case and for both classes above, we prove that $\Reg(P,\gf)$ is a bounded homomorphic image of a free lattice (cf. Theorems~\ref{T:FinSemilBded} and~\ref{T:bsPGBded}). For the closure space defined above in $P=\so{G}$,

\begin{itemize}
\item[---] We characterize those graphs~$G$ for which $\Clop(P,\gf)$ is a lattice; these turn out to be the block graphs without any $4$-clique (cf. Theorem~\ref{T:PGLatt}).

\item[---] We give a nontrivial description of the completely \jirr\ elements of $\Reg(P,\gf)$, in terms of so-called \emph{pseudo-ultrafilters} on nonempty connected subsets of~$G$ (cf. Theorem~\ref{T:cltjmujirr}). It follows that if~$G$ has no diamond-contractible induced subgraph, then every completely \jirr\ regular closed set is clopen (cf. Theorem~\ref{T:DiamContr}).

\item[---] It follows that if~$G$ is finite and either a block graph or a cycle, then $\Reg(P,\gf)$ is the Dedekind-MacNeille completion of $\Clop(P,\gf)$ (cf. Corollary~\ref{C:CJIBGorCyc}).

\item[---] We find a finite graph~$G$ for which $\Reg(P,\gf)$ is not the Dedekind-MacNeille completion of $\Clop(P,\gf)$ (cf. Corollary~\ref{C:NonClopjirr}).

\item[--] If~$G$ is a complete graph on seven vertices, we find a regular open subset of~$\so{G}$ which is not a union of clopen subsets (cf. Theorem~\ref{T:hostile7}).

\end{itemize}

For the closure space defined above on a \js~$S$,

\begin{itemize}
\item[---] We give a precise description of the minimal neighborhoods of elements of~$S$ (cf. Theorem~\ref{T:MinNbhdSemil}) and the completely \jirr\ elements of~$\Reg{S}$ (cf. Theorem~\ref{T:JirrRegSemil}), in terms of differences of ideals of~$S$. It follows that these sets are all clopen.

\item[---] We prove that every open subset of~$S$ is a union of clopen subsets of~$S$, thus that~$\Reg S$ is the Dedekind-MacNeille completion of~$\Clop S$ (cf. Corollary~\ref{C:MinNbhdSemil2}).

\item[---] It follows that $\Reg S=\Clop S$ if{f} $\Clop S$ is a lattice, if{f} $\Clop S$ is a complete sublattice of~$\Reg S$ (cf. Corollary~\ref{C:MinNbhdSemil3}).
\end{itemize}

We illustrate our paper with many examples and counterexamples.

\section{Basic concepts}
\label{S:Basic}

We refer the reader to Gr\"atzer \cite{LTF} for basic facts and notation about lattice theory.

We shall denote by~$0$ (resp., $1$) the least
(resp., largest) element of a partially ordered set (from now on
\emph{poset})~$(P,\leq)$, if they exist. For subsets~$\ba$ and~$\bx$ in a poset~$P$, we shall set 
 \begin{align*}
 \ba\dnw\bx&=
 \setm{p \in \ba}
 {(\exists x\in\bx)(p\leq x)}\,,\\
 \ba\ddnw\bx&=
 \setm{p \in \ba}
 {(\exists x\in\bx)(p<x)}\,,\\
  \ba\upw\bx&=
 \setm{p \in \ba}
 {(\exists x\in\bx)(p\geq x)}\,.
 \end{align*}
We shall say that~$\bx$ is a \emph{lower subset} of~$P$ if $\bx=P\dnw\bx$. For $x \in P$, we shall write $\ba\dnw x$ ($\ba\ddnw x$, $\ba\upw x$, respectively) instead of $\ba\dnw \set{x}$ ($\ba\ddnw\set{x}$, $\ba\upw\set{x}$, respectively). For posets~$P$ and~$Q$, a map
$f\colon P\to Q$ is \emph{isotone} (resp., \emph{antitone}) if $x\leq
y$ implies that $f(x)\leq f(y)$ (resp., $f(y)\leq f(x)$), for all
$x,y\in P$.

A \emph{lower cover} of an element $p\in P$ is an element $x\in P$ such that $x<p$ and there is no~$y$ such that $x<y<p$; then we write $x\prec p$. If~$p$ has a unique lower cover, then we shall denote this element by~$p_*$. \emph{Upper covers}, and the notation~$p^*$, are defined dually.
A nonzero element~$p$ in a \js~$L$ is \emph{\jirr} if $p=x\vee y$ implies that $p\in\set{x,y}$, for all $x,y\in L$. We say that~$p$ is
\emph{completely \jirr} if it has a unique lower cover~$p_{*}$ and every element $y < p$ is such that $y \leq p_{*}$. \emph{\Mirr}
and \emph{completely \mirr} elements are defined dually. We denote by~$\Ji L$ (resp., $\Mi L$) the set of all \jirr\ (resp., \mirr) elements of~$L$.

Every completely \jirr\ element is \jirr\ and, in a finite lattice, the two concepts are equivalent.  A lattice~$L$ is \emph{spatial} if
every element of $L$ is a (possibly infinite) join of completely \jirr\ elements of $L$. Equivalently, a lattice $L$ is spatial if, for all $a,b\in L$, $a\nleq b$ implies that there exists a completely
\jirr\ element~$p$ of~$L$ such that $p\leq a$ and $p\nleq b$. For a
completely \jirr\ element~$p$ and a completely \mirr\ element~$u$
of~$L$, let $p\nearrow u$ hold if $p\leq u^*$ and $p\nleq u$.
Symmetrically, let $u\searrow p$ hold if $p_*\leq u$ and $p\nleq u$.
The \emph{join-dependency relation}~$\Dr$ is defined on completely
\jirr\ elements by
 \[
 p\Dr q\quad\text{if}\quad
 \pI{p\neq q\text{ and }(\exists x)
 (p\leq q\vee x
 \text{ and }p\nleq q_*\vee x}\,.
 \]
 It is well-known (cf. Freese, Je\v{z}ek, and Nation
 \cite[Lemma~11.10]{FJN}) that the join-dependency relation~$\Dr$ on a
 finite lattice~$L$ can be conveniently expressed in terms of the arrow relations~$\nearrow$ and~$\searrow$ between~$\Ji L$ and~$\Mi L$, as stated in the next Lemma. 

\begin{lemma}\label{L:Arr2D}
  Let $p$, $q$ be distinct \jirr\ elements in a finite lattice~$L$. Then $p\Dr q$ if{f} there exists $u\in\Mi L$ such that $p\nearrow u\searrow q$.
\end{lemma}

A lattice~$L$ is \emph{\jsd} if $x\vee z=y\vee z$ implies that $x\vee z=(x\wedge y)\vee z$, for all $x,y,z\in L$. \emph{\Msdy} is defined dually. A lattice is \emph{semidistributive} if it is both join- and \msd.

A lattice~$L$ is a \emph{bounded homomorphic image of a free lattice} if there are a free lattice~$F$ and a surjective lattice homomorphism
$f\colon F\twoheadrightarrow L$ such that~$f^{-1}\set{x}$ has both a least and a largest element, for each $x\in L$. These lattices, introduced by McKenzie in~\cite{McKe72}, play a key role in the theory of lattice varieties; often called ``bounded'', they are not to be confused with lattices with both a least and a largest element. A finite lattice is bounded (in the sense of McKenzie) if{f} the
join-dependency relations on~$L$ and its dual lattice are both cycle-free (cf. Freese, Je\v{z}ek, and Nation \cite[Corollary~2.39]{FJN}). Every bounded lattice is semidistributive (cf. Freese, Je\v{z}ek, and Nation \cite[Theorem~2.20]{FJN}), but the converse fails, even for finite lattices (cf. Freese, Je\v{z}ek, and Nation \cite[Figure~5.5]{FJN}).

An \emph{orthocomplementation} on a poset~$P$ with least and largest element is a map $x\mapsto\orth{x}$ of~$P$ to itself such that
\begin{itemize}
\item[(O1)] $x\leq y$ implies that $\orth{y}\leq\orth{x}$,

\item[(O2)] $\oorth{x}=x$,

\item[(O3)] $x\wedge\orth{x}=0$ (in view of~(O1) and~(O2), this is equivalent to $x\vee\orth{x}=1$),
\end{itemize}
for all $x,y\in P$. Elements $x,y\in P$ are \emph{orthogonal} if $x\leq\orth{y}$, equivalently $y\leq\orth{x}$.

An \emph{orthoposet} is a poset endowed with an orthocomplementation. Of course, any orthocomplementation of~$P$ is a dual automorphism of~$(P,\leq)$. In particular, if~$P$ is a lattice, then \emph{de Morgan's rules}
 \[
 \orth{(x\vee y)}=\orth{x}\wedge\orth{y}\,,\quad
 \orth{(x\wedge y)}=\orth{x}\vee\orth{y}
 \]
hold for all $x,y\in P$. An \emph{ortholattice} is a lattice endowed with an orthocomplementation.

The \emph{parallel sum} $L=A\parallel B$ of lattices~$A$ and~$B$ is defined by adding a top and a bottom element to the disjoint union $A\cup B$.

A \emph{graph} is a structure $(G,\sim)$, 
where~$\sim$ is an irreflexive and symmetric binary relation on the set~$G$. We shall often identify a subset~$X\subseteq G$ with the
corresponding induced subgraph $\pI{X,\sim\cap(X\times X)}$. Let $x,y \in G$; a \emph{path} from $x$ to $y$ in $(G,\sim)$ is a finite sequence $x= z_{0},z_{1},\ldots ,z_{n}= y$ such that $z_{i} \sim z_{i +1}$ for each $i < n$. If the~$z_i$ are distinct and $z_i\sim z_j$ implies $i-j=\pm1$, then we say that the path is \emph{induced}. A subset $X$ of $(G,\sim)$ is \emph{connected} if, for each $x,y \in X$, there exists a path from~$x$ to~$y$ in $X$. A connected subset~$X$ of~$G$ is
\emph{biconnected} if it is connected and $X\setminus\set{x}$ is connected for each $x\in X$. We shall denote by~$\cK_n$ the complete
graph (or \emph{clique}) on~$n$ vertices, for any positive integer~$n$.

We say that~$G$ is a \emph{block graph} if each biconnected subset of~$G$ is a clique (\emph{we do not assume that block graphs are connected}). Equivalently, none of the cycles~$\cC_n$, for $n\geq4$, nor the diamond~$\cD$ (cf. Figure~\ref{Fig:Graphs}) embeds into~$G$ as an induced subgraph.

\begin{figure}[htb]
\includegraphics{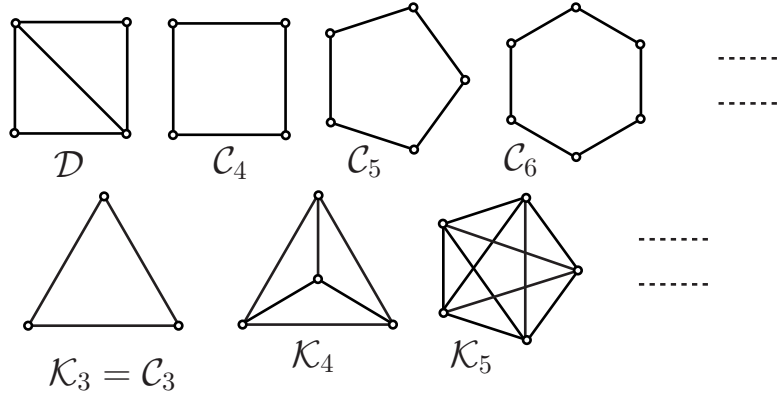}
\caption{Cycles, diamond, and cliques}\label{Fig:Graphs}
\end{figure}

Also, block graphs are characterized as those graphs where there is at
most one induced path between any two vertices, respectively the
graphs where any intersection of connected subsets is connected. For
references, see Bandelt and Mulder~\cite{BaMu86},
Howorka~\cite{Howo79}, Kay and Chartrand~\cite{KaCh65}, and the
wonderful online database \url{http://www.graphclasses.org/}. Block
graphs have been sometimes (for example in Howorka~\cite{Howo79})
called \emph{Husimi trees}.  

We shall denote by~$\Pow X$ the powerset of a set~$X$. For every positive integer~$n$, $[n]$ will denote the set $\set{1,2,\dots,n}$.

\section{Regular closed subsets with respect to a closure operator}\label{S:RegCl}

A \emph{closure operator} on a set~$P$ is usually defined as an extensive, idempotent, isotone map $\gf\colon\Pow P\to\Pow P$; that is, $\bx\subseteq\gf(\bx)$, $\gf(\gf(\bx)) = \gf(\bx)$, and $\gf(\bx)\subseteq \gf(\by)$ if $\bx \subseteq \by$, for all $\bx,\by\subseteq P$. Throughout this paper we
shall require that a closure operator $\gf$ satisfies the additional
condition $\gf(\es)=\es$. A \emph{closure space} is a pair $(P,\gf)$,
where~$\gf$ is a closure operator on~$P$.

We say that the closure space~$(P,\gf)$ is \emph{atomistic} if $\gf(\set{p})=\set{p}$ for each $p\in P$. The associated \emph{kernel} (or \emph{interior}) \emph{operator} is defined by $\cgf(\bx)=P\setminus\gf(P\setminus\bx)$ for each $\bx\subseteq P$. We shall often call~$\gf(\bx)$ the \emph{closure} of~$\bx$ and~$\cgf(\bx)$ the \emph{interior} of~$\bx$. Then both~$\gf$ and~$\cgf$ are idempotent operators, with $\cgf\leq\id\leq\gf$. It is very easy to find examples with $\gf\neq\gf\cgf\gf$. However,

\begin{lemma}\label{L:gfcgfidemp}
The operators $\gf\cgf$ and $\cgf\gf$ are both idempotent. Thus, $\cgf\gf$ is a closure operator on the collection of open sets, and $\gf\cgf$ is a kernel operator on the collection of closed sets.
\end{lemma}

\begin{proof}
Let $\bx\subseteq P$. {}From $\cgf\leq\id$ it follows that $\gf\cgf\gf\cgf(\bx)\subseteq\gf\gf\cgf(\bx)=\gf\cgf(\bx)$. {}From $\id\leq\gf$ it follows that $\gf\cgf\gf\cgf(\bx)\supseteq\gf\cgf\cgf(\bx)=\gf\cgf(\bx)$.

If~$\bx$ is open, then $\bx=\cgf(\bx)\subseteq\gf\cgf(\bx)$. As $\cgf\gf$ is isotone, it follows that $\cgf\gf$ is a closure operator on the collection of all open sets. Dually, $\gf\cgf$ is a kernel operator
on the collection of all closed sets.
\end{proof}

\begin{definition}\label{D:RegCl}
For a closure space $(P,\gf)$, a subset~$\bx$ of~$P$ is
\begin{itemize}
\item[---] \emph{closed} if $\bx=\gf(\bx)$,

\item[---] \emph{open} if $\bx=\cgf(\bx)$,

\item[---] \emph{regular closed} if $\bx=\gf\cgf(\bx)$,

\item[---] \emph{regular open} if $\bx=\cgf\gf(\bx)$,

\item[---] \emph{clopen} if it is simultaneously closed and open.
\end{itemize}
We denote by $\Clop(P,\gf)$ ($\Reg(P,\gf)$, $\Regop(P,\gf)$, respectively) the set of all clopen (regular closed, regular open, respectively) subsets of~$P$, ordered by set inclusion. Due to the condition $\gf(\es)=\es$, the sets~$\es$ and~$P$ are both clopen.
\end{definition}

Of course, a set~$\bx$ is open (closed, regular closed, regular open, clopen, respectively) if{f} its complement $\bx^\cpl=P\setminus\bx$ is closed (open, regular open, regular closed, clopen, respectively). A straightforward application of Lemma~\ref{L:gfcgfidemp} yields the following.

\goodbreak
\begin{lemma}\label{L:BasicReg}
\hfill
\begin{enumerate}
\item A subset~$\bx$ of~$P$ is regular closed if{f} $\bx=\gf(\bu)$ for some open set~$\bu$.

\item The poset~$\Reg(P,\gf)$ is a complete lattice, with meet and join given by
 \begin{align*}
 \bigvee\vecm{\ba_i}{i\in I}&=\gf
 \pI{\bigcup\vecm{\ba_i}{i\in I}}\,,\\
 \bigwedge\vecm{\ba_i}{i\in I}&=\gf\cgf
 \pI{\bigcap\vecm{\ba_i}{i\in I}}\,, 
 \end{align*}
for any family $\vecm{\ba_i}{i\in I}$ of regular closed sets.
\end{enumerate}
\end{lemma}

\begin{remark}\label{Rk:gfunionregop}
The previous Lemma is an immediate consequence of the fact that $\gf\cgf$ is a kernel operator on closed sets. For a direct proof, we need to argue that $\gf\pI{\bigcup\vecm{\ba_i}{i\in I}}$ is regular closed. To this goal, observe that this set is equal to $\gf\pI{\bigcup\vecm{\cgf(\ba_i)}{i\in I}}$ and, more generally,
    \begin{align*}
      \bigvee\vecm{\ba_i}{i\in I}&
      = \gf \pI{\bigcup\vecm{\ba_j}{j\in J} \cup
      \bigcup\vecm{\cgf(\ba_j)}{j \notin J}}\,,
    \end{align*}
whenever~$J$ is a subset of $I$.
\end{remark}

The complement of a regular closed set may not be closed. Nevertheless, we shall now see that there is an obvious ``complementation-like'' map from the regular closed sets to the regular closed sets.

\begin{definition}\label{D:OrthReg}
We define the \emph{orthogonal} of~$\bx$ as $\orth{\bx}=\gf(\bx^\cpl)$, for any $\bx\subseteq P$.
\end{definition}

\begin{lemma}\label{L:BasicOrth}
\hfill
\begin{enumerate}
\item $\orth{\bx}$ is regular closed, for any $\bx\subseteq P$.

\item The assignment $\orth{}\colon\bx\mapsto\orth{\bx}$ defines an orthocomplementation of~$\Reg(P,\gf)$.
\end{enumerate}
\end{lemma}

\begin{proof}
(i). This follows right away from Lemma~\ref{L:BasicReg}(i).

(ii). It is obvious that the map $\orth{}$ is antitone. Now, using Lemma~\ref{L:gfcgfidemp}, we obtain
 \[
 \oorth{\bx}=\gf((\orth{\bx})^\cpl)=\gf(\gf(\bx^\cpl)^\cpl)
 =\gf(\cgf(\bx))=\bx\,,\quad\text{for each }\bx\in\Reg(P,\gf)\,.
 \]
Therefore, $\orth{}$ defines a dual automorphism of the lattice~$\Reg(P,\gf)$. As~$\orth{\bx}$ contains~$\bx^\cpl$, $P=\bx\cup\orth{\bx}$ for each~$\bx\subseteq P$, hence $P=\bx\vee\orth{\bx}$ in case $\bx\in\Reg(P,\gf)$.
\end{proof}

In particular, $\Reg(P,\gf)$ is self-dual. As $\bx\mapsto\bx^\cpl$ defines a dual isomorphism from $\Reg(P,\gf)$ to $\Regop(P,\gf)$, we obtain the following.

\begin{corollary}\label{C:RegP2RegopP}
  Let $(P,\gf)$ be a closure space. Then the lattices~$\Reg(P,\gf)$
  and $\Regop(P,\gf)$ are both self-dual. Moreover, the maps
  $\cgf\colon\Reg(P,\gf)\to\Regop(P,\gf)$ and
  $\gf\colon\Regop(P,\gf)\to\Reg(P,\gf)$ are mutually inverse
  isomorphisms.
\end{corollary}

As the following result shows, there is nothing special about orthoposets of the form~$\Clop(P,\gf)$, or complete ortholattices of the form~$\Reg(P,\gf)$.

\begin{proposition}\label{P:Mayet}
Let~$(L,0,1,\leq,\orth{})$ be an orthoposet. Then there exists a closure space $(\gO,\gf)$ such that $L\cong\Clop(\gO,\gf)$, and such that, in addition, $\Reg(\gO,\gf)$ is the Dedekind-MacNeille completion of $\Clop(\gO,\gf)$.
\end{proposition}

\begin{proof}[Outline of proof]
We invoke a construction due to Mayet~\cite{Mayet82}, or, equivalently, Katrno\v{s}ka~\cite{Katr82}. As everything needed here is already proved in those papers, we just give an outline of the proof, leaving the details as an exercise.

For such an orthoposet~$L$, we say that a subset $X$ of
$L$ is \emph{anti-orthogonal} if its elements are pairwise
non-orthogonal. We denote then by~$\gO$ the set of all maximal
anti-orthogonal subsets of~$L$. We set $\rZ(p)=\setm{X\in\gO}{p\in
  X}$, for each $p\in L$, and we call the sets~$\rZ(p)$
\emph{elementary clopen}. We define~$\gf(\bx)$ as the intersection of
all elementary clopen sets containing~$\bx$, for
each~$\bx\subseteq\gO$. The pair $(\gO,\gf)$ is a closure space. It turns out that the clopen sets, with respect to that closure space, are exactly the elementary clopen sets. A key property, to be verified in the course of the proof above, is that $\rZ(\orth{p})=\gO\setminus\rZ(p)$ for every $p\in L$. Hence, the assignment $p\mapsto\rZ(p)$ defines an isomorphism from~$(L,0,1,\leq,\orth{})$
onto~$(\Clop(\gO,\gf),\es,\gO,\subseteq,\complement)$, and clopen is the same as elementary clopen.

Every closed set is, by definition, an intersection of clopen sets. Hence, by Lemma~\ref{L:MacNeille}, $\Reg(\gO,\gf)$ is then the Dedekind-MacNeille completion of $\Clop(\gO,\gf)$.
\end{proof}

While Proposition~\ref{P:Mayet} implies that every finite orthocomplemented lattice has the form $\Reg(P,\gf)$, for some finite closure space~$(P,\gf)$, we shall now establish a restriction on $\Reg(P,\gf)$ in case $(P,\gf)$ is a \emph{convex geometry}, that is (cf. Edelman and Jamison~\cite{EdJa}), $\gf(\bx\cup\set{p})=\gf(\bx\cup\set{q})$ implies that~$p=q$, for all $\bx\subseteq P$ and all $p,q\in P\setminus\gf(\bx)$.

Recall that a lattice~$L$ with zero is \emph{pseudocomplemented} if for each $x\in L$, there exists a largest $y\in L$, called the \emph{pseudocomplement} of~$x$, such that $x\wedge y=0$. It is mentioned in Chameni-Nembua and Monjardet~\cite{ChaMon92} (and credited there to a personal communication by Le Conte de Poly-Barbut) that every permutohedron is pseudocomplemented; see also Markowsky~\cite{Mark94}.

While not every orthocomplemented lattice is pseudocomplemented (the easiest counterexample is~$\sM_4$, see Figure~\ref{Fig:M4}), we shall now see that the lattice of regular closed subsets of a finite convex geometry is always pseudocomplemented. Our generalization is formally similar to Hahmann \cite[Lemma~4.17]{Hahm08}, although the existence of a precise connection between Hahmann's work and the present paper remains, for the moment, mostly hypothetical.

\begin{figure}[htb]
\includegraphics{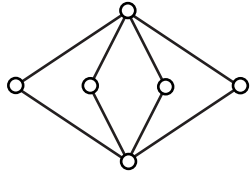}
\caption{The orthocomplemented, non pseudocomplemented lattice~$\sM_4$}\label{Fig:M4}
\end{figure}

We set $\partial\ba=\setm{x\in\ba}{x\notin\gf(\ba\setminus\set{x})}$, for every subset~$\ba$ in a closure space $(P,\gf)$. Observe that $p \in \gf(\bx)$ implies $p \in \bx$, for any $p\in\partial P$ and any $\bx\subseteq P$. It is well-known that $P=\gf(\partial P)$  for any finite convex geometry~$(P,\gf)$ (cf. Edelman and Jamison \cite[Theorem~2.1]{EdJa}), and an easy exercise to find finite examples, with $P=\gf(\partial P)$, which are not convex geometries.

\begin{proposition}\label{P:PseudoCpl}
The lattice~$\Reg(P,\gf)$ is 
pseudocomplemented, for any closure
space $(P,\gf)$ such that $P=\gf(\partial P)$.
In particular, $\Reg(P,\gf)$ is pseudocomplemented in case~$(P,\gf)$ is a finite convex geometry.
\end{proposition}

\begin{proof}
Let $\bb\in\Reg(P,\gf)$ and let $\vecm{\ba_i}{i\in I}$ be a family of elements of~$\Reg(P,\gf)$ with join~$\ba$ such that $\ba_i\wedge\bb=\es$. We must prove that $\ba\wedge\bb=\es$. Suppose otherwise and set $\bd=\cgf(\ba\cap\bb)$. {}From $\ba\wedge\bb=\gf(\bd)$ it follows that $\bd\neq\es$. If $\partial P\cap\bd=\es$, then $\partial P\subseteq\bd^\cpl$, thus, as $P=\gf(\partial P)$ and $\bd^\cpl$ is closed, $\bd=\es$, \contr.

Pick $p\in\partial P\cap\bd$. As
$p\in\bd\subseteq\ba=\gf\pI{\bigcup_{i\in I}\ba_i}$ and $p\in\partial
P$, we get $p\in\ba_i$ for some $i\in I$. Furthermore,
$p\in\bd\subseteq\bb$, so $p\in\ba_i\cap\bb$. On the other hand, from
$\ba_i\wedge\bb=\es$ it follows that $\cgf(\ba_i\cap\bb)=\es$, thus
$\gf(P\setminus(\ba_i\cap\bb))=P$, and thus (as $p\in\partial P$) we
get $p\in P\setminus(\ba_i\cap\bb)$, \contr.
\end{proof}

As we shall see in Example \ref{Ex:InfteMainSD}, the result of the second part of Proposition~\ref{P:PseudoCpl} cannot be extended to the infinite case. Observe that $\Reg(P,\gf)$ being pseudocomplemented is, in the finite case, an immediate consequence of $\Reg(P,\gf)$ being \msd\ (which is here, by self-duality, equivalent to being semidistributive). Example~\ref{Ex:ClopM3} shows that the lattice $\Reg(P,\gf)$ need not be semidistributive, even in case $(P,\gf)$ is a convex geometry.

\section{Regular closed as Dedekind-MacNeille completion of clopen}\label{S:DMcN}

For any closure space $(P,\gf)$, the lattice $\Reg(P,\gf)$ contains
the poset $\Clop(P,\gf)$. It turns out that in many cases, the
inclusion is a Dedekind-MacNeille completion. The following result
will be always used for $K=\Clop(P,\gf)$, except in
Sections~\ref{S:ConvE} and~\ref{S:HypArr}.

\begin{lemma}\label{L:MacNeille}
  The following statements hold, for any closure space $(P,\gf)$ and
  any subset~$K$ of $\Reg(P,\gf)$.
  \begin{enumerate}
  \item $\Reg(P,\gf)$ is the Dedekind-MacNeille completion of~$K$ if{f} every regular closed set is a join of members of~$K$. This occurs, in particular, if every regular open set is a union of members of~$K$.
  
  \item If $\Reg(P,\gf)$ is the
Dedekind-MacNeille completion of~$K$, then every completely \jirr\ element in $\Reg(P,\gf)$ belongs to~$K$. If~$\Reg(P,\gf)$ is spatial, then the converse holds.

\end{enumerate}
\end{lemma}

\begin{proof}
It is well-known that a complete lattice~$L$ is the Dedekind-MacNeille completion of a subset~$K$ if{f} every element of~$L$ is simultaneously a join of elements of~$K$ and a meet of elements of~$K$ (cf. Davey and Priestley \cite[Theorem~7.41]{DP02}). Item~(i) follows easily.

If~$\Reg(P,\gf)$ is the Dedekind-MacNeille completion of~$K$, then every element of~$\Reg(P,\gf)$ is a join of elements of~$K$, thus every completely \jirr\ element of~$\Reg(P,\gf)$ belongs to~$K$. Conversely, if $\Reg(P,\gf)$ is spatial and every completely \jirr\ element in $\Reg(P,\gf)$ belongs to~$K$, then every element of~$\Reg(P,\gf)$ is a join of members of~$K$, thus, using the orthocomplementation, also a meet of clopen subsets. By~(i), the conclusion of~(ii) follows.
\end{proof}

\begin{definition}\label{D:tight}
A subset~$K$ of a poset~$L$ is \emph{tight} in~$L$ if the inclusion map preserves all existing (not necessarily finite) joins and meets. Namely,
 \begin{align}
 a=\bigvee X\text{ in }K&\Longrightarrow
 a=\bigvee X\text{ in }L\,,\quad
 \text{for all }
 a\in K\text{ and all }X\subseteq K\,.
 \label{Eq:PresCreatJJ}\\
 a=\bigwedge X\text{ in }K&
 \Longrightarrow
 a=\bigwedge X\text{ in }L\,, \quad
 \text{for all }
 a\in K\text{ and all }X\subseteq K\,.
 \label{Eq:PresCreatMM}
 \end{align}
\end{definition}

It is well-known (and quite easy to verify) that if the lattice~$L$ is
the Dedekind-MacNeille completion of the poset~$K$, then~$K$ is tight in~$L$. We shall observe---see Theorem~\ref{T:AbundClop}---that $\Clop(P,\gf)$ is often tight in $\Reg(P,\gf)$; in particular this holds if~$P$ is a finite set and~$\gf$ has \emph{semilattice type}, as defined later in~\ref{D:SemilType}. Yet, even under those additional assumptions, there are many examples where $\Reg(P,\gf)$ is not the Dedekind-MacNeille completion of~$\Clop(P,\gf)$ (even in the finite case), see Examples~\ref{Ex:roP} and~\ref{Ex:FinPosClop}.

In our next section, we shall discuss a well-known class of convex geometries for which $\Reg(P,\gf)$ is always the Dedekind-MacNeille completion of $\Clop(P,\gf)$.

\section{Convex subsets in affine spaces}\label{S:ConvE}

Denote by $\conv(X)$ the convex hull of a subset~$X$ in any left affine space~$\Delta$ over a linearly ordered division ring~$\KK$. For a subset~$E$ of~$\Delta $, the \emph{convex hull operator relatively to~$E$} is the map $\conv_E\colon\Pow{E}\to\Pow{E}$ defined by
 \[
 \conv_E(X)=
 \conv(X)\cap E\,,\quad
 \text{for any }X\subseteq E\,.
 \]
The map $\conv_E$ is a closure operator on~$E$. It is well-known that $(E,\conv_E)$ is an atomistic convex geometry (cf. Edelman and Jamison \cite[Example~I]{EdJa}). The fixpoints of~$\conv_E$ are the \emph{relatively convex subsets} of~$E$. The poset $\Clop(E,\conv_E)$ consists of the \emph{relatively bi-convex} subsets of~$E$, that is, those $X\subseteq E$ such that both~$X$ and $E\setminus X$ are relatively convex; equivalently, $\conv_E(X)\cap\conv_E(E\setminus X)=\es$.
A subset~$X$ of~$E$ is \emph{strongly bi-convex} (relatively to~$E$) if $\conv(X)\cap\conv(E\setminus X)=\es$. We denote by $\Clop^*(E,\conv_E)$ the set of all strongly bi-convex subsets of~$E$. This set is contained in $\Clop(E,\conv_E)$, and the containment may be proper.

An \emph{extended affine functional on~$\Delta$} is a map of the form $\ell\colon\Delta\to{}^*\KK$, where~${}^*\KK$ is a ultrapower of~$\KK$ and
 \[
 \ell((1-\lambda)x+\lambda y)=
 (1-\lambda)\ell(x)+ \lambda\ell(y)\,,
 \quad\text{for all }x,y\in\Delta
 \text{ and all }\lambda\in\KK\,.
 \]
If ${}^*\KK=\KK$, then we say that~$\ell$ is an \emph{affine functional} on~$\Delta$.

\begin{lemma}\label{L:SepAff}
Let $X\subseteq E\subseteq\Delta$, with~$E$ finite, and let $p\in E\setminus\conv(X)$. Then there exists an affine functional $\ell\colon\Delta\to\KK$ such that
\begin{enumerate}
\item $E\cap\ell^{-1}\set{0}=\set{p}$;

\item $\ell(x)<0$ for each $x\in X$.

\end{enumerate}
\end{lemma}

\begin{proof}
We may assume without loss of generality that~$X$ is a maximal subset of~$E$ with the property that $p\notin\conv(X)$. Since~$X$ is finite, there exists an affine functional~$\ell$ such that $\ell(p)=0$ and $\ell(x)<0$ for each $x\in X$.

Suppose that $\ell(y)=0$ for some $y\in E\setminus\set{p}$. Necessarily, $y\in E\setminus X$, thus, by the maximality assumption on~$X$, we get $p\in\conv(X\cup\set{y})$, so $p=(1-\lambda)x+\lambda y$ for some $x\in\conv(X)$ and some $\lambda\in\KK$ with $0\leq\lambda\leq 1$. {}From $p\neq y$ it follows that $\lambda<1$. Since $\ell(p)=\ell(y)=0$, it follows that $\ell(x)=0$, a contradiction.
\end{proof}

\begin{lemma}\label{L:SepExtAff}
  Let $X\subseteq E\subseteq\Delta$ and let $p\in
  E\setminus\conv(X)$. Then there are a ultrapower~${}^*\KK$ of~$\KK$ and an extended affine functional
  $\ell\colon\Delta\to{}^*\KK$ such that
\begin{enumerate}
\item $E\cap\ell^{-1}\set{0}=\set{p}$;

\item $\ell(x)<0$ for each $x\in X$.

\end{enumerate}
\end{lemma}

\begin{proof}
  Denote by~$I$ the set of all finite subsets of~$E$ and let~$\mathcal{U}$ be a ultrafilter on~$I$ such that $\setm{I\upw F}{F\in I}\subseteq\mathcal{U}$. We denote by~${}^*\KK$
  the ultrapower of~$\KK$ by~$\mathcal{U}$. It follows from
  Lemma~\ref{L:SepAff} that for each $F\in I$, there exists an affine
  functional $\ell_F\colon\Delta\to\KK$ such that
  $F\cap\ell_F^{-1}\set{0}=\set{p}$ and $\ell_F(x)<0$ for each $x\in
  X\cap F$. For each $v\in\Delta$, denote by~$\ell(v)$ the equivalence
  class modulo~$\mathcal{U}$ of the family $\vecm{\ell_F(v)}{F\in
    I}$. Then~$\ell$ is as required.
\end{proof}

Say that an extended affine functional $\ell\colon\Delta\to{}^*\KK$ is
\emph{special}, with respect to a subset~$E$ of~$\Delta$, if
$\ell^{-1}\set{0}\cap E$ is a singleton. A \emph{special relative
  half-space} of~$E$ is a subset of~$E$ of the form $\ell^{-1}[<0]\cap
E$ (where we set $\ell^{-1}[<0]=\setm{x\in E}{\ell(x)<0}$), for some
special affine functional~$\ell$. It is obvious that every special
relative half-space is a strongly bi-convex, proper subset of~$E$. The
converse statement for~$E$ finite is an easy exercise. For~$E$
infinite, the converse may fail (take $\Delta=\RR$, $E=\setm{1/n}{0<n<\omega}\cup\set{0}$, and $X=\set{0}$).

\begin{corollary}\label{C:SepExtAff}
Let $E\subseteq\Delta$. Then every relatively convex subset of~$E$ is the intersection of all special relative half-spaces of~$E$ containing it.
\end{corollary}

\begin{proof}
  Any relatively convex subset~$X$ of~$E$ is trivially contained in
  the intersection~$\widetilde{X}$ of all special relative half-spaces
  of~$E$ containing~$X$. Let $p\in\widetilde{X}\setminus X$. Since~$X$
  is relatively convex, $p\notin\conv(X)$. By Lemma~\ref{L:SepExtAff},
  there are a ultrapower~${}^*\KK$ of~$\KK$ and an
  extended affine functional $\ell\colon\Delta\to{}^*\KK$ such that
  $E\cap\ell^{-1}\set{0}=\set{p}$ and $X\subseteq\ell^{-1}[<0]$. The
  set~$\widetilde{X}$ is, by definition, contained in $\ell^{-1}[<0]$,
  whence $\ell(p)<0$, a contradiction. Therefore, $\widetilde{X}=X$.
\end{proof}

Since every (special) relative half-space of~$\Delta$ is strongly bi-convex, a simple application of Lemma~\ref{L:MacNeille} yields the following.

\begin{corollary}\label{C:SepExtAffDmcN}
Let $E\subseteq\Delta$. Then $\Reg(E,\conv_E)$ is the Dedekind-MacNeille completion of $\Clop^*(E,\conv_E)$ \pup{thus also of $\Clop(E,\conv_E)$}.
\end{corollary}

In addition, we point in the following result a few noticeable features of the completely \jirr\ elements of $\Reg(E,\conv_E)$.

\begin{theorem}\label{T:SepExtAffCplete}
Let $E\subseteq\Delta$. For every completely \jirr\ element~$P$ of $\Reg(E,\conv_E)$, there are $p\in P$, a ultrapower~${}^*\KK$ of~$\KK$, and a special extended affine functional $\ell\colon\Delta\to{}^*\KK$ such that the following statements hold:
\begin{enumerate}
\item $\ell(p)=0$;

\item $P=\ell^{-1}[\geq 0]\cap E$;

\item $P_*=P\setminus\set{p}$;

\item both~$P$ and~$P_*$ are strongly bi-convex.
\end{enumerate}
In particular, the element~$p$ above is unique. Furthermore, if~$E$ is finite, then~$\ell$ can be taken an affine functional \pup{i.e., ${}^*\KK=\KK$}.
\end{theorem}

\begin{proof}
It follows from Corollary \ref{C:SepExtAffDmcN} that~$P$ is strongly bi-convex. Since $E\setminus P$ is a completely \mirr\ element of $\Reg(E,\conv_E)$ and since, by Corollary~\ref{C:SepExtAff}, $E\setminus P$ is an intersection (thus also a meet) of special relative half-spaces of~$E$, $E\setminus P$ is itself a special relative half-space of~$E$, so $E\setminus P=\ell^{-1}[<0]\cap E$ for some special extended affine functional $\ell\colon\Delta\to{}^*\KK$. Denote by~$p$ the unique element of $\ell^{-1}\set{0}\cap E$.

Setting $Q=\ell^{-1}[>0]\cap E=P\setminus\set{p}$, we get $E\setminus Q=\ell^{-1}[\leq0]\cap E$, so~$Q$ is strongly bi-convex as well.

Let $X\in\Reg(E,\conv_E)$ be properly contained in~$P$. If~$X$ is not contained in~$Q$, then $p\in X$, thus $P=Q\cup X$, and thus, \emph{a fortiori}, $P=Q\vee X$ in $\Reg(E,\conv_E)$. Since~$P$ is \jirr\ and since $Q\neq P$, it follows that $X=P$, a contradiction. Therefore, $X\subseteq Q$, thus completing the proof of~(iii).

If~$E$ is finite, then Lemma~\ref{L:SepAff} can be used in place of Lemma~\ref{L:SepExtAff} in the proof of Corollary \ref{C:SepExtAff}, so ``extended affine functional'' can be replaced by ``affine functional'' in the argument above.
\end{proof}

\section{Posets of regions of central hyperplane arrangements}\label{S:HypArr}

In this section we shall fix a positive integer~$d$, together with a \emph{central hyperplane arrangement} in~$\RR^d$, that is, a finite set~$\cH$ of hyperplanes of~$\RR^d$ through the origin. The open set $\RR^d\setminus\bigcup\cH$ has finitely many connected components, of course all of them open, called the \emph{regions}. We shall denote by $\cR$ the set of all regions. Set
 \[
 \sep{X}{Y}=\setm{H\in\cH}
 {H\text{ separates }X\text{ and }Y}\,,
 \quad\text{for all }X,Y\in\cR\,.
 \]
After Edelman~\cite{Edel84}, we fix a distinguished ``base region'' $B$ and define a partial ordering~$\leq_B$ on the set~$\cR$ of all regions, by
 \[
 X\leq_B Y\quad\text{if}\quad
 \sep{B}{X}\subseteq\sep{B}{Y}\,,
 \quad\text{for all }X,Y\in\cR\,.
 \]
The poset $\Pos(\cH,B)=(\cR,\leq_B)$ has a natural orthocomplementation, given by $X\mapsto -X=\setm{-x}{x\in X}$. This poset is not always a lattice, see Bj\"orner, Edelman, and Ziegler \cite[Example~3.3]{BEZ90}.

Denote by $(x,y)\mapsto\scp{x}{y}$ the standard inner product on~$\RR^d$ and pick, for each $H\in\cH$, a vector~$z_H\in\RR^d$, on the same side of~$H$ as~$B$, such that
 \[
 H=\setm{x\in\RR^d}{\scp{z_H}{x}=0}\,.
 \]
Fix $b\in B$. Observing that $\scp{z_H}{b}>0$ for each $H\in\cH$, we may scale~$z_H$ and thus assume that
 \begin{equation}\label{Eq:NormzH}
 \scp{z_H}{b}=1\quad\text{for each}
 \quad H\in\cH\,. 
 \end{equation}
 The set $\Delta=\setm{x\in\RR^d}{\scp{x}{b}=1}$ is an affine
 hyperplane of~$\RR^d$, containing $E=\setm{z_H}{H\in\cH}$. For each $R \in \cR$ and each $z\in E$, the sign of $\scp{z}{x}$, for $x\in R$, is constant. Accordingly, we shall write $\scp{z}{R}>0$ instead of $\scp{z}{x}>0$ for some (every) $x\in R$; and similarly for $\scp{z}{R}<0$.  The following result is contained in Bj\"orner, Edelman, and Ziegler \cite[Remark~5.3]{BEZ90}. Due to~\eqref{Eq:NormzH}, $z_H$ can be expressed in the form $\sum_{i\in I}\lambda_iz_{H_i}$, with all $\lambda_i\geq0$, if{f} it belongs to the convex hull of $\setm{z_{H_i}}{i\in I}$ (i.e., one can take $\sum_{i\in I}\lambda_i=1$); hence our formulation involves convex sets instead of convex cones.

\begin{lemma}\label{L:EmbPosHB}
The assignment $\eps\colon R\mapsto\setm{z\in E}{\scp{z}{R}<0}$ defines an order-isomorphism from $\Pos(\cH,B)$ onto the set $\Clop^*(E,\conv_E)$ of all strongly bi-convex subsets of~$E$.
\end{lemma}

\begin{proof}
For any $r\in R$, the set $\eps(R)=\setm{z\in E}{\scp{z}{r}<0}$ has complement $E\setminus\eps(R)=\setm{z\in E}{\scp{z}{r}\geq0}$, hence it is strongly bi-convex in~$E$. Furthermore, for $R,S\in\cR$, $\eps(R)\subseteq\eps(S)$ if{f} $\scp{z_H}{R}<0$ implies that $\scp{z_H}{S}<0$ for any $H\in\cH$, if{f} $\sep{B}{R}\subseteq\sep{B}{S}$, if{f} $R\leq_BS$; whence~$\eps$ is an order-embedding.

Finally, we must prove that every strongly bi-convex subset~$U$ of~$E$ belongs to the image of~$\eps$. Since $\es=\eps(B)$ and $E=\eps(-B)$, we may assume that $U\neq\es$ and $U\neq E$. Since~$U$ and $E\setminus U$ have disjoint convex hulls in~$\Delta$, there exists an affine functional~$\ell$ on~$\Delta$ such that
 \[
 U=\setm{z\in E}{\ell(z)<0}
 \quad\text{and}\quad
 E\setminus U=
 \setm{z\in E}{\ell(z)>0}\,.
 \]
Let~$c$ be any normal vector to the unique hyperplane of~$\RR^d$ containing $\set{0}\cup\ell^{-1}\set{0}$, on the same side of that hyperplane as $E\setminus U$. Then
 \[
 (\forall z\in U)\pI{\scp{z}{c}<0}\quad
 \text{and}\quad
 (\forall z\in E\setminus U)
 \pI{\scp{z}{c}>0}\,.
 \]
In particular, $c\notin\bigcup\cH$. Furthermore, if~$R$ denotes the unique region such that $c\in R$, we get
 \begin{equation*}
 \eps(R)=
 \setm{z\in E}{\scp{z}{R}<0}
 =\setm{z\in E}{\scp{z}{c}<0}=U\,.
 \tag*{\qed}
 \end{equation*}
\renewcommand{\qed}{}
\end{proof}

According to Lemma~\ref{L:EmbPosHB}, we shall identify $\Pos(\cH,B)$ with the collection of all strongly bi-convex subsets of~$E$.

\begin{theorem}\label{T:DMcNPosHB}
The lattice~$\Reg(E,\conv_E)$ is the Dedekind-MacNeille completion of $\Pos(\cH,B)$ \pup{via the embedding~$\eps$}.
\end{theorem}

\begin{proof}
By Theorem~\ref{T:SepExtAffCplete}, every completely \jirr\ element~$P$ of the lattice $\Reg(E,\conv_E)$ is strongly bi-convex. By Lemma~\ref{L:EmbPosHB}, $P$ belongs to the image of~$\eps$. The conclusion follows then from Lemma~\ref{L:MacNeille}.
\end{proof}

The following corollary is a slight strengthening of Bj\"orner, Edelman, and Ziegler \cite[Theorem~5.5]{BEZ90}, obtained by changing ``bi-convex'' to ``regular closed''.

\begin{corollary}\label{C:DMcNPosHB0}
The poset of regions $\Pos(\cH,B)$ is a lattice if{f} every regular closed subset of~$E$ is strongly bi-convex, that is, it has the form~$\eps(R)$ for some region~$R$.
\end{corollary}

\begin{proof}
By Theorem~\ref{T:DMcNPosHB}, the lattice $\Reg(E,\conv_E)$ is generated by the image of~$\eps$.
\end{proof}

\begin{corollary}\label{C:DMcNPosHB1}
The Dedekind-MacNeille completion of $\Pos(\cH,B)$ is a pseudocomplemented lattice.
\end{corollary}

\begin{proof}
By Theorem~\ref{T:DMcNPosHB}, $\Reg(E,\conv_E)$ is the Dedekind-MacNeille completion of $\Pos(\cH,B)$. Now $(E,\conv_E)$ is a convex geometry, so the conclusion follows from Proposition~\ref{P:PseudoCpl}.
\end{proof}

\begin{remark}\label{Rk:SDHyp}
There are many important cases where $\Pos(\cH,B)$ is a lattice, see Bj\"orner, Edelman, and Ziegler~\cite{BEZ90}. Further lattice-theoretical properties of $\Pos(\cH,B)$ are investigated in Reading~\cite{Read03}. In particular, even if $\Pos(\cH,B)$ is a lattice, it may not be semidistributive (cf. Reading \cite[Figure~3]{Read03}); and even if it is semidistributive, it may not be bounded (cf. Reading \cite[Figure~5]{Read03}).
\end{remark}

\begin{remark}\label{Rk:ConvGenHyp}
We established in Lemma~\ref{L:EmbPosHB} that the poset of all regions of any central hyperplane arrangement with base region is isomorphic to the poset of all strongly bi-convex subsets of some finite set~$E$. Conversely, the collection of all strongly bi-convex subsets of any finite subset~$E$ in any finite-dimensional real affine space~$\Delta$ arises in this fashion. Indeed, embed~$\Delta$ as a hyperplane, avoiding the origin, into some~$\RR^d$, and pick $b\in\RR^d$ such that $\scp{b}{x}=1$ for all $x\in\Delta$. The set~$\cH$ of orthogonals of all the elements of~$E$ is a central hyperplane arrangement of~$\RR^d$, and the set $\setm{z_H}{H\in\cH}$ associated to~$\cH$ and~$b$ as above is exactly~$E$. The corresponding base region~$B$ is the one containing~$b$, that is,
 \[
 B=\setm{x\in\RR^d}
 {\scp{z}{x}>0
 \text{ for each }z\in E}\,.
 \] 
By the discussion above,
$\Pos(\cH,B)\cong\Clop^*(E,\conv_E)$. 
\end{remark}

\section{Closure operators of poset and semilattice type}
\label{S:SemilType}

We study in this section closure spaces $(P,\gf)$ where~$P$ is a poset and the closure operator~$\gf$ is related to the order of $P$; such a relation will make it possible to derive properties of $\Reg(P,\gf)$ from the order. Closure spaces of this kind originate from concrete examples generalizing permutohedra, investigated in further sections.

Let $(P,\gf)$ be a closure space. A \emph{covering} of an element $p\in P$ is a subset~$\bx$ of~$P$ such that $p\in\gf(\bx)$. If~$\bx$ is minimal, with respect to set inclusion, for the property
of being a covering, then we shall say that~$\bx$ is a \emph{minimal covering} of~$p$. We shall denote by~$\cM_{\gf}(p)$, or~$\cM(p)$
if~$\gf$ is understood, the set of all minimal coverings of~$p$. Due to the condition $\gf(\es)=\es$, every minimal covering of an element of~$P$ is nonempty.
We say that a covering~$\bx$ of~$p$ is
\emph{nontrivial} if $p\notin\bx$.

We say that $(P,\gf)$ is \emph{algebraic} if~$\gf(\bx)$ is the union of the~$\gf(\by)$, for all finite subsets~$\by$ of~$\bx$, for any $\bx\subseteq P$. A great deal of the relevance of algebraic closure spaces for our purposes is contained in the following straightforward lemma.

\begin{lemma}\label{L:Alg2MinCov}
Let~$p$ be an element in an algebraic closure space $(P,\gf)$. Then every minimal covering of~$p$ is a finite set, and every covering of~$p$ contains a minimal covering of~$p$.
\end{lemma}

The following trivial observation is quite convenient for the understanding of open sets and the interior operator~$\cgf$.

\begin{lemma}\label{L:OpenSemType}
Let $(P,\gf)$ be an algebraic closure space, let $p\in P$, and let $\ba\subseteq P$. Then $p\in\cgf(\ba)$ if{f} every minimal covering of~$p$ meets~$\ba$.
\end{lemma}

\begin{proof}
Assume first that~$p\in\cgf(\ba)$ and let~$\bx$ be a minimal covering of~$p$. If $\bx\subseteq\ba^\cpl$ then $p\in\gf(\ba^\cpl)$, \contr. Conversely, suppose that $p\notin\cgf(\ba)$, that is, $p\in\gf(\ba^\cpl)$. By Lemma~\ref{L:Alg2MinCov}, there exists $\bx\in\cM(p)$ contained in $\ba^\cpl$.
\end{proof}

\begin{definition}\label{D:SemilType}
We say that an algebraic closure space~$(P,\gf)$, with~$P$ a poset, has
\begin{itemize}
\item[---] \emph{poset type}, if $\bx\subseteq P\dnw p$ whenever $\bx\in\cM_\gf(p)$,
  
\item[---] \emph{semilattice type}, if
    $p=\bigvee\bx$ whenever $\bx\in\cM_\gf(p)$.
\end{itemize}
We say that an algebraic closure space~$(P,\gf)$, with~$P$ just a set, has poset (resp., semilattice) type if there exists a poset structure on~$P$ such that, with respect to that structure, $(P,\gf)$ has poset (resp., semilattice) type.
\end{definition}

\begin{remark}\label{Rk:SemTypeNotSem}
We do not require, in the statement of Definition~\ref{D:SemilType}, that~$P$ be a \js; and indeed, in many important examples, this will not be the case.
\end{remark}

\begin{example}\label{Ex:roP}
Let~$P$ be a poset and set $\gf(\bx)=P\upw\bx=\setm{p\in P}{(\exists x\in\bx)(x\leq p)}$, for each $\bx\subseteq P$. Then $(P,\gf)$ is an algebraic closure space, and the elements of~$\cM(p)$ are exactly the singletons~$\set{q}$ with $q\leq p$, for each $p\in P$. In particular, $(P,\gf)$ has poset type, and it has semilattice type if{f} the ordering of~$P$ is trivial.

The lattice $\Reg(P,\gf)$ turns out to be complete and Boolean. It plays a fundamental role in the theory of set-theoretical forcing, where it is usually called the completion of~$P$ (or the Boolean algebra of all regular open subsets of~$P$), see for example Jech~\cite{Jech}. Any complete Boolean algebra can be described in this form, so~$\Reg(P,\gf)$ may not be spatial. Further, if~$P$ has a largest element, then $\Clop(P,\gf)=\set{\es,P}$, so even in the finite case, $\Reg(P,\gf)$ may not be the Dedekind-MacNeille completion of~$\Clop(P,\gf)$.
\end{example}

Another example of a closure space of poset type, but not of semilattice type, is the following. 

\begin{example}\label{Ex:PosNotSemCl}
Consider a four-element set $P=\set{p_0,p_1,q_0,q_1}$ and define the closure operator~$\gf$ on~$P$ by setting $\gf(\bx)=\bx$ unless $\set{q_0,q_1}\subseteq\bx$, in which case $\gf(\bx)=P$. Then the partial ordering~$\utr$ on~$P$ with the only nontrivial coverings $q_j\utr p_i$, for $i,j<2$, witnesses~$(P,\gf)$ being a closure space of poset type.

As $\set{q_0,q_1}$ is a minimal covering, with respect to~$\gf$, of
both~$p_0$ and~$p_1$, any ordering on~$P$ witnessing $(P,\gf)$ being
of semilattice type would thus satisfy that $p_0=q_0\vee q_1$ and
$p_1=q_0\vee q_1$ (with respect to that ordering), contradicting $p_{0} \neq p_{1}$.
\end{example}

An important feature of closure spaces of poset type, reminiscent of the absence of $D$-cycles in lower bounded homomorphic images of free lattices, is the following.

\begin{lemma}\label{L:DrelClSpSem}
Let $(P,\gf)$ be a closure space of poset type, let $p,q\in P$, and let $\ba\subseteq P$. If $p\in\gf(\ba\cup\set{q})\setminus\gf(\ba)$, then $p\geq q$.
\end{lemma}

\begin{proof}
By Lemma~\ref{L:Alg2MinCov}, there exists $\bx\in\cM(p)$ such that $\bx\subseteq\ba\cup\set{q}$. {}From $p\notin\gf(\ba)$ it follows that $q\in\bx$, while, as $(P,\gf)$ has poset type, $\bx\subseteq P\dnw p$.
\end{proof}

It follows easily from the previous Lemma that \emph{every closure space of poset type is a convex geometry}. For semilattice type, we get the following additional property.

\begin{lemma}\label{L:SemType2Atom}
Every closure space of semilattice type is atomistic.
\end{lemma}

\begin{proof}
If $y \in \gf(\set{x})$, then $\set{x}\in\cM(y)$ (because $\gf(\es)=\es$), so $y=\bigvee \set{x}=x$.
\end{proof}

\begin{example}\label{Ex:PermPoset}
Let~$\be$ be a transitive binary relation on some set. It is obvious that the transitive closure on subsets of~$\be$ gives rise to an algebraic closure operator~$\tau$ on~$\be$, the latter being viewed as a set of pairs. We study in our paper~\cite{SaWe12a} the lattice of all regular closed subsets of~$\be$. In particular, we prove  the following statement: \emph{If~$\be$ is finite and antisymmetric, then the lattice~$\Reg(\be,\tau)$ is a bounded homomorphic image of a free lattice}.
  
If~$\be$ is antisymmetric (\emph{not necessarily reflexive}), then we can define a partial ordering~$\sqsubseteq$ between elements of~$\be$ as follows:
 \begin{equation}\label{Eq:OrdsoPos}
 (x,y)\sqsubseteq(x',y')
 \quad\text{if}\quad
 \pI{x'=x\text{ or }(x',x)\in\be}
 \text{ and }
 \pI{y=y'\text{ or }(y,y')\in\be}\,.
 \end{equation}
We argue next that, with respect to this ordering, $(\be,\tau)$ is a closure space of semilattice type. For $(x,y)\in\be$ and $\bz\in\cM((x,y))$, the pair $(x,y)$ belongs to the (transitive) closure of~$\bz$, hence there exists a subdivision $x=z_0<z_1<\cdots<z_n=y$ such that each $(z_i,z_{i+1})\in\bz$. As $(x,y)$ does not belong to the closure of any proper subset of~$\bz$, it follows that $\bz=\setm{(z_i,z_{i+1})}{i<n}$; whence~$(x,y)$ is
 the join of~$\bz$ with respect to the ordering~$\sqsubseteq$.
 
 In case~$\be$ is the \emph{strict ordering} associated to a partial
 ordering~$(E,\leq)$, the poset $\Clop(\be,\tau)$ is the
 ``permutohedron-like'' poset denoted, in Pouzet~\emph{et
   al.}~\cite{PRRZ}, by~$\mathbf{N}(E)$. In particular, it is proved
 there that~\emph{$\mathbf{N}(E)$ is a lattice if{f}~$E$ contains no
   copy of the two-atom Boolean lattice~$\sB_2$}.  The latter fact is
 extended in our paper~\cite{SaWe12a} to all transitive relations. In
 particular, this holds for the full relation $\be=E\times E$ on any
 set~$E$. The corresponding lattice $\Reg(\be,\tau)=\Clop(\be,\tau)$
 is called the \emph{bipartition lattice} of~$E$. This structure
 originates in Foata and Zeilberger~\cite{FoZe96} and
 Han~\cite{Han96}. Its poset structure is investigated further
 in Hetyei and Krattenthaler~\cite{HetKra11}. However, it can be verified that the closure
 space $(E\times E,\tau)$ does not have poset type if $\card E\geq3$.
\end{example}

\begin{example}\label{Ex:FromJSemil}
Let $(S,\vee)$ be a \js. We set, for any $\bx\subseteq S$, $\tcl(\bx)=\bx^\vee$, the set of joins of all nonempty finite subsets of~$\bx$. The closure lattice of $(S,\tcl)$ is the lattice of all (possibly empty) \jsubsemi{s} of~$S$. We shall call $(S,\tcl)$ the \emph{closure space canonically associated to the \js~$S$}. For any $p\in S$, a nonempty subset $\bx\subseteq S$ belongs to~$\cM(p)$ if{f}~$p$ is the join of a nonempty finite subset of~$\bx$ but of no proper subset of~$\bx$; thus, $\bx$ is finite and $p=\bigvee\bx$.

Therefore, the closure space $(S,\tcl)$ thus constructed has semilattice type. The ortholattice $\Reg{S}=\Reg(S,\tcl)$ and the orthoposet $\Clop{S}=\Clop(S,\tcl)$ will be studied in some detail in the subsequent sections, in particular Sections~\ref{S:HostSemil} and~\ref{S:BdedReg}.
\end{example}

Another large class of examples, obtained from \emph{graphs}, will be studied in more detail in subsequent sections, in particular Sections~\ref{S:PermGraph} and~\ref{S:ClopGraph}.

\section{Minimal neighborhoods in closure spaces}\label{S:MinNbhd}

Minimal neighborhoods are a simple, but rather effective, technical tool for dealing with lattices of closed, or regular closed, subsets of an algebraic closure space.

\begin{definition}\label{D:nbhd}
Let $(P,\gf)$ be a closure space and let $p\in P$. A \emph{neighborhood} of~$p$ is a subset~$\bu$ of~$P$ such that $p\in\cgf(\bu)$.
\end{definition}

Since~$\cgf(\bu)$ is also a neighborhood of~$p$, it follows that every minimal neighborhood of~$p$ is open. The following result gives a simple sufficient condition, in terms of minimal neighborhoods, for $\Reg(P,\gf)$ being the Dedekind-MacNeille completion of $\Clop(P,\gf)$.

\begin{proposition}\label{P:MinNbhdClos}
The following statements hold, for any algebraic closure space $(P,\gf)$.
\begin{enumerate}
\item Every open subset of~$P$ is a union of minimal neighborhoods.

\item Every minimal neighborhood in~$P$ is clopen if{f} every open subset of~$P$ is a union of clopen sets, and in that case, $\Reg(P,\gf)$ is the Dedekind-MacNeille completion of $\Clop(P,\gf)$.
\end{enumerate}
\end{proposition}

\begin{proof}
(i). Let~$\bu$ be an open subset of~$P$. Since $(P,\gf)$ is an algebraic closure space, every downward directed intersection of open sets is open, so it follows from Zorn's Lemma that every element of~$\bu$ is contained in some minimal neighborhood of~$p$.

(ii) follows easily from~(i) together with Lemma~\ref{L:MacNeille}.
\end{proof}

Minimal neighborhoods can be easily recognized by the following test.

\begin{proposition}\label{P:MinNbhd}
The following are equivalent, for any algebraic closure space~$(P,\gf)$, any open subset~$\bu$ of~$P$, and any $p\in\bu$:
\begin{enumerate}
\item $\bu$ is a minimal neighborhood of~$p$.

\item For each $x\in\bu$, there exists $\bx\in\cM(p)$ such that $\bx\cap\bu=\set{x}$.
\end{enumerate}
\end{proposition}

\begin{proof}
(i)$\Rightarrow$(ii). For each $x\in\bu$, the interior $\cgf(\bu\setminus\set{x})$ is a proper open subset of~$\bu$, thus, by the minimality assumption on~$\bu$, it does not contain~$p$ as an element. Since $(P,\gf)$ is an algebraic closure space, there is $\bx\in\cM(p)$ such that $\bx\cap(\bu\setminus\set{x})=\es$. Since~$\bu$ is a neighborhood of~$p$, $\bx\cap\bu\neq\es$, so $\bx\cap\bu=\set{x}$.

(ii)$\Rightarrow$(i). Let~$\bv$ be a neighborhood of~$p$ properly contained in~$\bu$, and pick $x\in\bu\setminus\bv$. By assumption, there is $\bx\in\cM(p)$ such that $\bx\cap\bu=\set{x}$. Since~$\bv$ is a neighborhood of~$p$, $\bx\cap\bv\neq\es$, thus, as $\bv\subseteq\bu$, we get $\bx\cap\bv=\set{x}$, and thus $x\in\bv$, a contradiction.
\end{proof}

\section{Minimal neighborhoods in semilattices}\label{S:HostSemil}

The present section will be devoted to the study of minimal neighborhoods in a \js~$S$, endowed with its canonical closure operator introduced in Example~\ref{Ex:FromJSemil}. It will turn out that those minimal neighborhoods enjoy an especially simple structure.

The following crucial result gives a simple description of minimal neighborhoods (not just minimal regular open neighborhoods) of elements in a \js.

\begin{theorem}\label{T:MinNbhdSemil}
Let~$p$ be an element in a \js~$S$. Then the minimal neighborhoods of~$p$ are exactly the subsets of the form $(S\dnw p)\setminus\ba$, for maximal proper ideals~$\ba$ of~$S\dnw p$. In particular, every minimal neighborhood of~$p$ is clopen.
\end{theorem}

\begin{note}
We allow the empty set~$\es$ as an ideal of~$S$.
\end{note}

\begin{proof}
  Let~$\ba$ be an ideal of~$S\dnw p$. It is straightforward to verify
  that the subset $\bu=(S\dnw p)\setminus\ba$ is clopen. Now, assuming that~$\ba$ is a maximal proper ideal of~$S\dnw p$, we shall prove that~$\bu$ is a minimal neighborhood of~$p$. For each $x\in\bu$, it follows from the maximality assumption on~$\ba$ that~$p$ belongs to the ideal of~$S$ generated by $\ba\cup\set{x}$, thus either $p=x$ or there exists $a\in\ba$ such that $p=a\vee x$. Therefore, the set~$\bx$, defined as~$\set{p}$ in the first case and as~$\set{a,x}$
  in the second case, is a minimal covering of~$p$ that meets~$\bu$
  in~$\set{x}$. By Proposition~\ref{P:MinNbhd}, $\bu$ is a minimal
  neighborhood of~$p$.

Conversely, any minimal neighborhood~$\bu$ of~$p$ is open. Since~$\bu\dnw p$ is a lower subset of~$\bu$, it is open as well, hence $\bu=\bu\dnw p$. Since~$\bu$ is an open subset of the ideal~$S\dnw p$, the subset $\ba=(S\dnw p)\setminus\bu$ is closed, that is, $\ba$ is a subsemilattice of~$S\dnw p$.

\begin{sclaim}
For every $x\in\bu\setminus\set{p}$, there exists $a\in\ba$ such that $p=x\vee a$.
\end{sclaim}

\begin{scproof}
Since~$\bu$ is a minimal neighborhood of~$p$, it follows from Proposition~\ref{P:MinNbhd} that there exists $\bx\in\cM(p)$ such that $\bx\cap\bu=\set{x}$. {}From $x\neq p$ it follows that $\bx\neq\set{x}$, so $\bx\setminus\set{x}$ is a nonempty subset of~$\ba$, and so the element $a=\bigvee(\bx\setminus\set{x})$ is well-defined and belongs to~$\ba$. Therefore,
 \[
 p=\bigvee\bx=
 x\vee\bigvee(\bx\setminus\set{x})=
 x\vee a\,,
 \]
as desired.
\end{scproof}

Now let $x<y$ with $y\in\ba$, and suppose, by way of contradiction, that $x\notin\ba$, that is, $x\in\bu$. By the Claim above, there exists $a\in\ba$ such that $p=x\vee a$, thus $p\leq y\vee a$, and thus, as $\set{y,a}\subseteq\ba\subseteq S\dnw p$, we get $p=y\vee a\in\ba$, a contradiction. Therefore, $x\in\ba$, thus completing the proof that~$\ba$ is an ideal of~$S\dnw p$.

By definition, $p\notin\ba$. For each $x\in(S\dnw p)\setminus\ba$ (i.e., $x\in\bu$), there exists, by the Claim, $a\in\ba$ such that $p=x\vee a$. This proves that there is no proper ideal of~$S\dnw p$ containing $\ba\cup\set{x}$, so~$\ba$ is a maximal proper ideal of~$S\dnw p$.
\end{proof}

Obviously, every set-theoretical difference of ideals of~$S$ is
clopen. By combining Lemma~\ref{L:MacNeille},
Proposition~\ref{P:MinNbhdClos}, and Theorem~\ref{T:MinNbhdSemil}, we obtain the following results (recall that the open subsets of~$S$ are exactly the complements in~$S$ of the \jsubsemi{s} of~$S$).

\begin{corollary}
\label{C:MinNbhdSemil2}
The following statements hold, for any \js~$S$.

\begin{enumerate}
\item Every open subset of~$S$ is a set-theoretical union of differences of ideals of~$S$; thus it is a set-theoretical union of clopen subsets of~$S$.

\item $\Reg S$ is generated, as a complete ortholattice, by the set of all ideals of~$S$.

\item $\Reg S$ is the Dedekind-MacNeille completion of $\Clop S$.

\item Every completely \jirr\ element of~$\Reg S$ is clopen.

\item $\Clop S$ is tight in~$\Reg S$.
\end{enumerate}
\end{corollary}

\begin{corollary}\label{C:MinNbhdSemil3}
The following are equivalent, for any \js~$S$:
\begin{enumerate}
\item $\Clop S$ is a lattice.

\item $\Clop S$ is a complete sublattice of~$\Reg S$.

\item $\Clop S=\Reg S$.

\item The join-closure of any open subset of~$S$ is open.

\end{enumerate}
\end{corollary}

\begin{proof}
It is trivial that (ii) implies~(i). Suppose, conversely, that~(i) holds and let $\setm{\ba_i}{i\in I}$ be a collection of clopen subsets of~$S$, with join~$\ba$ in~$\Reg S$. We must prove that~$\ba$ is clopen. Since~$\Clop S$ is a lattice, for each finite $J\subseteq I$, the set $\setm{\ba_i}{i\in J}$ has a join, that we shall denote by~$\ba_{(J)}$, in~$\Clop S$. Since $\Clop S$ is tight in~$\Reg S$, $\ba_{(J)}$ is also the join of $\setm{\ba_i}{i\in J}$ in~$\Reg S$. Since~$\ba$ is the directed join of the clopen sets~$\ba_{(J)}$, for $J\subseteq I$ finite, it is clopen as well, thus completing the proof that $\Clop S$ is a complete sublattice of~$\Reg S$.

It is obvious that~(iii) and~(iv) are equivalent, and that they imply~(i). Hence it remains to prove that~(ii) implies~(iv). By Corollary~\ref{C:MinNbhdSemil2}, every regular closed subset of~$S$ is a join of clopen subsets, hence~$\Clop S$ generates~$\Reg S$ as a complete sublattice. Since~$\Clop S$ is a complete sublattice of~$\Reg S$, (iii) follows.
\end{proof}

\begin{example}\label{Ex:RS2S3}
Denote by $\sS_m$ the \js\ of all nonempty subsets of $[m]$, for any positive integer~$m$. It is an easy exercise to verify that~$\Clop\sS_2=\Reg\sS_2$  is isomorphic to the permutohedron on three letters~$\sP(3)$, which is the six-element ``benzene lattice''.

On the other hand, the lattice~$\Reg\sS_3=\Clop\sS_3$ has apparently not been met until now.

Denote by~$a$, $b$, $c$ the generators of the \js~$\sS_3$ (see the left hand side of Figure~\ref{Fig:RS3}).  The lattice~$\Clop\sS_3$ is represented on the right hand side of Figure~\ref{Fig:RS3}, by using the following labeling convention:
 \[
 \set{a}\mapsto a\,,\quad
 \set{a,a\vee b}\mapsto a^2b\,,\quad
 \set{a,b,a\vee b}\mapsto a^2b^2\,, 
 \]
 (the ``variables'' $a$, $b$, $c$ being thought of as pairwise commuting, so for example $a^2b=ba^2$) and similarly for the pairs $\set{b,c}$ and $\set{a,c}$, and further, $\ol{\es}=\sS_3$, $\ol{a}^2\ol{b}=\sS_3\setminus(a^2b)$, and so on.

\begin{figure}[htb]
\centering
\includegraphics{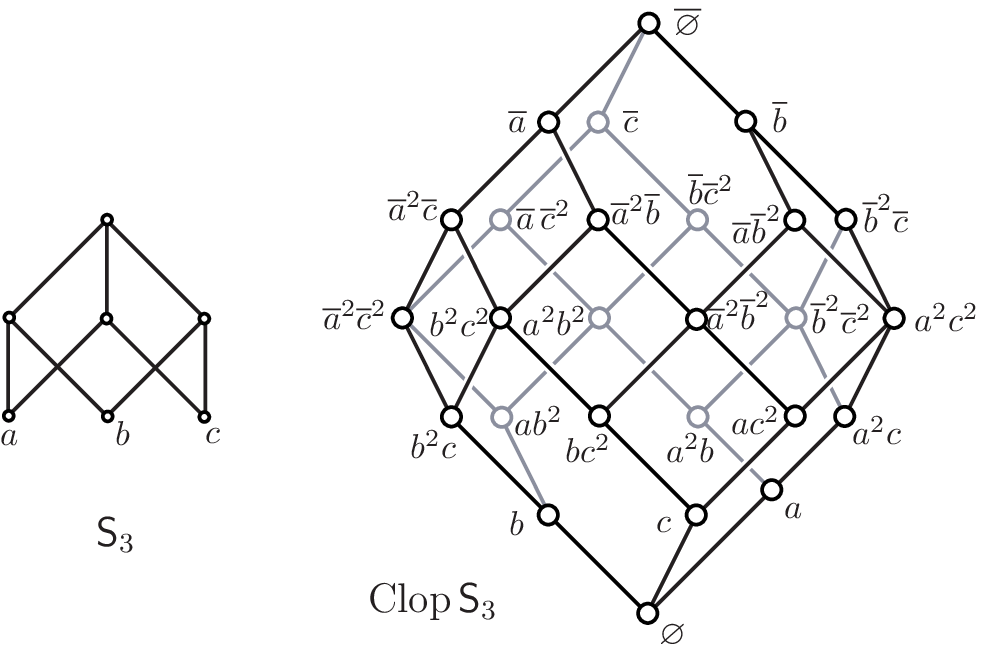}
\caption{The permutohedron on the \js\ $\sS_3$}
\label{Fig:RS3}
\end{figure}

Going to higher dimensions, it turns out that $\Clop\sS_4$ is not a
lattice. In order to see this, observe that (denoting by~$a$, $b$,
$c$, $d$ the generators of~$\sS_4$) the subsets $\bx=\set{a,a\vee b}$
and $\by=\set{c,c\vee d}$ are clopen, and their join $\bx\vee\by$
in~$\Reg\sS_4$ is the regular closed set $\bz=\set{a,c,a\vee b,c\vee
  d,a\vee c,a\vee b\vee c,a\vee c\vee d,a\vee b\vee c\vee d}$, which
is not clopen (for $a\vee b\vee c\vee d=(a\vee d)\vee(b\vee c)$ with
neither $a\vee d$ nor $b\vee c$ in~$\bz$). Brute force computation
shows that $\card\Reg\sS_4=162$ while $\card\Clop\sS_4=150$. Every
\jirr\ element of~$\Reg\sS_4$ belongs to~$\Clop\sS_4$
(cf. Corollary~\ref{C:MinNbhdSemil2}).
\end{example}

\begin{example}\label{Ex:PsubSRS}
Unlike the situation with graphs (cf. Theorem~\ref{T:PGLatt}), the property of $\Clop S$ being a lattice is not preserved by subsemilattices, so it cannot be expressed by the exclusion of a list of ``forbidden subsemilattices''. For example, consider the subsemilattice~$P$ of~$\sS_3$ represented on the left hand side of Figure~\ref{Fig:RnotPSem}.

\begin{figure}[htb]
\includegraphics{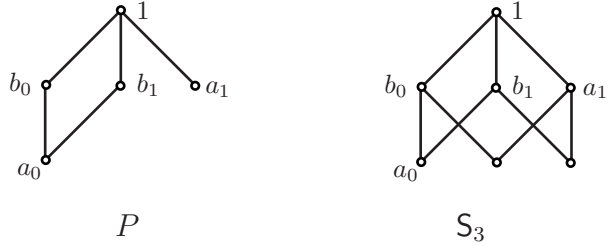}
\caption{A subsemilattice of~$\sS_3$}\label{Fig:RnotPSem}
\end{figure}

Then the sets~$\set{a_i}$ and~$\set{a_0,a_1,1,b_j}$ are clopen in~$P$, with $\set{a_i}\subset\set{a_0,a_1,1,b_j}$, for all $i,j<2$. However, there is no $\bc\in\Clop P$ such that $\set{a_i}\subseteq\bc\subseteq\set{a_0,a_1,1,b_j}$ for all $i,j<2$. Hence $\Clop P$ is not a lattice. On the other hand, $P$ is a subsemilattice of~$\sS_3$ and $\Clop\sS_3$ is a lattice (cf. Example~\ref{Ex:RS2S3}).
\end{example}

\section{A collection of quasi-identities for closure spaces of poset type}\label{S:SDReg}

A natural strengthening of pseudocomplementedness, holding in particular for all permutohedra, and even in all finite Coxeter lattices (see Le Conte de Poly-Barbut~\cite{Poly94}), is \msdy. Although we shall verify shortly (Example~\ref{Ex:ClopM3}) that semidistributivity may not hold in~$\Reg(P,\gf)$ even for $(P,\gf)$ of semilattice type, we shall prove later that once it holds, then it implies, in the finite case, a much stronger property, namely being bounded (cf. Theorem~\ref{T:SDReg2bounded}). Recall that the implication ``semidistributive $\Rightarrow$ bounded'' does not hold for arbitrary finite ortholattices, see Example~\ref{Ex:OrthSDnotBded}, also Reading \cite[Figure~5]{Read03}.

\begin{examplepf}\label{Ex:ClopM3}
A finite closure space $(P,\gf)$ of semilattice type such that the lattice~$\Reg(P,\gf)$ is not semidistributive.
\end{examplepf}

\begin{proof}
Denoting by~$a$, $b$, $c$ the atoms of the five-element modular nondistributive lattice~$\sM_3$ (see the left hand side of Figure~\ref{Fig:PM3}), we endow $P=\sM_3^-=\set{a,b,c,1}$ with the restriction of the ordering of~$\sM_3$. For any subset~$\bx$ of~$P$, we set $\gf(\bx)=\bx$, unless $\bx=\set{a,b,c}$, in which case we set $\gf(\bx)=P$. The only nontrivial covering of $(P,\gf)$ is $1\in\gf(\set{a,b,c})$, and indeed $1=a\vee b\vee c$ in~$P$, hence $(P,\gf)$ has semilattice type.

The lattice $\Reg(P,\gf)=\Clop(P,\gf)$ is represented on the right hand side of Figure~\ref{Fig:PM3}. Its elements are labeled as $\set{a} \mapsto a$, $\set{1,a,b} \mapsto 1ab$, and so on.

\begin{figure}[htb]
\includegraphics{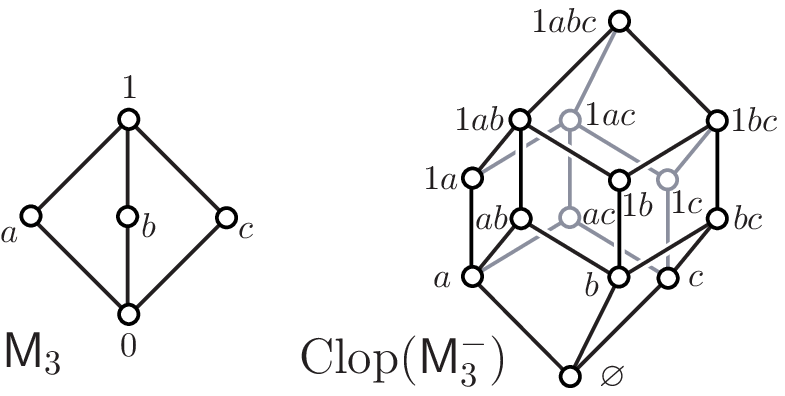}
\caption{The lattice $\Clop(\sM_3^-,\gf)$}\label{Fig:PM3}
\end{figure}

It is not semidistributive, as $ab\wedge 1b=bc\wedge 1b=b$ while $(ab\vee bc)\wedge 1b=1b>b$.
\end{proof}

We shall now introduce certain weakenings of semidistributivity, which are always satisfied by lattices of regular closed subsets of well-founded closure spaces of poset type (a poset~$P$ is \emph{well-founded} if every nonempty subset of~$P$ has a minimal element). We begin with an easy lemma.

\begin{lemma}\label{L:downpReg}
Let $(P,\gf)$ be a closure space. The following statements hold, for any $p\in P$ and any $\ba\subseteq P$.
\begin{enumerate}
\item If $(P,\gf)$ has semilattice type, then $\gf(\ba\dnw p)=\gf(\ba)\dnw p$.

\item If $(P,\gf)$ has poset type, then $\cgf(\ba\dnw p)=\cgf(\ba)\dnw p$.

\item If $(P,\gf)$ has semilattice type and if $\ba$ is closed \pup{resp., open, regular closed, regular open, clopen, respectively}, then so is $\ba\dnw p$.
\end{enumerate}
\end{lemma}

\begin{proof}
(i). Let $q\in\gf(\ba\dnw p)$. By Lemma~\ref{L:Alg2MinCov}, there exists $\bx\in\cM(q)$ such that $\bx\subseteq\ba\dnw p$, so $q=\bigvee\bx\leq p$, and so $q\in\gf(\ba)\dnw p$. Conversely, let $q\in\gf(\ba)\dnw p$. By Lemma~\ref{L:Alg2MinCov}, there exists $\bx\in\cM(q)$ such that $\bx\subseteq\ba$; as moreover $\bigvee\bx = q\leq p$, we get $\bx\subseteq\ba\dnw p$, whence $q\in\gf(\ba\dnw p)$.
  
  (ii). The containment $\cgf(\ba\dnw p)\subseteq\cgf(\ba)\dnw p$ is trivial: $\cgf(\ba\dnw p)\subseteq\cgf(\ba)$ since $\cgf$ is isotone, while $\cgf(\ba\dnw p)\subseteq\ba\dnw p\subseteq P\dnw p$. Conversely, let $q\in\cgf(\ba)\dnw p$ and let $\bx\in\cM(q)$. Since $(P,\gf)$ has poset type, $\bx\subseteq P\dnw q\subseteq P\dnw p$. {}From $q\in\cgf(\ba)$ it follows (cf. Lemma~\ref{L:OpenSemType}) that $\bx\cap\ba\neq\es$; hence $\bx\cap(\ba\dnw p)\neq\es$. By using again Lemma~\ref{L:OpenSemType}, we obtain that $q\in\cgf(\ba\dnw p)$.

(iii) follows trivially from the combination of~(i) and~(ii).
\end{proof}

\begin{lemma}\label{L:AlmostSDReg}
Let~$\ba$ and~$\bc$ be subsets in a closure space $(P,\gf)$ of poset type. Then every minimal element~$x$ of $\gf(\ba\cup\bc)\setminus\gf(\bc)$ belongs to~$\ba$.
\end{lemma}

\begin{proof}
  Pick $\bu\in\cM(x)$ such that $\bu\subseteq\ba\cup\bc$ and suppose
  that~$\bu$ is nontrivial. Since~$(P,\gf)$ has poset type, this implies
  that every $u\in\bu$ is smaller than~$x$, thus, if~$u\in\ba$ (so
  \emph{a fortiori} $u\in\gf(\ba\cup\bc)$), it follows from the
  minimality assumption on~$x$ that $u\in\gf(\bc)$. Hence
  $\bu\cap\ba\subseteq\gf(\bc)$, but
  $\bu\subseteq\ba\cup\bc\subseteq\ba\cup\gf(\bc)$, thus
  $\bu\subseteq\gf(\bc)$, and thus, as $\bu\in\cM(x)$ and~$\gf(\bc)$
  is closed, we get $x\in\gf(\bc)$, \contr. Therefore, $\bu$ is the trivial covering~$\set{x}$, so that we get $x\in\ba$ from $x\notin\bc$.
\end{proof}

Although semidistributivity may fail in finite lattices of regular closed sets (cf. Example~\ref{Ex:ClopM3}), we shall now prove that a certain weak form of semidistributivity always holds in those lattices, whenever $(P,\gf)$ has poset type. This will be sufficient to yield, in Corollary~\ref{C:AlmostSDReg2}, a characterization of semidistributivity by the exclusion of a single lattice.

\begin{theorem}\label{T:AlmostSDReg2}
Let $(P,\gf)$ be a well-founded closure space of poset type, let $o\in I$, let $\vecm{\ba_i}{i\in I}$ be a nonempty family of regular closed subsets of~$P$, and let $\bc,\bd\subseteq P$ be regular closed subsets such that \pup{the joins and meets being evaluated in~$\Reg(P,\gf)$}
\begin{enumerate}
\item $\ba_i\vee\bc=\bd$ for each $i\in I$;

\item $\ba_o\wedge\bc=
\bigwedge_{i\in I}\ba_i$.
\end{enumerate}

Then $\ba_i\subseteq\bc$ for each $i\in I$.
\end{theorem}

\begin{proof}
Suppose, by way of contradiction, that $\bigcup_{i \in I} \ba_{i} \not\subseteq \bc$. Then the set $\bd\setminus\bc$ is nonempty, thus, since~$P$ is well-founded, $\bd\setminus\bc$ has a minimal element~$e$. Observing that $\gf\pI{\cgf(\ba_i)\cup\bc}=\bd$ for each $i\in I$, it follows from Lemma~\ref{L:AlmostSDReg} that
 \begin{equation}\label{Eq:einallai}
 e\in\cgf(\ba_i)\text{ for each }i\in I\,.
 \end{equation}
On the other hand, from $e\notin\bc$ together with Assumption~(ii) it
follows that $e\notin\bigwedge_{i\in I}\ba_i$, so, \emph{a fortiori},
$e\notin\cgf\pI{\bigcap_{i\in I}\ba_i}$; consequently,
there exists $\bu\in\cM(e)$ such that
 \begin{equation}\label{Eq:bucapcaai=es}
 \bu\cap\bigcap_{i\in I}\ba_i=\es\,.
 \end{equation}
{}From~\eqref{Eq:einallai} and~\eqref{Eq:bucapcaai=es} it follows that~$e\notin\bu$. Since $(P,\gf)$ has poset type, it follows that $u<e$ for each $u\in\bu$.

Let $u\in\bu$. {}From~\eqref{Eq:bucapcaai=es} together with Assumption~(ii) it follows that $u\notin\ba_o\wedge\bc$, hence $u\notin\cgf(\ba_o\cap\bc)$, and hence there exists $\bz_u\in\cM(u)$ such that
 \begin{equation}\label{Eq:zusmall}
 \bz_u\cap\ba_o\cap\bc=\es\,.
 \end{equation}
 Since $(P,\gf)$ has poset type, $\bz_u\subseteq P\dnw u$ for each
 $u\in\bu$.  Set $\bz=\bigcup_{u\in\bu}\bz_u$. Since $u\in\gf(\bz_u)$ for
 each $u\in\bu$ and as $e\in\gf(\bu)$, we get $e\in\gf(\bz)$. Since $e\in\cgf(\ba_o)$ (cf.~\eqref{Eq:einallai}), it follows that $\bz\cap\ba_o\neq\es$. Pick $z\in\bz\cap\ba_o$. There exists $u\in\bu$ such that $z\in\bz_u$. {}From $z\leq u$ and $u<e$ it follows that $z<e$. Since $z\in\ba_o\subseteq\bd$ and by the minimality statement on~$e$, we get $z\in\bc$, so $z\in\bz_u\cap\ba_o\cap\bc$, which contradicts~\eqref{Eq:zusmall}.
\end{proof}

In particular, whenever~$(P,\gf)$ is a closure system of poset type with~$P$ well-founded and~$m$ is a positive integer, the lattice~$\Reg(P,\gf)$ satisfies the following quasi-identity, weaker than \jsdy:
 \begin{equation}
 \pII{a_0\vee c=a_1\vee c=\cdots=
 a_m\vee c\text{ and }
 a_0\wedge c=
 \bigwedge_{0\leq i\leq m}a_i}
 \ \Longrightarrow\ a_0\leq c\,.
 \tag*{\RSD{m}}
 \end{equation}

\begin{proposition}\label{P:RSDm}
The following statements hold, for every positive integer~$m$.
\begin{enumerate}
\item Meet-semidistributivity and \jsdy\ both imply \RSD{m}.

\item \RSD{m+1} implies \RSD{m}.
\end{enumerate}
\end{proposition}

\begin{proof}
(i). Let $a_0$, \dots, $a_m$, $c$ be elements in a lattice~$L$, satisfying the premise of \RSD{m}. If~$L$ is \msd, then, as $a_0\wedge c=\bigwedge_{0\leq i\leq m}a_i$ and by J\'onsson and Kiefer \cite[Theorem~2.1]{JoKi62} (see also Freese, Je\v{z}ek, and Nation \cite[Theorem~1.21]{FJN}),
 \[
 a_0\wedge c=
 \bigwedge_{0\leq i\leq m}(a_0\vee a_i)\wedge
 \bigwedge_{0\leq i\leq m}(c \vee a_i)
=a_0\,,
 \]
so $a_0\leq c$, as desired. If~$L$ is \jsd, then
 \[
 a_0\vee c=\pII{\bigwedge_{0\leq i\leq m} a_i}\vee c
 =(a_0\wedge c)\vee c=c\,,
 \]
so $a_0\leq c$ again.

(ii) is trivial.
\end{proof}

A computer search, using the software \texttt{Mace4} (see McCune~\cite{McCune}), yields that every 24-element (or less) lattice satisfying \RSD{1} also satisfies \RSD{2}, nevertheless that there exists a 25-element lattice~$K$ satisfying both \RSD{1} and its dual, but not \RSD{2}. It follows that the 52-element ortholattice $K\parallel K^\op$ (cf. Section~\ref{S:Basic}) satisfies \RSD{1} but not \RSD{2}. We do not know whether \RSD{m+1} is properly stronger than the conjunction of \RSD{m} and its dual, for each positive integer~$m$, although this seems highly plausible.

The quasi-identity~\RSD{1} does not hold in any of the lattices~$\sM_3$, $\sL_3$, and~$\sL_4$ represented in Figure~\ref{Fig:NonSD}. (We are following the notation of Jipsen and Rose~\cite{JiRo} for those lattices.) Since~$\Reg(P,\gf)$ is self-dual, none of those lattices neither their duals can be embedded into~$\Reg(P,\gf)$, whenever~$P$ is well-founded of poset type.

\begin{figure}[htb]
\includegraphics{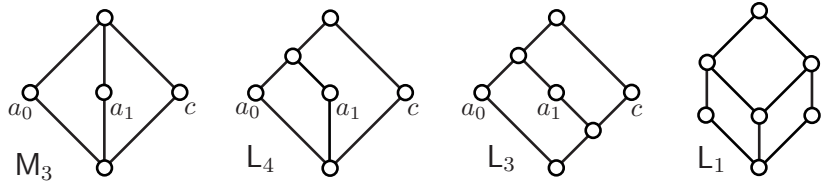}
\caption{Non-semidistributive lattices, with
failures of~\RSD{1} marked whenever possible}
\label{Fig:NonSD}
\end{figure}

As, in the finite case, semidistributivity is characterized by the exclusion, as sublattices, of~$\sM_3$, $\sL_3$, $\sL_4$, together with the lattice~$\sL_1$ of Figure~\ref{Fig:NonSD}, and the dual lattices of~$\sL_1$ and~$\sL_4$ (cf. Davey, Poguntke, and Rival~\cite{DPR} or Freese, Je\v{z}ek, and Nation \cite[Theorem~5.56]{FJN}), it follows from the self-duality of~$\Reg(P,\gf)$ together with Theorem~\ref{T:AlmostSDReg2} that the semidistributivity of~$\Reg(P,\gf)$ takes the following very simple form.

\begin{corollary}\label{C:AlmostSDReg2}
Let $(P,\gf)$ be a finite closure system of poset type. Then $\Reg(P,\gf)$ is semidistributive if{f} it contains no copy of~$\sL_1$.
\end{corollary}

The following example shows that the assumption, in Theorem~\ref{T:AlmostSDReg2}, of $(P,\gf)$ being of poset type cannot be relaxed to $(P,\gf)$ being a convex geometry.

\begin{examplepf}\label{Ex:AlmostSDReg2}
A finite atomistic convex geometry $(P,\gf)$ such that $\Reg(P,\gf)$ does not satisfy~\RSD{1}.
\end{examplepf}

\begin{proof}
Consider a six-element set $P=\set{a,b,c,d,e,u}$, and let a subset~$\bx$ of~$P$ be \emph{closed} if
 \begin{align*}
 \set{c,d,u}\subseteq\bx&\Rightarrow
 \set{a,b,e}\subseteq\bx\,,\\
 \set{a,b,u}\subseteq\bx&\Rightarrow e\in\bx\,,\\
 \set{c,d,e}\subseteq\bx&\Rightarrow
 \set{a,b}\subseteq\bx\,. 
 \end{align*}
Denote by~$\gf(\bx)$ the least closed set containing~$\bx$, for each $\bx\subseteq P$. It is obvious that $(P,\gf)$ is an atomistic closure space. Brute force calculation also shows that $(P,\gf)$ is a convex geometry. There are~$51$ closed sets and~$40$ regular closed sets, the latter all clopen. The following subsets
 \begin{align*}
 \ba_0&=\set{a,d,e}\,,\\
 \ba_1&=\set{b,d,e}\,,\\
 \bc&=\set{c,d} 
 \end{align*}
 are all clopen. Moreover, $\ba_0\cap\ba_1=\set{e,d}$ and $e\notin\cgf(\ba_0\cap\ba_1)$ (because $\set{a,b,u}$ is a minimal covering of~$e$ disjoint from $\ba_0\cap\ba_1$), so $\ba_0\wedge\ba_1=\set{d}$. Furthermore, $\ba_0\vee\bc=\ba_1\vee\bc=\set{a,b,c,d,e}$ and $(\ba_0\vee\ba_1)\cap\bc=\set{a,b,d,e}\cap\bc = \set{d}$. Therefore,
 \emph{$\set{\ba_0,\ba_1,\bc}$ generates a sublattice of $\Reg(P,\gf)$ isomorphic to~$\sL_4$}, with labeling as given by Figure~\ref{Fig:NonSD}. In particular, \emph{$\Reg(P,\gf)$ does not satisfy the quasi-identity~\RSD{1}}.
\end{proof}

Nevertheless, the elements~$\set{c,u}$, $\set{d,u}$, and $\set{e,u}$ are the atoms of a copy of~$\sL_1$ in~$\Reg(P,\gf)$. Hence the construction of Example~\ref{Ex:AlmostSDReg2} is not sufficient to settle whether Corollary~\ref{C:AlmostSDReg2} can be extended to the case of convex geometries.
Observe, also, that although there are lattice embeddings from both~$\sL_1$ and~$\sL_4$ into $\Reg(P,\gf)$, there is no $0$-lattice embedding from either~$\sL_1$ or~$\sL_4$ into~$\Reg(P,\gf)$: this follows from Proposition~\ref{P:PseudoCpl} (indeed, neither~$\sL_1$ nor~$\sL_4$ is pseudocomplemented).

Figure \ref{fig:example5.10} illustrates the closure lattice of the
closure space $(P,\gf)$ of Example~\ref{Ex:AlmostSDReg2}, with the
copy of $\sL_1$ in gray and the copy of $\sL_4$ in black.

\begin{figure}[htbp]
  \centering
  \includegraphics[scale=0.4]{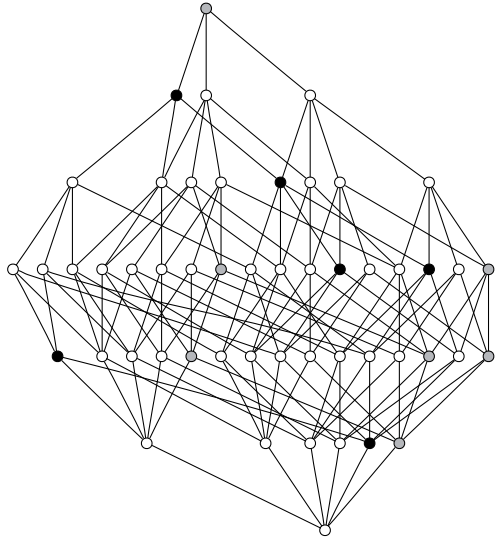}
  \label{fig:example5.10}
\caption{The closure lattice of Example \ref{Ex:AlmostSDReg2}}
\end{figure}

\section{{}From semidistributivity to boundedness for semilattice type}\label{S:BdedReg}

In Lemmas~\ref{L:DescrJirrRegP}--\ref{L:PerspInt2SD} let $(P,\gf)$ be
a closure space of semilattice type.

We begin with a useful structural property of the completely \jirr\ 
elements of $\Reg(P,\gf)$. We refer to Section~\ref{S:RegCl} for the notation~$\partial\ba$.

\begin{lemma}\label{L:DescrJirrRegP}
  Every completely \jirr\ element~$\ba$ of the lattice~$\Reg(P,\gf)$
  has a largest element~$p$, and
  $\ba\setminus\ba_*=\cgf(\ba)\setminus\cgf(\ba_*)=\set{p}$.
  Furthermore, for every $x\in\partial\ba$, there exists $\bx\in\cM(p)$
  such that $\bx\cap\ba=\set{x}$.
\end{lemma}

\begin{proof}
{}From the trivial observation that~$\ba$ is the union of all~$\ba\dnw x$, for $x\in\ba$, and by Lemma~\ref{L:downpReg}, it follows that $\ba=\bigvee\vecm{\ba\dnw x}{x\in\ba}$ in $\Reg(P,\gf)$. Since~$\ba$ is completely \jirr, there exists $p\in\ba$ such that $\ba=\ba\dnw p$. Of course, $p$ is necessarily the largest element of~$\ba$.

\setcounter{claim}{0}

\begin{claim}\label{Cl:xsuba2x=a}
Let $\bx\in\Reg(P,\gf)$ be contained in~$\ba$. Then $p\in\bx$ implies that $\bx=\ba$.
\end{claim}

\begin{cproof}
{}From $p\in\bx$ and $\bx\subseteq\ba$ it follows that $\ba=\bx\cup(\ba\setminus\set{p})$, hence
 \[
 \ba=\bx\cup\bigcup\vecm{\ba\dnw x}
 {x\in\ba\setminus\set{p}}\,,
 \]
and hence, by Lemma~\ref{L:downpReg},
 \[
 \ba=\bx\vee\bigvee\vecm{\ba\dnw x}
 {x\in\ba\setminus\set{p}}
 \text{ in }\Reg(P,\gf)\,.
 \]
Since~$\ba$ is completely \jirr\ and $p\in\ba\setminus(\ba\dnw x)$ for each $x\in\ba\setminus\set{p}$, the desired conclusion follows.
\end{cproof}

\begin{claim}\label{Cl:a-pinRegP}
The set $\ba\setminus\set{p}$ is regular closed.
\end{claim}

\begin{cproof}
Evaluate the join $\bx=\bigvee\vecm{\ba\dnw x}{x\in\ba\setminus\set{p}}$ in~$\Reg(P,\gf)$. Since each of the joinands~$\ba\dnw x$ is smaller than~$\ba$ and the latter is completely \jirr, we obtain that $\bx\subsetneqq\ba$, thus $p\notin\bx$ (cf. Claim~\ref{Cl:xsuba2x=a}), so $\bx\subseteq\ba\setminus\set{p}$. The converse containment being trivial, $\bx=\ba\setminus\set{p}$.
\end{cproof}

{}From Claims~\ref{Cl:xsuba2x=a} and~\ref{Cl:a-pinRegP} it follows that $\ba_*=\ba\setminus\set{p}$.

Let $q\in\cgf(\ba)\setminus\cgf(\ba_*)$. Since $q\notin\cgf(\ba_*)$, there exists $\bx\in\cM(q)$ such that $\bx\cap\ba_*=\es$. {}From $q\in\cgf(\ba)$ it follows that $\bx\cap\ba\neq\es$, thus $p\in\bx$ (as $\ba=\ba_*\cup\set{p}$) and $p\leq q$ (as $q=\bigvee\bx$); whence, as also $p=\max\ba$, we get $p=q$. Therefore, $\cgf(\ba)\setminus\cgf(\ba_*)\subseteq\set{p}$. Since $\cgf(\ba)\setminus\cgf(\ba_*)\neq\es$ (because $\ba_*$ is properly contained in~$\ba$ and both sets are regular closed), it follows that $\cgf(\ba)\setminus\cgf(\ba_*)=\set{p}$.

Finally, let $x\in\partial\ba$. Since $\ba\setminus\set{x}$ is closed
and does not contain~$x$ as an element, $\cgf(\ba\setminus\set{x})$ is
regular open and properly contained in~$\cgf(\ba)$, thus it is
contained in
$\cgf(\ba)_*=\cgf(\ba_*)=\cgf(\ba)\setminus\set{p}$. Hence
$p\notin\cgf(\ba\setminus\set{x})$, which means that there exists
$\bx\in\cM(p)$ such that $\bx\cap(\ba\setminus\set{x})=\es$. {}From
$p\in\cgf(\ba)$ it follows that $\bx\cap\ba\neq\es$, so
$\bx\cap\ba=\set{x}$.
\end{proof}

As we shall see from Example \ref{Ex:FinPosClop}, not every \jirr\ element of $\Reg(P,\gf)$ needs to be clopen.

\begin{lemma}\label{L:PerspInt2SD}
The following statements hold, for any completely \jirr\ elements~$\ba$ and~$\bb$ of $\Reg(P,\gf)$.
\begin{enumerate}
\item If $\orth{\ba}\searrow\bb$, then $\max\ba\geq\max\bb$; if, moreover, $\max\ba=\max\bb$, then $[\orth{\ba},\orth{(\ba_*)}]$ is down-perspective to $[\bb_*,\bb]$.
  
\item If $\ba\nearrow\orth{\bb}$, then $\max\ba\geq\max\bb$; if, moreover, $\max\ba=\max\bb$, then $[\ba_*,\ba]$ is up-perspective to $[\orth{\bb},\orth{(\bb_*)}]$.

\end{enumerate}
\end{lemma}

Let us observe that the first parts of items~(i) and~(ii) in the Lemma
immediately imply the following property (as well as its dual): if
$[\orth{\ba},\orth{(\ba_*)}]$ is down-perspective to $[\bb_*,\bb]$, then
$\max \ba = \max \bb$. The situation can be visualized on
Figure~\ref{Fig:4persp}. (Following the convention used in Freese,
Je\v{z}ek, and Nation~\cite{FJN}, prime intervals are highlighted by
crossing them with a perpendicular dash.) 

\begin{figure}[htb]
\includegraphics{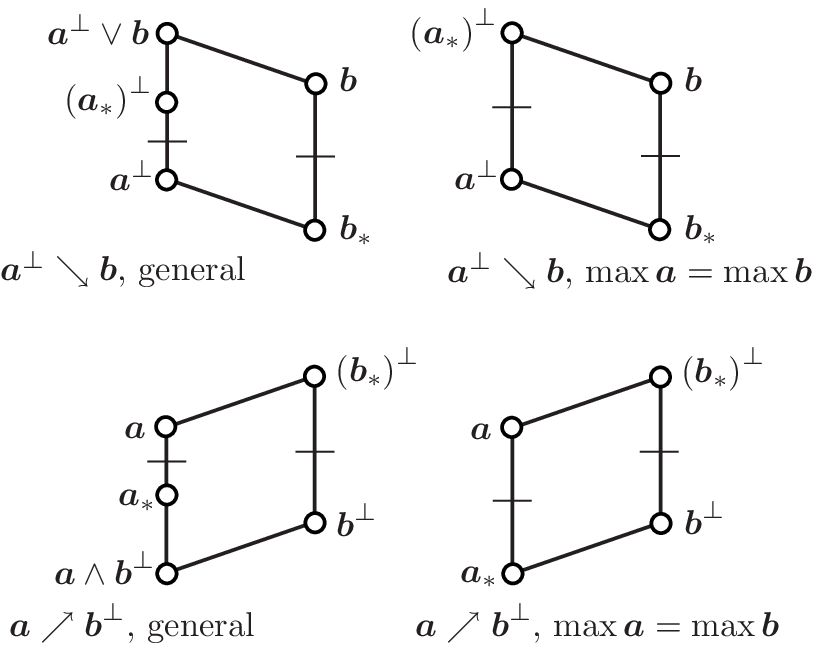}
\caption{Illustrating $\orth{\ba}\searrow\bb$ and $\ba\nearrow\orth{\bb}$}\label{Fig:4persp}
\end{figure}

\begin{proof}
Since~(ii) is dual of~(i) (\emph{via} the dual automorphism $\bx\mapsto\orth{\bx}$), it suffices to prove~(i).

Set $p=\max\ba$ and $q=\max\bb$. Let us state the next observation as a Claim.
 
\begin{sclaim}
The relation $\orth{\ba}\searrow\bb$ holds if{f} $\cgf(\ba)\cap\bb=\set{q}$.
\end{sclaim}

\begin{scproof}
Since $\orth{\ba}=\pI{\cgf(\ba)}^\cpl$, the relation $\orth{\ba}\searrow\bb$ means that $\cgf(\ba)\cap\bb\neq\es$ and $\cgf(\ba)\cap\bb_*=\es$. Recalling that $\bb=\bb_*\cup\set{q}$ (cf. Lemma~\ref{L:DescrJirrRegP}), the Claim follows immediately.
\end{scproof}

Using our Claim, we see that if $\orth{\ba}\searrow\bb$, then, as $p=\max\ba$, we get $q \in \ba$, thus $q \leq p$.

Suppose next that $\orth{\ba}\searrow\bb$ and $p=q$. Since $\cgf(\ba)\setminus\cgf(\ba_*)=\set{p}$ (cf. Lemma~\ref{L:DescrJirrRegP}) and $\cgf(\ba)\cap\bb=\set{p}$ (by our Claim together with $p = q$), it follows that $\cgf(\ba_*)\cap\bb=\es$, that is, $\bb\subseteq\orth{(\ba_*)} = (\orth{\ba})^*$. Since $\bb\nleq\orth{\ba}$ follows from $\orth{\ba}\searrow\bb$, we get $\bb \nearrow \orth{\ba}$, showing that the interval $[\orth{\ba},\orth{(\ba_*)}]$ is down-perspective to $[\bb_*,\bb]$. 
\end{proof}

\begin{theorem}\label{T:SDReg2bounded}
Let $(P,\gf)$ be a finite closure space of semilattice type. Then $\Reg(P,\gf)$ is semidistributive if{f} it is a bounded homomorphic image of a free lattice.
\end{theorem}

\begin{proof}
It is well-known that every bounded homomorphic image of a free lattice is semidistributive, see Freese, Je\v{z}ek, and Nation \cite[Theorem~2.20]{FJN}.

Conversely, suppose that~$P$ is finite and that $\Reg(P,\gf)$ is semidistributive. Since~$\Reg(P,\gf)$ is self-dual (\emph{via} the natural orthocomplementation), it suffices to prove that it is a lower bounded homomorphic image of a free lattice. This amounts, in turn, to proving that $\Reg(P,\gf)$ has no cycle of \jirr\ elements with respect to the join-dependency relation~$D$ (cf. Freese, Je\v{z}ek, and Nation \cite[Corollary~2.39]{FJN}). In order to prove this, it is sufficient to prove that $\ba\Dr\bb$ implies that $\max\ba>\max\bb$, for all \jirr\ elements~$\ba$ and~$\bb$ of~$\Reg(P,\gf)$. By Lemma~\ref{L:Arr2D}, there exists a \mirr\ $\bu\in\Reg(P,\gf)$ such that $\ba\nearrow\bu\searrow\bb$. The element $\bc=\orth{\bu}$ is \jirr, and, by Lemma~\ref{L:PerspInt2SD}, $\max\ba\geq\max\bc\geq\max\bb$. Suppose that $\max\ba=\max\bc=\max\bb$. By Lemma~\ref{L:PerspInt2SD}, $\orth{(\bc_*)}=\ba\vee\orth{\bc}=\bb\vee\orth{\bc}$, thus, as~$\Reg(P,\gf)$ is \jsd, $\orth{(\bc_*)}=(\ba\wedge\bb)\vee\orth{\bc}$. On the other hand, $\ba\Dr\bb$ thus $\ba\neq\bb$, and thus $\ba\wedge\bb$ lies either below~$\ba_*$ or below~$\bb_*$, hence below $\orth{\bc}$. It follows that $\orth{(\bc_*)}=\orth{\bc}$, \contr. Therefore, $\max\ba>\max\bb$.
\end{proof}

\begin{example}\label{Ex:OrthSDnotBded}
The lattice~$K$ represented on the left hand side of Figure~\ref{Fig:KplusK}, taken from J\'onsson and Nation~\cite{JoNa77} (see also Freese, Je\v{z}ek, and Nation \cite[page~111]{FJN}), is semidistributive but not bounded. It follows that the parallel sum $L=K\parallel K$ (cf. Section~\ref{S:Basic}), represented on the right hand side of Figure~\ref{Fig:KplusK}, is also semidistributive and not bounded. Now observe that~$K$ has an involutive dual automorphism~$\ga$. Sending each~$x$ in one copy of~$K$ to~$\ga(x)$ in the other copy of~$K$ (and exchanging~$0$ and~$1$) defines an orthocomplementation of~$L$.

\begin{figure}[htb]
\includegraphics{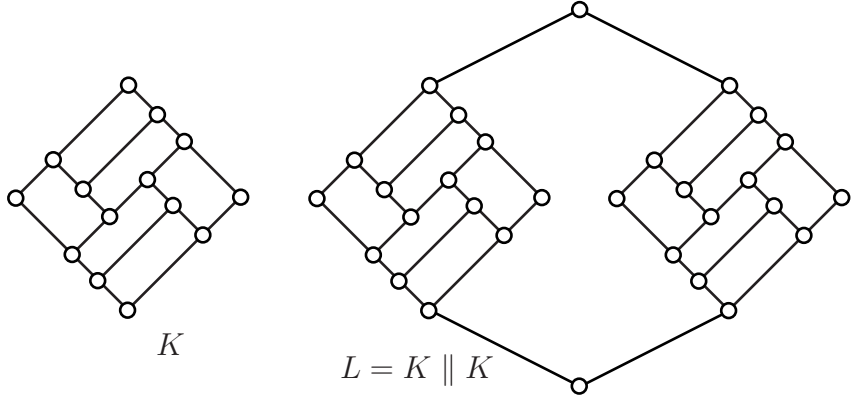}
\caption{The lattices $K$ and $L=K\parallel K$}\label{Fig:KplusK}
\end{figure}

This shows that \emph{a finite, semidistributive ortholattice need not be bounded}.
\end{example}

Although $\Reg(P,\gf)$ may not be semidistributive, even for a finite closure space $(P,\gf)$ of semilattice type (cf. Example~\ref{Ex:ClopM3}), there are important cases where semidistributivity holds, such as the case of the closure space associated to an antisymmetric, transitive binary relation (cf. Santocanale and Wehrung~\cite{SaWe12a}). Further such situations will be investigated in Section~\ref{S:SemilBded} and~\ref{S:PermGraph}.

\section{Boundedness of lattices of regular closed subsets of semilattices}\label{S:SemilBded}

Although $\Reg(P,\gf)$ may not be semidistributive, even for a finite
closure space $(P,\gf)$ of semilattice type, the situation changes for
the closure space associated to a finite semilattice. We first state
an easy lemma.

\begin{lemma}\label{L:Jirr2MinNbhdSem}
  Let~$S$ be a \js. Then every completely \jirr\ member~$\ba$
  of~$\Reg{S}$ is a minimal neighborhood of some element of~$S$.
\end{lemma}

\begin{proof}
Since $\tin(\ba)=\bigcup_{i\in I}\ba_i$ for minimal neighborhoods~$\ba_i$ (cf. Proposition~\ref{P:MinNbhdClos}) and every minimal neighborhood is clopen (cf. Theorem~\ref{T:MinNbhdSemil}), $\ba=\tcl\tin(\ba)=\bigvee_{i\in I}\ba_i$ is a join of minimal neighborhoods in~$\Reg{S}$. Since~$\ba$ is completely \jirr, it follows that~$\ba$ is itself a minimal neighborhood.
\end{proof}

On the other hand, there are easy examples of finite \js{s} containing join-reducible minimal neighborhoods. The structure of completely \jirr\ elements of $\Reg{S}$ will be further investigated in Section~\ref{S:CJISemil}.

A further illustration of Theorem~\ref{T:SDReg2bounded} is provided by the following result.

\begin{theorem}\label{T:FinSemilBded}
For any \emph{finite} \js~$S$, the lattice $\Reg S$ is a bounded homomorphic image of a free lattice.
\end{theorem}

\begin{proof}
  As in the proof of Theorem~\ref{T:SDReg2bounded}, it is sufficient
  to prove that $\ba\Dr\bb$ implies that $\max\ba>\max\bb$, for all
  \jirr\ elements~$\ba$ and~$\bb$ of~$\Reg S$. By Lemma~\ref{L:Arr2D},
  there exists a \mirr\ $\bu\in\Reg S$ such that
  $\ba\nearrow\bu\searrow\bb$.  The element $\bc=\orth{\bu}(=\bu^{\cpl})$ is \jirr\ and it follows from
  Lemma~\ref{L:PerspInt2SD} that
  $\max\ba\geq\max\bc\geq\max\bb$. Suppose that
  $\max\ba=\max\bc=\max\bb$ and denote that element by~$p$. It follows
  from Lemma~\ref{L:Jirr2MinNbhdSem} that~$\ba$, $\bb$, and~$\bc$ are
  minimal neighborhoods of~$p$. Thus, by Theorem~\ref{T:MinNbhdSemil},
  the sets~$\ba$, $\bb$, $\bc$ are all clopen and their respective
  complements~$\tilde{\ba}$, $\tilde{\bb}$, $\tilde{\bc}$ in~$S\dnw p$
  are maximal proper ideals of~$S\dnw p$. {}From
  $\ba\nearrow\bc^{\cpl}\searrow\bb$ and Lemma~\ref{L:DescrJirrRegP}
  it follows that $\ba\cap\bc=\bc\cap\bb=\set{p}$, that is,
  $\tilde{\ba}\cup\tilde{\bc}=\tilde{\bb}\cup\tilde{\bc}=S\ddnw p$. If
  either $\tilde{\ba}\subseteq\tilde{\bc}$ or
  $\tilde{\bb}\subseteq\tilde{\bc}$, then, by the maximality
  statements on both~$\tilde{\ba}$ and~$\tilde{\bb}$, we get $\tilde{\ba}=\tilde{\bc}=\tilde{\bb}$,  thus $\ba=\bb$, in contradiction with $\ba\Dr\bb$. Hence, there are
  $a_0\in\tilde{\ba}\setminus\tilde{\bc}$ and
  $b_0\in\tilde{\bb}\setminus\tilde{\bc}$. For each $a\in\tilde{\ba}$, the element~$a\vee a_0$ belongs to $\tilde{\ba}\setminus\tilde{\bc}=\tilde{\bb}\setminus\tilde{\bc}$, thus $a\in\tilde{\bb}$, so $\tilde{\ba}\subseteq\tilde{\bb}$. Likewise, using~$b_0$, we get $\tilde{\bb}\subseteq\tilde{\ba}$, therefore $\ba=\bb$, a contradiction.
\end{proof}

\begin{example}\label{Ex:NonSDSem}
Theorem~\ref{T:FinSemilBded} implies trivially that~$\Reg S$ is semidistributive, for any finite \js~$S$. We show in the present example that this result cannot be extended to infinite semilattices.

Shiryaev characterized in~\cite{Shir85} the semilattices with semidistributive lattice of subsemilattices. His results were extended to lattices of various kinds of subsemilattices, and various versions of semidistributivity, in Adaricheva~\cite{Adar88}. We show here, \emph{via} a straightforward modification of one of Shiryaev's constructions, that there is a distributive lattice~$\Delta$ whose lattice of all regular closed \jsubsemi{s} is not semidistributive.

\begin{figure}[htb]
\includegraphics{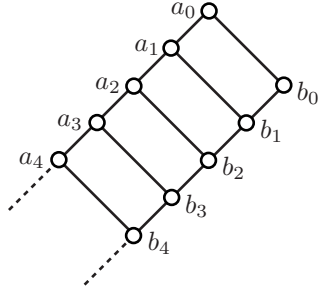}
\caption{The lattice~$\Delta$}
\label{Fig:Delta}
\end{figure}

Endow $\Delta=\omega^{\op}\times\set{0,1}$ with the componentwise ordering, and set $a_n=(n,1)$ and $b_n=(n,0)$, for each $n<\omega$. The lattice~$\Delta$ is represented in Figure~\ref{Fig:Delta}. Of course, $\Delta$ is a distributive lattice; we shall view it as a \js.

\begin{proposition}\label{P:DeltanotSD}
Every regular closed subset of~$\Delta$ is clopen \pup{i.e., $\Clop\Delta=\Reg\Delta$}. Furthermore, $\Reg\Delta$ has nonzero elements $\ba$, $\bb_0$, and~$\bb_1$ such that $\ba\wedge\bb_0=\ba\wedge\bb_1=\es$ and $\ba\subseteq\bb_0\cup\bb_1$. In particular, $\Reg\Delta$ is neither semidistributive nor pseudocomplemented.
\end{proposition}

\begin{proof}
We first observe that the nontrivial irredundant joins in~$\Delta$ are exactly those of the form
 \[
 a_m=a_n\vee b_m\,,
 \quad \text{with } m<n<\omega\,.
 \]
Let~$\bu$ be an open subset of~$\Delta$, we shall prove that~$\tcl(\bu)$ is open as well. By the observation above, it suffices to prove that whenever $m<n<\omega$, $a_m\in\tcl(\bu)$ implies that either $b_m\in\tcl(\bu)$ or $a_n\in\tcl(\bu)$. Since~$\bu$ is open, that conclusion is obvious if $a_m\in\bu$. Now suppose that $a_m\in\tcl(\bu)\setminus\bu$. Since~$\ba_m$ is obtained as a nontrivial irredundant join of elements of~$\bu$, there exists an integer $k>m$ such that $\set{a_k,b_m}\subseteq\bu$. In particular, $b_m\in\bu$ and the desired conclusion holds again. This completes the proof that $\Reg\Delta=\Clop\Delta$.

Now we set
 \begin{align*}
 \ba&=\setm{a_n}{n<\omega}\,,\\
 \bb_0&=\setm{a_{2k}}{k<\omega}\cup
 \setm{b_{2k}}{k<\omega}\,,\\
 \bb_1&=\setm{a_{2k+1}}{k<\omega}\cup
 \setm{b_{2k+1}}{k<\omega}\,.
 \end{align*}
It is straightforward to verify that $\ba$, $\bb_0$, and~$\bb_1$ are all clopen and that $\ba\subseteq\bb_0\cup\bb_1$. Furthermore, for each $i\in\set{0,1}$, the subset
 \[
 \ba\cap\bb_i=
 \setm{a_{2k+i}}{k<\omega}
 \]
has empty interior (for $a_{2k+i}=a_{2k+i+1}\vee b_{2k+i}$), that is, $\ba\wedge\bb_i=\es$.
\end{proof}

\end{example}

\section{Completely \jirr\ regular closed sets in semilattices}\label{S:CJISemil}

We observed in Lemma~\ref{L:Jirr2MinNbhdSem} that for any \js~$S$, every completely \jirr\ element of~$\Reg{S}$ is a minimal neighborhood. By invoking the structure theorem of minimal neighborhoods in~$\Reg{S}$ (viz. Theorem~\ref{T:MinNbhdSemil}), we shall obtain, in this section, a complete description of the completely \jirr\ elements of~$\Reg{S}$.

We start with an easy lemma, which will make it possible to identify the possible top elements of completely \jirr\ elements of~$\Reg{S}$. We denote by~$\Id S$ the lattice of all ideals of~$S$ (the empty set included), under set inclusion.

\begin{lemma}\label{L:nlowcov}
The following statements are equivalent, for any element~$p$ in a \js~$S$ and any positive integer~$n$.
\begin{enumerate}
\item $S\dnw p$ has at most~$n$ lower covers in~$\Id S$.

\item There are ideals~$\ba_1$, \dots, $\ba_n$ of~$S$ such that $S\ddnw p=\bigcup_{1\leq i\leq n}\ba_i$.

\item There are lower covers~$\ba_1$, \dots, $\ba_n$ of~$S\dnw p$ in~$\Id S$ such that $S\ddnw p=\bigcup_{1\leq i\leq n}\ba_i$.

\item There is no $(n+1)$-element subset~$W$ of~$S\ddnw p$ such that $p=u\vee v$ for all distinct $u,v\in W$.
\end{enumerate}
\end{lemma}

\begin{proof}
(ii)$\Rightarrow$(iv). Let~$W$ be an $(n+1)$-element subset of~$S\ddnw p$ such that $p=u\vee v$ for all distinct $u,v\in W$. Every element of~$W$ belongs to some~$\ba_i$, thus there are $i\in[n]$ and distinct $u,v\in W$ such that $u,v\in\ba_i$. Hence $p=u\vee v$ belongs to~$\ba_i$, a contradiction.

(iv)$\Rightarrow$(iii). Let $W\subseteq S\ddnw p$ such that $p=u\vee
v$ for all distinct $u,v\in W$ (we say that~$W$ is
\emph{anti-orthogonal}), of maximal cardinality, necessarily at
most~$n$, with respect to that property. The set $\ba_u=\setm{x\in
  S}{x\vee u<p}$ is a lower subset of $S\ddnw p$, for each $u\in
W$. If~$\ba_u$ is not an ideal, then there are $x,y\in\ba_u$ such that
$p=x\vee y\vee u$, and then $W'=\set{x\vee u,y\vee
  u}\cup(W\setminus\set{u})$ is anti-orthogonal with $\card W'>\card
W$, a contradiction; hence~$\ba_u$ is an ideal of~$S\dnw p$. Let~$\bb$
be an ideal of~$S$ with $\ba_u\subsetneqq\bb\subseteq S\dnw p$. Every
$x\in\bb\setminus\ba_u$ satisfies $p=x\vee u$. Since $x\in\bb$ and
$u\in\ba_u\subseteq\bb$, we get $p\in\bb$, so $\bb=S\dnw p$, thus
completing the proof that $\ba_u\prec S\dnw p$. Finally, it follows
from the maximality assumption on~$W$ that $S\ddnw p=\bigcup_{u\in
  W}\ba_u$.

(iii)$\Rightarrow$(ii) is trivial, so (ii)--(iv) are equivalent. Trivially, (iii) implies~(i). Finally, suppose that~(i) holds. By Zorn's Lemma, every $x\in S\ddnw p$ is contained in some lower cover of~$S\dnw p$; hence~(iii) holds.
\end{proof}

Referring to the canonical join-embedding $S\hookrightarrow\Id S$, $p\mapsto S\dnw p$, we shall often identify~$p$ and~$S\dnw p$ and thus state~(i) above by saying that ``$p$ has at most~$n$ lower covers in the ideal lattice of~$S$''. Since~$\es$ is an ideal, $S\dnw p$ has always a lower cover in~$\Id S$.

\begin{theorem}\label{T:JirrRegSemil}
For any \js~$S$, the completely \jirr\ members of~$\Reg{S}$ are exactly the set differences $(S\dnw p)\setminus\ba'$, for $p\in S$ with at most two lower covers in the ideal lattice of~$S$, one of them being~$\ba'$.
\end{theorem}

\begin{proof}
  It follows from Lemma~\ref{L:Jirr2MinNbhdSem} that every completely
  \jirr\ member~$\ba$ of $\Reg{S}$ is a minimal neighborhood of some
  $p\in S$. By Theorem~\ref{T:MinNbhdSemil}, there exists a lower
  cover~$\ba'$ of~$S\dnw p$ in~$\Id S$ such that $\ba=(S\dnw
  p)\setminus\ba'$. Furthermore, it follows from
  Lemma~\ref{L:DescrJirrRegP} that $\ba_*=\ba\setminus\set{p}$, so
  $\ba''=S\dnw\ba_*$ is a proper ideal of~$S\dnw p$. Since
  $\ba'\cup\ba''$ contains $\ba'\cup\ba_*=S\ddnw p$, we get
  $\ba'\cup\ba''=S\ddnw p$. By Lemma~\ref{L:nlowcov}, it follows
  that~$S\dnw p$ has at most two lower covers in~$\Id S$.

  Conversely, let $p\in S$ with at least one, but at most two, lower
  covers~$\ba'$ and~$\ba''$ in~$\Id S$. The set $\ba=(S\dnw
  p)\setminus\ba'$ is clopen. Moreover, it follows from
  Lemma~\ref{L:nlowcov} that $S\ddnw p=\ba'\cup\ba''$. Now
  $\ba\setminus\set{p}$ is a lower subset of~$\ba$, thus it is
  open. Further, $\ba\setminus\set{p}$ is contained in~$\ba''$,
  thus~$p$ does not belong to its closure; since~$\ba$ is closed, it
  follows that $\ba\setminus\set{p}$ is closed, so it is
  clopen. Let~$\bb\in\Reg{S}$ such that $p\in\bb\subseteq\ba$. Since
  $p\in \bb = \tcl\tin(\bb)$ and~$p$ is not a join of elements of
  $\ba\setminus\set{p}$, we get $p\in\tin(\bb)$. {}From $\ba'\prec
  S\dnw p$ it follows that for each $x\in\ba\setminus\set{p}$, there
  exists $a\in\ba'$ such that $p=x\vee a$. Since $p\in\tin(\bb)$ and
  $a\notin\bb$, it follows that $x\in\bb$. Therefore,
  $\ba\setminus\set{p}\subseteq\bb$, so $\bb=\ba$, thus completing the
  proof that $\ba\setminus\set{p}$ is the unique lower cover of~$\ba$
  in~$\Reg{S}$.
\end{proof}

\section{Boundedness of lattices of regular closed subsets from graphs}\label{S:PermGraph}

Let~$G$ be a graph. We denote by~$\so{G}^+$ the poset of all connected subsets of~$G$, ordered by set inclusion, and we set $\so{G}=\so{G}^+\setminus\set{\es}$. A nonempty finite subset~$\bx$ of~$\so{G}$ is a \emph{partition} of an element~$X\in\so{G}$, in notation $X=\bprt\bx$, if~$X$ is the disjoint union of all members of~$\bx$. In case $\bx=\set{X_1,\dots,X_n}$, we shall sometimes write $X=X_1\prt\cdots\prt X_n$ instead of $X=\bprt\bx$.

For any $\bx\subseteq\so{G}$, let~$\tcl(\bx)$ be the \emph{closure of~$\bx$ under disjoint unions}, that is,
 \[
 \tcl(\bx)=\Setm{X\in\so{G}}
 {(\exists\by\subseteq\bx)
 \pI{X=\bigsqcup\by}}\,.
 \]
Dually, we denote by $\tin(\bx)$ the  interior of~$\bx$, that is, the largest open subset of~$\bx$.

It is straightforward to verify that~$\tcl$ is an algebraic closure operator on~$\so{G}$.
With respect to that closure operator, a subset~$\ba$ of~$\so{G}$ is closed if{f} for any partition $X=X_1\prt\cdots\prt X_n$ in~$\so{G}$, $\set{X_1,\dots,X_n}\subseteq\ba$ implies that $X\in\ba$. Dually, $\ba$ is open if{f} for any partition $X=X_1\prt\cdots\prt X_n$ in~$\so{G}$, $X\in\ba$ implies that $X_i\in\ba$ for some~$i$. In both statements, it is sufficient to take $n=2$ (for whenever $X=X_1\prt\cdots\prt X_n$, there exists $i>1$ such that $X_1\cup X_i$ is connected, and then $X=(X_1\prt X_i)\prt\bprt_{j\notin\set{1,i}}X_j$). In our arguments about graphs, we shall often allow, by convention, the empty set in partitions, thus letting $X=\es\prt X_1\prt\cdots\prt X_n$ simply mean that $X=X_1\prt\cdots\prt X_n$.
We shall call $(\so{G},\tcl)$ the \emph{closure space canonically associated to the graph~$G$}.

For $P\in\so{G}$, a nonempty subset $\bx\subseteq\so{G}$ belongs to~$\cM(P)$ if{f}~$P$ is the disjoint union of a nonempty finite subset of~$\bx$, but of no proper subset of~$\bx$. Hence $\bx\in\cM(P)$ if{f}~$\bx$ is finite and $P=\bigsqcup\bx$, and we get the following.

\begin{proposition}\label{P:soGSemilType}
The closure space $(\so{G},\tcl)$ has semilattice type, for every graph~$G$.
\end{proposition}

Observe that $(\so{G},\subseteq)$ might
not be a \js, for example if $G = \cC_{4}$. On the other hand, if $G$ is a block graph, then $(\so{G},\subseteq)$ is a \js\ if{f}~$G$ is connected.

\begin{definition}\label{D:PermG}
The \emph{permutohedron} (resp., \emph{extended permutohedron}) on~$G$ is the set~$\sP(G)$ (resp., $\sR(G)$) of all clopen (resp., regular closed) subsets of~$\so{G}$, ordered by set inclusion. That is, $\sP(G)=\Clop(\so{G},\tcl)$ and $\sR(G)=\Reg(\so{G},\tcl)$.
\end{definition}

In particular, $\sR(G)$ is always a lattice. We will see in Section~\ref{S:ClopGraph} in which case~$\sP(G)$ is a lattice.

\begin{example}\label{Ex:PermSnGrph}
The Dynkin diagram~$G_{n}$ of the symmetric group~$\fS_n$ consists of all transpositions $\sigma_i=\begin{pmatrix}i & i+1\end{pmatrix}$, where $1\leq i<n$, with $\sigma_i\sim\sigma_j$ if{f} $i-j=\pm1$.
  
Observe that there is a bijection between the connected subgraphs of~$G_{n}$ and the pairs $(i,j)$ with $1 \leq i < j \leq n$, whereas a set of connected subsets is closed if{f} the corresponding pairs form a transitive relation. Hence, $\sP(G_{n})$ is isomorphic to the lattice of permutations on~$n$ elements, that is, to the classical permutohedron~$\sP(n)$.
\end{example}

We define the collection of all \emph{cuts}, respectively \emph{proper cuts}, of a connected subset~$H$ in a graph~$G$ as
 \begin{align*}
 \Cuts(H)&=
 \setm{X\subseteq H
 \text{ nonempty}}
 {X\text{ and }H\setminus X
 \text{ are both connected}}\,,\\
 \Cuts_*(H)&=
 \Cuts(H)\setminus\set{H}\,.
 \end{align*}
The following lemma says that any completely \jirr\ element of~$\sR(G)$ is ``open on cuts''.

\begin{lemma}\label{L:JIalmOpen}
Let~$\ba$ be a completely \jirr\ element of~$\sR(G)$, with largest element~$H$. Then $\ba\cap\Cuts(H)$ is contained in $\tin(\ba)$.
\end{lemma}

\begin{proof}
Let $X=X_1\prt\cdots\prt X_n$ with $X\in\ba\cap\Cuts(H)$, we must prove that $X_i\in\ba$ for some~$i$. If $X=H$ then this follows from $H\in\tin(\ba)$ (cf. Lemma~\ref{L:DescrJirrRegP}). Suppose from now on that $X\neq H$. The complement $Y=H\setminus X$ is connected. Furthermore, $Y\notin\ba$ (otherwise~$X$ and~$Y$ would both belong to $\ba\setminus\set{H}=\ba_*$, so $H\in\ba_*$, a contradiction) and $H=Y\prt\bprt_{1\leq i\leq n}X_i$ belongs to~$\tin(\ba)$ (cf. Lemma~\ref{L:DescrJirrRegP}), thus $X_i\in\ba$ for some~$i$.
\end{proof}

For subsets~$U$ and~$V$ in a graph~$G$, we set
 \begin{align}
 U\simeq V&\quad\text{if}\quad
 \pI{\exists(u,v)\in U\times V}
 (\text{either }u=v
 \text{ or }u\sim v)\,,
 \label{Eq:DefnUsimeqV}\\
 U\sim V&\quad\text{if}\quad
 \pII{U\cap V=\es\text{ and }
 \pI{\exists(u,v)\in U\times V}
 (u\sim v)}\,.
 \label{Eq:DefnUsimV}
 \end{align}
Hence, $U\sim V$ if{f} $U\simeq V$ and $U\cap V=\es$. Moreover, if
 $U,V\in\so{G}$, then $U\simeq V$ if{f} $U\cup V$ is connected. We
 denote by~$\coco(X)$ the set of all connected components of a
 subset~$X$ of~$G$. We omit the straightforward proof of the following
 lemma.

\begin{lemma}\label{L:HominusX}
The following statements hold, for all $X,H\in\so{G}$ with $X\subseteq H$ and all $U,V\in\coco(H\setminus X)$:
\begin{enumerate}
\item $U$ is a cut of~$H$ and $U\sim X$;

\item if $U\neq V$, then $U\not\simeq V$.
\end{enumerate}
\end{lemma}

If $X\subseteq H$ in~$\so{G}$, let $X\leq^{\oplus}H$ hold if $\coco(H\setminus X)$ is finite.

\begin{lemma}\label{L:H-Xdisjba}
Let~$G$ be a graph, let~$\ba$ be a completely \jirr\ element of~$\sR(G)$ with greatest element~$H$, and let $X\in\partial\ba$. Then $X\leq^{\oplus}H$ and\linebreak $\coco(H\setminus X)\cap\ba=\es$.
\end{lemma}

\begin{proof}
By the final statement of Lemma~\ref{L:DescrJirrRegP}, there exists $\bx\in\cM(H)$ such that $\bx\cap\ba=\set{X}$. {}From $H=X\prt\bprt\vecm{Z}{Z\in\bx\setminus\set{X}}$ it follows that $X\leq^{\oplus}H$. Now let $Y\in\coco(H\setminus X)$ and suppose that~$Y\in\ba$. Since~$Y$ is a cut of~$H$ (cf. Lemma~\ref{L:HominusX}) and by Lemma~\ref{L:JIalmOpen}, $Y\in\tin(\ba)$. Furthermore, $Y=\bprt\by$ for some $\by\subseteq\bx\setminus\set{X}$, thus $\by\cap\ba\neq\es$, and thus $(\bx\setminus\set{X})\cap\ba\neq\es$, a contradiction.
\end{proof}

The following lemma means that in the finite case, the \jirr\ members of  $\sR(G)$ are determined by their proper cuts. This result will be extended to the infinite case, with a noticeably harder proof, in Corollary~\ref{C:cltjmujirr}.

\begin{lemma}\label{L:baiiCuts2ba}
Let~$G$ be a finite graph, let~$\ba$ and~$\bb$ be \jirr\ elements of~$\sR(G)$ with the same largest element~$H$. If $\ba\cap\Cuts_*(H)=\bb\cap\Cuts_*(H)$, then $\ba=\bb$.
\end{lemma}

\begin{proof}
By symmetry, it suffices to prove that $\ba\subseteq\bb$. Since $\ba=\tcl(\partial\ba)$ (cf. Edelman and Jamison \cite[Theorem~2.1]{EdJa}), it suffices to prove that every $X\in\partial\ba$ belongs to~$\bb$. If $X=H$ this is obvious, so suppose that $X\neq H$. It follows from Lemma~\ref{L:H-Xdisjba} that $X\leq^{\oplus}H$ and $\coco(H\setminus X)\cap\ba=\es$. Since every element of $\coco(H\setminus X)$ is a (proper) cut of~$H$ and by assumption, it follows that $\coco(H\setminus X)\cap\bb=\es$. Since $H\in\tin(\bb)$ (cf. Lemma~\ref{L:DescrJirrRegP}) and $H=X\prt\bprt\vecm{Y}{Y\in\coco(H\setminus X)}$, it follows that $X\in\bb$, as desired.
\end{proof}

\begin{lemma}\label{L:babccomploncuts}
Let~$G$ be a graph, let~$\ba$ and~$\bc$ be completely \jirr\ elements of~$\sR(G)$, with the same largest element~$H$, such that $\orth{\ba}\searrow\bc$. Then $\ba\cap\Cuts_*(H)$ and $\bc\cap\Cuts_*(H)$ are complementary in $\Cuts_*(H)$.
\end{lemma}

\begin{proof}
The statement $\tin(\ba)\cap\bc=\set{H}$ is established in the Claim within the proof of Lemma~\ref{L:PerspInt2SD}. By Lemma~\ref{L:JIalmOpen}, it follows that $\ba\cap\bc\cap\Cuts(H)=\set{H}$.

Now let $X\in\Cuts(H)\setminus\ba$, we must prove that $X\in\bc$. Necessarily, $X\neq H$. {}From $H=X\prt(H\setminus X)$, $X\notin\ba$, and $H\in\tin(\ba)$ it follows that $H\setminus X\in\tin(\ba)$. Since $\tin(\ba)\cap\bc=\set{H}$, it follows that $H\setminus X\notin\bc$. {}From $H=X\prt(H\setminus X)$, $H\setminus X\notin\bc$, and $H\in\tin(\bc)$ it follows that $X\in\bc$.
\end{proof}

\begin{theorem}\label{T:bsPGBded}
The extended permutohedron $\sR(G)=\Reg(\so{G},\tcl)$ on a finite graph~$G$ is a bounded homomorphic image of a free lattice.
\end{theorem}

\begin{proof}
As in the proof of Theorem~\ref{T:SDReg2bounded}, it is sufficient to prove that $\ba\Dr\bb$ implies that $\max\ba>\max\bb$, for all \jirr\ elements~$\ba$ and~$\bb$ of~$\sR(G)$. By Lemma~\ref{L:Arr2D}, there exists a \mirr\ $\bu\in\sR(G)$ such that $\ba\nearrow\bu\searrow\bb$. The element $\bc=\orth{\bu}$ is \jirr\ and it follows from Lemma~\ref{L:PerspInt2SD} that $\max\ba\geq\max\bc\geq\max\bb$. Suppose that $\max\ba=\max\bc=\max\bb$ and denote that element by~$H$. It follows from Lemma~\ref{L:babccomploncuts} that $\pI{\ba\cap\Cuts_*(H),\bc\cap\Cuts_*(H)}$ and $\pI{\bc\cap\Cuts_*(H),\bb\cap\Cuts_*(H)}$ are both complementary pairs, within~$\Cuts_*(H)$, of proper cuts. It follows that $\ba\cap\Cuts_*(H)=\bb\cap\Cuts_*(H)$, so, by Lemma~\ref{L:baiiCuts2ba}, $\ba=\bb$, in contradiction with $\ba\Dr\bb$.
\end{proof}

\begin{example}\label{Ex:InfteMainSD}
The conclusion of Theorem~\ref{T:bsPGBded} implies, in particular, that~$\sR(G)$ is semidistributive for any finite graph~$G$. We show here that this conclusion cannot be extended to infinite graphs. The infinite path $\cP_{\go}=\set{0,1,2,\dots}$, with graph incidence defined by $i\sim j$ if $i-j=\pm1$, is an infinite tree. The subsets~$\ba$, $\bb$, $\bc$ of~$\so{\cP_{\go}}$ defined by
 \begin{align*}
 \ba&=\setm{\co{2m,\infty}}{m<\go}\cup
 \setm{[2m,2n]}{m\leq n<\go}\,,\\
 \bb&=\setm{\co{2m+1,\infty}}{m<\go}\cup
 \setm{[2m+1,2n+1]}{m\leq n<\go}\,,\\
 \bc&=\setm{\co{n,\infty}}{n<\go}
 \end{align*}
are all clopen.  Furthermore,
 \begin{align*}
 \ba\cap\bc&=
 \setm{\co{2m,\infty}}{m<\go}\,,\\
 \bb\cap\bc&=
 \setm{\co{2m+1,\infty}}{m<\go}
 \end{align*}
have both empty interior, so $\ba\wedge\bc=\bb\wedge\bc=\es$. On the other hand, $\bc\subseteq\ba\cup\bb$, thus $(\ba\vee\bb)\wedge\bc=\bc\neq\es$. Therefore, $\sR(\cP_\go)$ (which, by Theorem~\ref{T:PGLatt}, turns out to be identical to~$\sP(\cP_\omega)$) is neither pseudocomplemented nor \msd\ (so it is not \jsd\ either).
\end{example}

\section{Graphs whose permutohedron is a lattice}\label{S:ClopGraph}

The main goal of this section is a characterization of those graphs~$G$ such that the permutohedron $\sP(G)$ ($=\Clop(\so{G},\tcl)$) is a lattice. Observe that unlike Theorem~\ref{T:ClopWFLatt}, the statement of Theorem~\ref{T:PGLatt} does not require any finiteness assumption on~$G$.

\begin{theorem}\label{T:PGLatt}
The following are equivalent, for any graph~$G$:
\begin{enumerate}
\item $\sP(G)$ is a lattice.

\item The closure of any open subset of~$\so{G}$ is open \pup{i.e., $\sP(G)=\sR(G)$}.

\item $G$ is a block graph without $4$-cliques.
\end{enumerate}
\end{theorem}

\begin{proof}
(ii)$\Rightarrow$(i) is trivial.

(i)$\Rightarrow$(iii). Suppose that~$G$ is not a block graph without $4$-cliques; we shall prove that~$\sP(G)$ is not a lattice. It is easy to construct, for each graph~$P$ in the collection $\Gamma=\set{\cK_4,\cD}\cup\setm{\cC_n}{4\leq n<\go}$ (cf. Figure~\ref{Fig:Graphs}), nonempty connected subsets~$P_0$, $P_1$, $P_2$, and~$P_3$ of~$P$ such that 
 \begin{equation}\label{Eq:PiPjneqes}
 P=P_0\prt P_2=P_1\prt P_3\text{ while }
 P_i\cap P_{i+1}\neq\es\text{ for all }i<4
 \end{equation}
(with indices reduced modulo~$4$). By assumption, one of the members of~$\Gamma$ embeds into~$G$ as an induced subgraph. This yields nonempty connected subsets~$P_i$, for $0\leq i\leq 3$, and~$P$ of~$G$ satisfying~\eqref{Eq:PiPjneqes}.

It is easy to see that the set $\ba_i=\setm{X\in\so{G}}{X\subseteq P_i\text{ and }X\cap P_{i+1}\neq\es}$ is clopen in~$\so{G}$, for each $i<4$. We claim that $\ba_i\cap\ba_j=\es$ whenever $i\in\set{0,2}$ and $j\in\set{1,3}$. Indeed, suppose otherwise and let $Z\in\ba_i\cap\ba_j$. {}From $Z\subseteq P_i$ and $Z\cap P_{j+1}\neq\es$ it follows that $P_i\cap P_{j+1}\neq\es$, so $i-j\in\set{0,1,2}$, and so $i-j=1$. Likewise, $j-i=1$; \contr.

It follows that $\ba_i\subseteq\so{G}\setminus\ba_j$ for all
$i\in\set{0,2}$ and all $j\in\set{1,3}$. Suppose that there exists
$\bb\in\sP(G)$ such that
$\ba_i\subseteq\bb\subseteq\so{G}\setminus\ba_j$. {}From $P_i\in\ba_i$
and $P=P_0\prt P_2$ it follows (using the closedness of~$\bb$) that
$P\in\bb$. Since $P=P_1\prt P_3$ and~$\bb$ is open, either $P_1\in\bb$ or
$P_3\in\bb$. In the first case, we get $P_{1} \notin\ba_{1}$ from $\bb\subseteq\so{G}\setminus\ba_1$, while in the
second case we get $P_{3}\notin\ba_{3}$;  \contr\ in both cases.

(iii)$\Rightarrow$(ii).
Suppose that~$G$ is a block graph without $4$-cliques, and suppose that there exists an open set $\bu\subseteq\so{G}$ such that $\tcl(\bu)$ is not open. This means that there are $P\in\tcl(\bu)$ and a partition~$\by$ of~$P$ such that
 \begin{equation}\label{Eq:bycapclu=es}
 \by\cap\tcl(\bu)=\es\,. 
 \end{equation}
On the other hand, from $P\in\tcl(\bu)$ it follows that there exists a partition~$\bx$ of~$P$ such that
 \begin{equation}\label{Eq:bxinu}
 \bx\subseteq\bu\,. 
 \end{equation}
In particular, from~\eqref{Eq:bycapclu=es} and~\eqref{Eq:bxinu} it follows that $\bx\cap\by=\es$.
Moreover, as~$G$ is a block graph, the intersection of any two connected subsets of~$G$ is connected, hence, as $P=\bprt\bx=\bprt\by$, we get the following decompositions in~$\so{G}$:
 \begin{align}
 X&=\bprt\vecm{X\cap Y}
 {X\cap Y \neq \es\;,\ Y\in\by}&&
 (\text{for each }X\in\bx)\,,\label{Eq:DecompXxy}\\
 Y&=\bprt\vecm{X\cap Y}
 {X\cap Y \neq \es\;,\ X\in\bx}&&
 (\text{for each }Y\in\by)\,.\label{Eq:DecompYxy}
 \end{align}
For each $X\in\bx$, it follows from $X\in\bu$ (cf.~\eqref{Eq:bxinu}) and the openness of~$\bu$ that there exists $s(X)\in\by$ such that $X\cap s(X)\in\bu$. On the other hand, each $Y\in\by$ belongs to the complement of~$\tcl(\bu)$ (cf. \eqref{Eq:bycapclu=es}), thus there exists $s(Y)\in\bx$ such that $s(Y)\cap Y$ is nonempty and does not belong to~$\bu$. {}From this it is easy to deduce that
 \begin{equation}\label{Eq:sgenercycle}
 Z\cap s(Z)\neq\es\text{ and }s^2(Z)\neq Z\,,
 \quad\text{ for each }Z\in\bx\cup\by\,. 
 \end{equation}
Consider the graph with vertex set $\bz=\bx\cup\by$, and incidence
relation~$\asymp$ defined by $X\asymp Y$ if{f} $X\neq Y$ and $X\cap
Y\neq\es$, for all $X,Y\in\bz$. Since~$\bx$ and~$\by$ are both partitions
of~$P$, the graph~$\bz$ is bipartite. Let~$n$ be a positive integer,
minimal with the property that there exists $Z\in\bz$ such that
$s^n(Z)=Z$. It follows from~\eqref{Eq:sgenercycle} that
$n\geq3$. Further, by the minimality assumption on~$n$, all sets~$Z$,
$s(Z)$, \dots, $s^{n-1}(Z)$ are pairwise distinct. Since $s^k(Z)\asymp s^{k+1}(Z)$ for each~$k$, it follows that the graph $(\bz,\asymp)$ has an induced cycle of length $\geq3$. Since this graph is bipartite, the cycle above has the form $\vec{P}=(P_0,P_1,\dots,P_{2n-1})$, for some integer $n\geq2$.

Pick $g_i\in P_i\cap P_{i+1}$, for each $i<2n$ (indices are reduced modulo~$2n$). Since~$g_i$ and~$g_{i+1}$ both belong to the connected set~$P_{i+1}$, they are joined by a path~$\vec{g}_i$ contained in~$P_i$. By joining the~$\vec{g}_i$ together, we obtain a path~$\vec{g}$ (not induced \emph{a priori}) containing all the~$g_i$ as vertices.

Choose the~$\vec{g}_i$ in such a way that the length~$N$ of~$\vec{g}$ is as small as possible. Since the~$g_i$ are pairwise distinct (because~$\vec{P}$ is an induced path), $N\geq2n$.

Since~$\vec{P}$ is an induced cycle in $(\bz,\asymp)$, the (ranges of) the paths~$\vec{g}_i$ and~$\vec{g}_j$ meet if{f} $i-j\in\set{0,1,-1}$, for all $i,j<2n$. Moreover, it follows from the minimality assumption on~$N$ that $\vec{g}_i\cap\vec{g}_{i+1}=\set{g_{i+1}}$, for each $i<2n$. Therefore, the range of the path~$\vec{g}$ is biconnected, so, as~$G$ is a block graph, $\vec{g}$ is a clique, and so, by assumption, $N\leq3$, in contradiction with $N\geq2n$.
\end{proof}

\begin{example}\label{Ex:PK2K3}
(Observe the similarity with Example~\ref{Ex:RS2S3}.)
It is an easy exercise to verify that~$\sP(\cK_2)=\sR(\cK_2)$ is isomorphic to the permutohedron on three letters~$\sP(3)$, which is the six-element ``benzene lattice''.

On the other hand, the lattice~$\sP(\cK_3)=\sR(\cK_3)$ has apparently not been met until now.

Denote by~$a$, $b$, $c$ the vertices of the graph~$\cK_3$.  The
lattice~$\sP(\cK_3)$ is represented on the right hand side of
Figure~\ref{Fig:PK3}, by using the following labeling convention:
 \begin{gather*}
 \set{\set{a}}\mapsto a\,,\quad
 \set{\set{a,b},\set{a}}\mapsto
 a^2b\,,\quad
 \set{\set{a,b},\set{a},\set{b}}
 \mapsto a^2b^2\,,\\
 \set{\set{a,b},\set{a,c},\set{a}}
 \mapsto a^3bc
 \end{gather*}
(the ``variables'' $a$, $b$, $c$ being thought of as pairwise commuting, so for example $a^2b=ba^2$), then $\ol{\es}=\so{\cK_3}$, $\ol{a}^2\ol{b}^2=\so{\cK_3}\setminus (a^2b^2)$, and so on.

\begin{figure}[htb]
\centering
\includegraphics{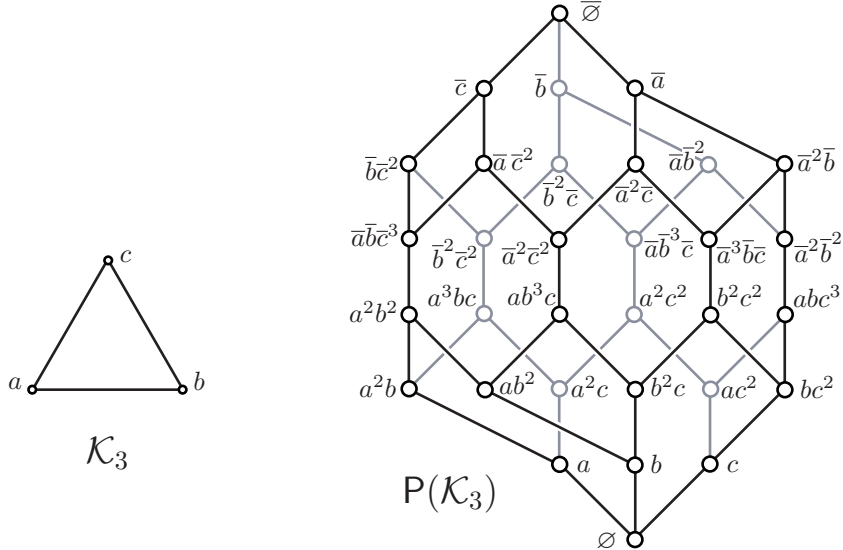}
\caption{The permutohedron on the graph $\cK_3$}
\label{Fig:PK3}
\end{figure}

While we prove in~\cite{SaWe12a} that every open subset of a transitive binary relation is a union of clopen subsets, the open subset $\bu=\set{a,b,c,abc}$ of~$\so{\cK_3}$ is not a union of clopen subsets.

The \jirr\ elements of~$\sP(\cK_3)$ are~$a$, $a^2b$, $\ol{a}\ol{b}\ol{c}^3$, $\ol{a}^2\ol{b}^2$, and cyclically. They are all closed under intersection, and they never contain all the members of a nontrivial partition.

By Theorem~\ref{T:PGLatt}, the permutohedron~$\sP(\cK_4)$ is not a lattice. Brute force computation shows that $\card\sP(\cK_4)=370$ while $\card\sR(\cK_4)=382$. Every \jirr\ element of~$\sR(\cK_4)$ belongs to~$\sP(\cK_4)$. Labeling the vertices of~$\cK_4$ as~$a$, $b$, $c$, $d$, we get the \jirr\ element $\set{b,c,ab,ac,bc,abc,bcd,abcd}$ in~$\Reg\cK_4$. It contains $ab$ and $ac$ but not their intersection~$a$. It contains all entries of the partition $bc=b\prt c$.

\end{example}

Variants of~$\sP(G)$ and~$\sR(G)$, with the collection of all connected subsets of~$G$ replaced by other alignments, are studied in more detail in Santocanale and Weh\-rung~\cite{CSW2}.

\section{Completely \jirr\ regular closed sets in graphs}\label{S:CJIBlockGraph}

While $\Reg S$ is always the Dedekind-MacNeille completion of $\Clop
S$, for any \js~$S$ (cf. Corollary~\ref{C:MinNbhdSemil2}), the
situation for graphs is more complex. In this section we shall give a
convenient description of the completely \jirr\ members of~$\sR(G)$,
in terms of so-called \emph{pseudo-ultrafilters} on members
of~$\so{G}$, for an arbitrary graph~$G$. This will imply that the completely \jirr\ elements 
are determined by the proper cuts of their top element, thus extending Lemma~\ref{L:baiiCuts2ba}
to the infinite case (cf. Corollary~\ref{C:cltjmujirr}). In addition,
this will yield a large class of graphs~$G$ for which every completely
\jirr\ member of~$\sR(G)$ is clopen (cf. Theorem~\ref{T:DiamContr} and
Corollary~\ref{C:CJIBGorCyc}).

In this section we will constantly refer to the restrictions to~$\so{G}$ of the binary relations~$\simeq$ and~$\sim$ introduced in~\eqref{Eq:DefnUsimeqV} and~\eqref{Eq:DefnUsimV}. {}From Lemma~\ref{L:HminusX+Y} to Proposition~\ref{P:cltjmujirr} we shall fix a graph~$G$ and a nonempty connected subset~$H$ of~$G$.
  
\begin{lemma}\label{L:HminusX+Y}
  Let $X,Y,Z\in\so{H}$ with $Z=X\prt Y$. Denote by~$\widehat{X}$ the
  unique member of~$\coco(H\setminus X)$ containing~$Y$, and
  define~$\widehat{Y}$ similarly with~$X$ and~$Y$ interchanged. The
  following statements hold:
\begin{enumerate}
\item Every member of $\coco(H\setminus X)\setminus\set{\widehat{X}}$ is contained in~$\widehat{Y}$, and symmetrically with~$(X,\widehat{X})$ and~$(Y,\widehat{Y})$ interchanged.

\item $H=\widehat{X}\cup\widehat{Y}$.

\item $\coco(H\setminus Z)=\coco(\widehat{X}\cap\widehat{Y})\cup\pI{\coco(H\setminus X)\setminus\set{\widehat{X}}}\cup\pI{\coco(H\setminus Y)\setminus\set{\widehat{Y}}}$.

\item $T\sim X$ and $T\sim Y$, for any $T\in\coco(\widehat{X}\cap\widehat{Y})$.
\end{enumerate}
\end{lemma}

\begin{proof}
Let $U\in\coco(H\setminus X)\setminus\set{\widehat{X}}$. Suppose first that $U\simeq Y$ (cf.~\eqref{Eq:DefnUsimeqV}). Since $U\cup Y$ is connected, disjoint from~$X$, and contains~$Y$, it is contained in~$\widehat{X}$, hence $U=\widehat{X}$, a contradiction. Hence $U\not\simeq Y$, so $U\subseteq H\setminus Y$. Since $X\subseteq H\setminus Y$ and $U\sim X$ (cf. Lemma~\ref{L:HominusX}), it follows that $U\subseteq\widehat{Y}$, thus completing the proof of~(i).

Now set $X'=\bigcup\pI{
 \coco(H\setminus\set{X})
 \setminus\set{\widehat{X}}}$, and define~$Y'$ symmetrically. It follows from~(i) that $X\cup X'\subseteq\widehat{Y}$. Since $X'\cup\widehat{X}=\bigcup\coco(H\setminus\set{X})=H\setminus X$, it follows that $H=X\cup X'\cup\widehat{X}\subseteq\widehat{X}\cup\widehat{Y}$ and~(ii) follows.
 
As a further consequence of~(i), $(X\cup X')\cap(Y\cup Y')\subseteq\widehat{Y}\cap(Y\cup Y')=\es$. Since $H=X\cup X'\cup\widehat{X}=Y\cup Y'\cup\widehat{Y}$, it follows that $H=X\cup X'\cup Y\cup Y'\cup(\widehat{X}\cap\widehat{Y})$ (disjoint union), so the union of the right hand side of~(iii) is $H\setminus Z$. Furthermore, every element of the right hand side of~(iii) is, by definition, nonempty and connected. Finally, it follows easily from~(i) together with Lemma~\ref{L:HominusX} that any two distinct members~$U$ and~$V$ of the right hand side of~(iii) satisfy $U\not\simeq V$; (iii) follows.

(iv). Since~$\widehat{X}$ is connected and contains $T\cup Y$, there exists a path~$\gamma$, within~$\widehat{X}$, from an element $t\in T$ to an element of~$Y$. We may assume that the successor~$y$ of~$t$ in~$\gamma$ does not belong to~$T$. Recall now that~$y$ belongs to $H = Y \cup \bigcup \coco(H\setminus Y)$. If $y\in Y'$ for some $Y'\in\coco(H\setminus Y)\setminus\set{\widehat{Y}}$, then we get, from $t\in\widehat{Y}$, that $\widehat{Y}\sim Y'$, a contradiction. If $y\in\widehat{Y}$ then $T\sim W$ for some $W\in\coco(\widehat{X}\cap\widehat{Y})$ distinct of~$T$, a
contradiction. Hence, $y\in Y$, so $T\sim Y$. Symmetrically, $T\sim
X$.
\end{proof}

\begin{definition}\label{D:PFil}
A \emph{pseudo-ultrafilter} on~$H$ is a subset $\mu \subseteq \Cuts(H)$ such that $H\in\mu$ and whenever~$X$, $Y$, $Z$ are cuts of~$H$ such that $Z=X\prt Y$,
\begin{enumerate}
\item $X\in\mu$ and $Y\in\mu$ implies that $Z\in\mu$;

\item $X\notin\mu$ and $Y\notin\mu$ implies that $Z\notin\mu$;

\item $X\in\mu$ if{f} $H\setminus X\notin\mu$, whenever~$X$ is a proper cut of~$H$.
\end{enumerate}
\end{definition}

Observe that~$H$ is necessarily the largest element of~$\mu$. We leave to the reader the straightforward proof of the following lemma.

\begin{lemma}\label{L:mu2tildemu}
If~$\mu$ is a pseudo-ultrafilter on~$H$, then so is the \emph{conjugate pseudo-ultrafilter} $\tilde{\mu}=\pI{\Cuts(H)\setminus\mu}\cup\set{H}$.
\end{lemma}

Given a pseudo-ultrafilter $\mu$ on $H$, we define
 \begin{align*}
 \tj(\mu)&=
 \setm{X\in\so{H}}
 {X\leq^{\oplus}H\text{ and }
 \coco(H\setminus X)\cap\mu
 =\es}\,,\\
 \tj_{*}(\mu)&=\tj(\mu)\setminus\set{H}\,.
 \end{align*}
 We shall fix, until Proposition~\ref{P:cltjmujirr}, a
 pseudo-ultrafilter~$\mu$ on~$H$. It is obvious that $\tj(\mu) \cap \Cuts(H) = \mu$. This observation is extended in the following lemma.

\begin{lemma}\label{L:tjmuoncuts}
$\mu=\tcl\pI{\tj(\mu)}\cap\Cuts(H)$.
\end{lemma}

\begin{proof}
  We prove the nontrivial containment. We must prove that if
  $X=\bprt_{i<m}X_i$, with each $X_i\in\tj(\mu)$ and~$X$ a cut, then
  $X\in\mu$. If $X=H$ the conclusion is trivial, so we suppose that
  $X\neq H$. The complement $Y=H\setminus X$ is a proper cut of~$H$.  We argue by induction on~$m$. For $m=1$ the proof is
  straightforward, as $\tj(\mu) \cap \Cuts(H) \subseteq \mu$; let us suppose therefore that $m \geq 2$.

  \setcounter{claim}{0}

  \begin{claim}\label{Cl:allXicuts}
$X_i\in\mu$, for each $i<m$.
  \end{claim}
  
  \begin{cproof}
    Suppose that~$X_i$ is not a cut, denote by~$Y'$ the unique connected component
    of $H\setminus X_i$ containing~$Y$, and let $V\in\coco(H\setminus
    X_i)\setminus\set{Y'}$. Since~$V$ is both a cut and a (disjoint)
    union of some members of $\setm{X_j}{j\neq i}$, it follows from
    the induction hypothesis that $V\in\mu$. On the other hand, from
    $X_i\in\tj(\mu)$ together with the definition of~$\tj(\mu)$, it
    follows that $V\notin\mu$, a contradiction.
    
Since~$X_i$ is a cut and $\tj(\mu)\cap\Cuts(H)\subseteq\mu$, the conclusion follows.
  \end{cproof}

  By way of contradiction, we suppose next that $X \not\in \mu$, so
  that the cut $Y=H\setminus X$ belongs to~$\mu$. Since $H=Y\prt\bprt_{i<m}X_i$ is connected, there exists $i_{0}<m$ such
  that $X_{i_{0}}\sim Y$ (cf.~\eqref{Eq:DefnUsimV}); so $X_{i_{0}}\cup Y=X_{i_{0}}\prt Y$.

\begin{claim}\label{Cl:1conncomp}
The set $X_{i_{0}}\prt Y$ is a cut of~$H$.
\end{claim}

\begin{cproof}
  Suppose otherwise and let $U$, $V$ be distinct connected components
  of $H\setminus(X_{i_{0}}\cup Y)$. It follows from
  Claim~\ref{Cl:allXicuts} that $H\setminus X_{i_{0}}$ is connected,
  thus there exists a path~$\gamma$, within that set, between an
  element of~$U$ and an element of~$V$. The path~$\gamma$ meets
  necessarily~$Y$. Since~$Y$ is connected, we may assume that~$\gamma$
  enters and exits~$Y$ exactly once. If~$p$ (resp., $q$) denotes the
  entry (resp., exit) point of~$\gamma$ in~$Y$, then the predecessor
  of~$p$ in~$\gamma$ belongs to~$U$ and the successor of~$q$
  in~$\gamma$ belongs to~$V$. It follows that $U\sim Y$ and $V\sim Y$
  (cf.~\eqref{Eq:DefnUsimV}); hence $W\sim Y$ for each
  $W\in\coco\pI{H\setminus(X_{i_{0}}\cup Y)}$. Since~$Y$ is also a
  cut, a similar proof yields that $W\sim X_{i_{0}}$ for each
  $W\in\coco\pI{H\setminus(X_{i_{0}}\cup Y)}$.  Now every such~$W$ is
  a cut, and also a disjoint union of members of $\setm{X_k}{k\neq
    i}$, hence, by the induction hypothesis, $W\in\mu$. Fix such
  a~$W$. Then $X_{i_{0}}\prt W$ is a cut, whose complement in~$H$ is
  $Y\prt\bprt\pII{\coco\pI{H\setminus(X_{i_{0}}\cup
      Y)}\setminus\set{W}}$.  Since the cardinality of
    $\coco\pI{H\setminus(X_{i_{0}}\cup Y)}$ is smaller than $m$, it
    follows from the induction hypothesis that
    $H\setminus(X_{i_{0}}\prt W)$ belongs to~$\mu$. Further, as
    $X_{i_{0}}$, $W$, $X_{i_{0}}\prt W$ are cuts and $X_{i_{0}},W \in
    \mu$, it follows from the definition of a pseudo-ultrafilter that
    $X_{i_{0}}\prt W \in \mu$. Therefore, we obtain two complementary
    cuts both belonging to~$\mu$, a contradiction.
\end{cproof}

Since $H\setminus(X_{i_{0}}\prt Y)$ is a disjoint union of members of
$\setm{X_j}{j\neq i_{0}}$, it follows from the induction hypothesis
together with Claim~\ref{Cl:1conncomp} that $H\setminus(X_{i_{0}}\prt Y)\in\mu$. As $X_{i_{0}}$, $Y$, $X_{i_{0}}\prt Y$ are cuts and $X_{i_{0}},Y \in \mu$, we deduce that $X_{i_{0}}\prt Y \in \mu$, in contradiction with~$\mu$ being a pseudo-ultrafilter. This ends the proof of Lemma~\ref{L:tjmuoncuts}.
\end{proof}

\begin{lemma}\label{L:Hirredcljmu}
$H\notin\tcl\pI{\tj_*(\mu)}$.
\end{lemma}

\begin{proof}
Suppose, otherwise, that $H=\bprt_{1\leq i\leq n}X_i$\,, with $n\geq2$ and each $X_i\in\tj(\mu)$. Define a binary relation~$\sim$ on~$[n]$ by letting $i\sim j$ hold if{f} $X_i\sim X_j$. There exists $i\in[n]$ such that~$\set{i}$ is a cut of $([n],\sim)$ (e.g., take~$i$ at maximum $\sim$-distance from~$1$). Then~$X_i$ is a cut of~$H$, so it belongs to~$\mu$. By Lemma~\ref{L:tjmuoncuts}, $H\setminus X_i=\bprt_{j\neq i}X_j$ also belongs to~$\mu$, a contradiction.
\end{proof}

\begin{lemma}\label{L:tjmuopen}
The set~$\tj(\mu)$ is open.
\end{lemma}

\begin{proof}
Let $X,Y,Z\in\so{G}$ such that $Z=X\prt Y$ and $Z\in\tj(\mu)$. {}From $Z\leq^{\oplus}H$ it follows immediately that $X\leq^{\oplus}H$ and $Y\leq^{\oplus}H$. By Lemma~\ref{L:HminusX+Y}, we can write
 \begin{align}
 \coco(H\setminus X)&=\setm{X_i}{0\leq i\leq m}\,,
 \label{Eq:cocoH-X}\\
 \coco(H\setminus Y)&=\setm{Y_j}{0\leq j\leq n}
 \label{Eq:cocoH-Y}
 \intertext{without repetitions (e.g., $i\mapsto X_i$ is one-to-one)
   and with $X\subseteq Y_0$ and $Y\subseteq X_0$\,,} \coco(X_0\cap
 Y_0)&=\setm{Z_k}{k<\ell} \qquad\text{without repetitions}\,,
 \label{Eq:cocoX0Y0}\\
 \coco(H\setminus Z)&=\setm{Z_k}{k<\ell}\cup
 \setm{X_i}{1\leq i\leq m}\cup
 \setm{Y_j}{1\leq j\leq n}\,,
 \label{Eq:cocoH-Z}
 \end{align}
 for natural numbers~$m$, $n$, $\ell$. By~\eqref{Eq:cocoH-Z}, the
 assumption $Z\in\tj(\mu)$ means that all $Z_k\notin\mu$ while
 $X_i\notin\mu$ whenever $i>0$ and $Y_j\notin\mu$ whenever $j>0$.

Now suppose that $X,Y\notin\tj(\mu)$. By the paragraph above together
with~\eqref{Eq:cocoH-X} and~\eqref{Eq:cocoH-Y}, this means that
$X_0,Y_0\in\mu$. Since~$\mu$ is a pseudo-ultrafilter, $H\setminus
X_0\notin\mu$ and $H\setminus Y_0\notin\mu$. Since those two proper
cuts are disjoint (cf. Lemma~\ref{L:HminusX+Y}(ii)), we get a
partition $H=(H\setminus X_0)\prt(H\setminus
Y_0)\prt\bprt_{k<\ell}Z_k$ with all members belonging to the conjugate
pseudo-ultrafilter~$\tilde{\mu}$ (cf. Lemma~\ref{L:mu2tildemu}). By
Lemma~\ref{L:Hirredcljmu} (applied to~$\tilde{\mu}$), this is a
contradiction.
\end{proof}

We make cash of our previous observations with the following result. 

\begin{proposition}\label{P:cltjmujirr}
The set $\tcl\pI{\tj(\mu)}$ is completely \jirr\ in~$\sR(G)$, with lower cover $\tcl\pI{\tj(\mu)}\setminus\set{H}=\tcl\pI{\tj_{*}(\mu)}$.
\end{proposition}

\begin{proof}
It follows from Lemma~\ref{L:tjmuopen} that $\tcl\pI{\tj(\mu)}$ is regular closed. Furthermore, it follows from Lemma~\ref{L:Hirredcljmu} that $\tcl\pI{\tj(\mu)}\setminus\set{H}$ is closed; the equality $\tcl\pI{\tj(\mu)}\setminus\set{H}=\tcl\pI{\tj_{*}(\mu)}$ follows. Since $\tj_{*}(\mu)$ is a lower subset of the open set~$\tj(\mu)$, it is open as well; hence $\tcl\pI{\tj_{*}(\mu)}$ is regular closed.

It remains to prove that every regular closed subset $\ba$ of~$\tcl\pI{\tj(\mu)}$ with $H\in\ba$ contains~$\tj(\mu)$ (and thus is equal to $\tcl\pI{\tj(\mu)}$). If $H\notin\tin(\ba)$, then, since~$H$ belongs to $\ba=\tcl\tin(\ba)$, there is a partition of the form $H=\bprt_{i<m}X_i$ with $m\geq2$ and each $X_i\in\ba$, thus each $X_i\in\tcl\pI{\tj_{*}(\mu)}$, and thus $H\in\tcl\pI{\tj_{*}(\mu)}$, a contradiction. Hence $H\in\tin(\ba)$. Now let $X\in\tj(\mu)$ and write $\coco(H\setminus X)=\setm{X_i}{i<n}$. Then each~$X_i$ is a cut and $X_i\notin\mu$, thus $X_i\notin\tcl\pI{\tj(\mu)}$ by Lemma~\ref{L:tjmuoncuts}, and thus $X_i\notin\ba$. Since $H\in\tin(\ba)$ and $H=X\prt\bprt_{i<n}X_i$, it follows that $X\in\ba$, as desired.
\end{proof}

Now we are ready to prove the main result of this section.

\begin{theorem}\label{T:cltjmujirr}
Let~$G$ be a graph. Then the completely \jirr\ elements of~$\sR(G)$ are exactly the sets $\tcl\pI{\tj(\mu)}$, for pseudo-ultrafilters~$\mu$.
\end{theorem}

\begin{proof}
One direction is provided by Proposition~\ref{P:cltjmujirr}, so that it suffices to prove that every completely \jirr\ element~$\ba$ of~$\sR(G)$ has the form $\tcl\pI{\tj(\mu)}$.

It follows from Lemma~\ref{L:DescrJirrRegP} that~$\ba$ has a largest element~$H$, and $H\in\tin(\ba)$. Since~$\ba$ is closed, the set $\mu=\ba\cap\Cuts(H)$ satisfies item~(i) of Definition~\ref{D:PFil}.

Let $X$, $Y$, and~$Z$ be cuts of~$H$ such that $Z=X\prt Y$ and $Z\in\mu$. It follows from Lemma~\ref{L:JIalmOpen} that $Z\in\tin(\ba)$, so either $X\in\ba$ or $Y\in\ba$, that is, either $X\in\mu$ or $Y\in\mu$,  thus completing the proof of item~(ii) of
Definition~\ref{D:PFil}. 

Since $\ba_*=\ba\setminus\set{H}$ is closed and $H \in
  \tin(\ba)$, item~(iii) of Definition~\ref{D:PFil} is also
satisfied. Therefore, $\mu$ is a pseudo-ultrafilter.

We claim that $\tj(\mu)\subseteq\ba$. Let $X\in\tj(\mu)$ and write $\coco(H\setminus X)=\setm{X_i}{i<n}$. Then each~$X_i$ is a cut of~$H$ and $X_i\notin\mu$, so $X_i\notin\ba$. Since $H\in\tin(\ba)$ and $H=X\prt\bprt_{i<n}X_i$, it follows that $X\in\ba$, as desired.

By Lemma~\ref{L:tjmuopen}, $\tcl\pI{\tj(\mu)}$ is regular closed. This set contains~$H$ as an element, and, by the paragraph above, it is contained in~$\ba$. Hence $\tcl\pI{\tj(\mu)}$ is not contained in $\ba_*=\ba\setminus\set{H}$, and hence $\tcl\pI{\tj(\mu)}=\ba$.
\end{proof}

We obtain the following strengthening of Lemma~\ref{L:baiiCuts2ba}.

\begin{corollary}\label{C:cltjmujirr}
Let~$G$ be a graph, let~$\ba$ and~$\bb$ be completely \jirr\ elements of~$\sR(G)$ with the same largest element~$H$. If $\ba\cap\Cuts_*(H)\subseteq\bb\cap\Cuts_*(H)$, then $\ba=\bb$.
\end{corollary}

\begin{proof}
By Theorem~\ref{T:cltjmujirr}, the sets $\ga=\ba\cap\Cuts(H)$ and $\gb=\bb\cap\Cuts(H)$ are pseudo-ultrafilters on~$H$, $\ba=\tcl\pI{\tj(\ga)}$, and $\bb=\tcl\pI{\tj(\gb)}$. By assumption, $\ga\subseteq\gb$. Any $X\in\gb\setminus\ga$ is a proper cut of~$H$ and $H\setminus X\in\ga\subseteq\gb$, which contradicts $X\in\gb$; so $\ga=\gb$, and the desired conclusion follows.
\end{proof}

As we will see in Theorem~\ref{T:NonClopjirr}, a completely \jirr\ element of~$\sR(G)$ may not be clopen. However, the following result provides a large class of graphs for which every completely \jirr\ regular closed set is clopen.

A graph~$G$ is \emph{contractible} to a graph~$H$ if~$H$ can be obtained from~$G$ by contracting some edges; that is, there is a surjective map $\gf\colon G\twoheadrightarrow H$ with connected fibers such that for all distinct $x,y\in H$, $\gf^{-1}\set{x}\sim\gf^{-1}\set{y}$ if{f} $x\sim y$. If $H=\cD$, the diamond (cf. Figure~\ref{Fig:Graphs}), we say that~$G$ is \emph{diamond-contractible}.

\begin{theorem}\label{T:DiamContr}
The following are equivalent, for any graph~$G$:
\begin{enumerate}
\item $\tj(\mu)$ is clopen, for any pseudo-ultrafilter~$\mu$ on a member of~$\so{G}$.

\item $G$ has no diamond-contractible induced subgraph.
\end{enumerate}
Furthermore, if~\textup{(ii)} holds, then every completely \jirr\ element of~$\sR(G)$ is clopen.
\end{theorem}

\begin{proof}
Suppose first that~$G$ has a diamond-contractible induced subgraph~$H$. There exists a partition $H=X\prt Y\prt U\prt V$ in~$\so{H}$ such that $(X,Y,U,V)$ forms a diamond in $(\so{H},\sim)$, with diagonal $\set{X,Y}$. Pick $v\in V$. The set
 \[
 \mu=\setm{Z\in\so{H}}{v\notin Z}
 \cup\set{H}
 \]
is a pseudo-ultrafilter on~$H$, with $X,Y,U\in\mu$ (so $X,Y,U\in\tj(\mu)$) and $V\notin\mu$. Since $\coco\pI{H\setminus(X\cup Y)}=\set{U,V}$ meets~$\mu$, $X\prt Y\notin\tj(\mu)$, so~$\tj(\mu)$ is not closed.

Conversely, suppose that~$G$ has no diamond-contractible induced subgraph, let~$\mu$ be a pseudo-ultrafilter on $H\in\so{G}$, and let $Z=X\prt Y$ in~$\so{H}$ with $X,Y\in\tj(\mu)$. We can write
 \[
 \coco(H\setminus X)=\setm{X_i}
 {0\leq i\leq m}\quad\text{and}\quad
 \coco(H\setminus Y)=\setm{Y_j}
 {0\leq j\leq n}\,,
 \]
without repetitions, with all~$X_i\notin\mu$ and all $Y_j\notin\mu$, and in such a way that $X\subseteq Y_0$ and $Y\subseteq X_0$.

\begin{sclaim}
The set $X_0\cap Y_0$ is either empty, or a cut. Furthermore, $X_0\cap Y_0\notin\mu$.
\end{sclaim}

\begin{scproof}
By Lemma~\ref{L:HminusX+Y}(iv), for any distinct $U,V\in\coco(X_0\cap Y_0)$, the quadruple $(X,Y,U,V)$ forms a diamond in~$\so{H}$, with diagonal $\set{X,Y}$. Hence the (connected) induced subgraph $H'=X\cup Y\cup U\cup V$ is diamond-contractible, a contradiction. Therefore, $X_0\cap Y_0$ is connected.

Since $X_0$ and~$Y_0$ are both proper cuts, so are $H\setminus X_0$ and $H\setminus Y_0$. Moreover, $X\subseteq H\setminus X_0$,
$Y\subseteq H\setminus Y_0$, and $X\sim Y$, thus $H\setminus
X_0\simeq H\setminus Y_0$, and thus $H\setminus(X_0\cap
Y_0)=(H\setminus X_0)\cup(H\setminus Y_0)$ is connected.

{}From $X_0,Y_0\notin\mu$ it follows that both $H\setminus X_0$ and $H\setminus Y_0$ belong to~$\mu$. But those two sets are disjoint (cf. Lemma~\ref{L:HminusX+Y}(ii)), hence their union, namely $H\setminus(X_0\cap Y_0)$, belongs to~$\mu$; whence $X_0\cap Y_0\notin\mu$.
\end{scproof}

Now it follows from Lemma~\ref{L:HminusX+Y}(iii) that the set
 \[
 \coco(H\setminus Z)=\pI{\set{X_0\cap Y_0}\setminus\set{\es}}\cup\setm{X_i}{1\leq i\leq m}\cup\setm{Y_j}{1\leq j\leq n}
 \]
is disjoint from~$\mu$; that is, $Z\in\tj(\mu)$.

If~(ii) holds, then it follows from the equivalence above together with Theorems~\ref{L:tjmuopen} and~\ref{T:cltjmujirr} that every completely \jirr\ element of~$\sR(G)$ is clopen.
\end{proof}

In particular, it is easy to verify that if~$G$ is either a block graph or a cycle, then no induced subgraph of~$G$ is diamond-contractible. Therefore, by putting together Theorem~\ref{T:DiamContr} and Lemma~\ref{L:MacNeille}, we obtain the following result.

\begin{corollary}\label{C:CJIBGorCyc}
Let~$G$ be a finite graph. If~$G$ is either a block graph or a cycle, then $\sR(G)$ is the Dedekind-MacNeille completion of~$\sP(G)$.
\end{corollary}

\begin{remark}\label{Rk:ClopenRG}
  The statement that all~$\tj(\mu)$ are clopen, although it implies
  that every completely \jirr\ element of~$\sR(G)$ is clopen, is not
  equivalent to that statement: as a matter of fact, $\tcl(\tj(\mu))$ might be open while properly containing~$\tj(\mu)$. For example, for the diamond graph~$\cD$, every completely \jirr\ element in~$\sR(\cD)$ is clopen, but not every~$\tj(\mu)$ is clopen (or even regular open).
\end{remark}

\section{Lattices of clopen sets for poset and semilattice type}\label{S:ClopLat}

Some of the results that we have established in earlier sections, in the particular cases of semilattices ($\Reg S$ and $\Clop S$) or graphs ($\sR(G)$ and~$\sP(G)$), about $\Reg(P,\gf)$ being the Dedekind-MacNeille completion of $\Clop(P,\gf)$, can be extended to arbitrary closure spaces of semilattice type; some others cannot. In this section we shall survey some of the statements that can be extended, and give counterexamples to some of those that cannot be extended.

The following crucial lemma expresses the abundance of clopen subsets in any well-founded closure space of semilattice type.

\begin{lemma}\label{L:AbundClop}
Let $(P,\gf)$ be a well-founded closure space of semilattice type.
Let~$\ba$ and~$\bu$ be subsets of~$P$ with~$\ba$ clopen and~$\bu$ open. If $\bu\not\subseteq\ba$, then $(\ba\dnw p)\cup\set{p}$ is clopen for any minimal element~$p$ of $\bu\setminus\ba$.
\end{lemma}

\begin{proof}
For any $q\in\gf\pI{(\ba\dnw p)\cup\set{p}}$, there exists $\bx\in\cM(q)$ with $\bx\subseteq(\ba\dnw p)\cup\set{p}$. If $\bx\subseteq\ba\dnw p$, then, as~$\ba\dnw p$ is closed (cf. Lemma~\ref{L:downpReg}),
$q\in\ba\dnw p$. If $\bx\not\subseteq\ba\dnw p$, then $p\in\bx$, thus (as $q=\bigvee\bx$ and $\bx\subseteq P\dnw p$) $q=p$. In both cases, $q\in(\ba\dnw p)\cup\set{p}$. Hence $(\ba\dnw p)\cup\set{p}$ is closed.
  
Now let $q\in(\ba\dnw p)\cup\set{p}$ and let $\by\in\cM(q)$ be nontrivial.  {}From $q=\bigvee\by$ it follows that $\by\subseteq P\dnw q$. If $q\in\ba\dnw p$, then, as~$\ba$ is open, $\by\cap\ba\neq\es$, so $\by\cap(\ba\dnw p)\neq\es$. Suppose now that $q=p$. As~$\by$ is a nontrivial covering of~$p$ and $p=\bigvee\by$, every element of~$\by$ is smaller than~$p$, thus, by the minimality assumption on~$p$, we get $\by\cap\bu\subseteq\ba$, so $\by\cap\bu\subseteq\ba\dnw p$. As~$\bu$ is open, $p \in \bu$, and $\by \in \cM(p)$, we get $\by\cap\bu\neq\es$; whence $\by\cap(\ba\dnw p)\neq\es$. Hence $(\ba\dnw p)\cup\set{p}$ is open.
\end{proof}

\begin{theorem}\label{T:AbundClop}
Let $(P,\gf)$ be a well-founded closure space of semilattice type. Then the poset $\Clop(P,\gf)$ is tight in~$\Reg(P,\gf)$.
\end{theorem}

\begin{proof}
Let $\vecm{\ba_i}{i\in I}$ be a family of clopen subsets of~$P$, having a meet~$\ba$ in~$\Clop(P,\gf)$. It is obvious that~$\ba$ is contained in the open set $\bu=\cgf\pI{\bigcap\vecm{\ba_i}{i\in I}}$. Suppose that the containment is proper. By Lemma~\ref{L:AbundClop}, there exists $p\in\bu\setminus\ba$ such that $(\ba\dnw p)\cup\set{p}$ is clopen. {}From $p\in\bu$ it follows that $p\in\ba_i$ for each~$i$, thus $(\ba\dnw p)\cup\set{p}\subseteq\ba_i$ for each~$i$, and thus, by the definition of~$\ba$, we get $(\ba\dnw p)\cup\set{p}\subseteq\ba$, so $p\in\ba$, \contr. Therefore, $\bu=\ba$ is clopen, so the meet of the~$\ba_i$ in~$\Reg(P,\gf)$, which is equal to~$\gf(\bu)$ (cf. Lemma~\ref{L:BasicReg}), is equal to~$\ba$ as well.
\end{proof}

The analogue of Theorem~\ref{T:AbundClop} for regular closed subsets in a transitive binary relation holds as well, see Santocanale and Wehrung~\cite{SaWe12a}. This is also the case for semilattices, see Corollary~\ref{C:MinNbhdSemil2}.
For a precursor of those results, for permutohedra on posets, see Pouzet~\emph{et al.} \cite[Lemma~11]{PRRZ}.

\begin{theorem}\label{T:ClopWFLatt}
Let $(P,\gf)$ be a well-founded closure space of semilattice type. Then $\Clop(P,\gf)$ is a lattice if{f} $\Clop(P,\gf)=\Reg(P,\gf)$.
\end{theorem}

\begin{proof}
We prove the nontrivial direction. Suppose that~$\Clop(P,\gf)$ is a lattice. Given a regular open subset~$\bu$ of~$P$, we must prove that~$\bu$ is closed (thus clopen). As~$\gf$ is an algebraic closure operator, $\Clop(P,\gf)$ is closed under directed unions, thus it follows from Zorn's Lemma that the set of all clopen subsets of~$\bu$ has a maximal element, say~$\ba$.

Suppose that~$\ba$ is properly contained in~$\bu$. Since~$P$ is well-founded and by Lem\-ma~\ref{L:AbundClop}, there exists $p\in\bu\setminus\ba$ such that $\bb=(\ba\dnw p)\cup\set{p}$ is clopen. By assumption, the pair $\set{\ba,\bb}$ has a join~$\bd$ in~$\Clop(P,\gf)$. Furthermore, it follows from Theorem~\ref{T:AbundClop} that $\bd=\gf(\ba\cup\bb)$, whence $\bd\subseteq\gf(\bu)$, so $\bd=\cgf(\bd)\subseteq\cgf\gf(\bu)=\bu$. Since $\ba\subseteq\bd$ and~$\bd$ is clopen, it follows from the maximality statement on~$\ba$ that $\ba=\bd$, thus $p\in\ba$, \contr.
\end{proof}

\begin{examplepf}\label{Ex:FinPosClop}
A finite closure system $(P,\gf)$ of semilattice type with a non open, \jirr\ element of $\Reg(P,\gf)$.
\end{examplepf}

\begin{note}
By Theorem~\ref{T:ClopWFLatt}, for such an example, $\Clop(P,\gf)$ cannot be a lattice. By Lem\-ma~\ref{L:MacNeille}, $\Reg(P,\gf)$ is not the Dedekind-MacNeille completion of $\Clop(P,\gf)$.
\end{note}

\begin{figure}[htb]
\includegraphics{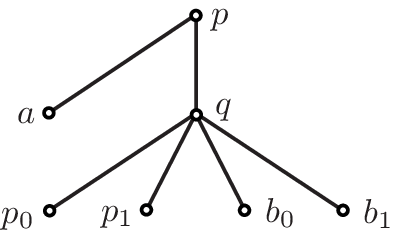}
\caption{The poset $P$ of Example \ref{Ex:FinPosClop}}\label{Fig:FinPosClop}
\end{figure}

\begin{proof}
Consider the poset~$P$ represented in Figure~\ref{Fig:FinPosClop}, and say that a subset~$\bx$ of~$P$ is \emph{closed} if
 \begin{align*}
 \set{a,p_i}\subseteq\bx&\ \Rightarrow\ 
 p\in\bx\,,
 &&\text{for each }i\in\set{0,1}\,,\\
 \set{p_0,p_1}\subseteq\bx&\ \Rightarrow\ 
 q\in\bx\,,\\
 \set{b_0,b_1}\subseteq\bx&\ \Rightarrow\ 
 q\in\bx\,.
 \end{align*}
Denote by~$\gf$ the corresponding closure operator. The nontrivial coverings in~$(P,\gf)$ are given by
 \[
 \set{a,p_0},\set{a,p_1}\in\cM(p)\text{ and }
 \set{p_0,p_1},\set{b_0,b_1}\in\cM(q)\,.
 \]
It follows that $(P,\gf)$ is a finite closure space of semilattice type.
 
Set $\ba=\set{p,p_0,p_1,q}$. Then $\cgf(\ba)=\set{p,p_0,p_1}$  and
$\gf\cgf(\ba)=\ba$, so~$\ba$ is regular closed. Moreover,
$\ba\setminus\set{p}=\set{p_0,p_1,q}$ is regular closed (and not open), and every regular closed proper subset of~$\ba$ is contained in $\ba\setminus\set{p}$. Hence~$\ba$ is \jirr\ in $\Reg(P,\gf)$ and $\ba_*=\ba\setminus\set{p}$. The regular closed set~$\ba$ is not open, as $q\in\ba$ while $b_0\notin\ba$ and $b_1\notin\ba$.
\end{proof}

For a related counterexample, arising from the context of graphs ($\sR(G)$ and~$\sP(G)$), see Section~\ref{S:NonClopNbhd}.

The following modification of Example~\ref{Ex:FinPosClop} shows that Theorem~\ref{T:AbundClop} cannot be extended to the non well-founded case. 

\begin{examplepf}\label{Ex:InfPosClop}
A closure space $(P,\gf)$ of semilattice type, with clopen subsets~$\ba_0$ and~$\ba_1$ with nonempty meet in~$\Reg(P,\gf)$ and with empty meet in~$\Clop(P,\gf)$.
\end{examplepf}

\begin{proof}
Denote by~$P$ the poset represented on the left hand side of Figure~\ref{Fig:InfPosClop}. A detail of~$P$ is shown on the right hand side of Figure~\ref{Fig:InfPosClop}. The index~$\xi$ ranges over the set~$\SS$ of all finite sequences of elements of~$\set{0,1}$, and $\xi 0$ (resp., $\xi 1$) stands for the concatenation of~$\xi$ and~$0$ (resp., $1$). We leave a formal definition of~$P$ to the reader.

\begin{figure}[htb]
\includegraphics{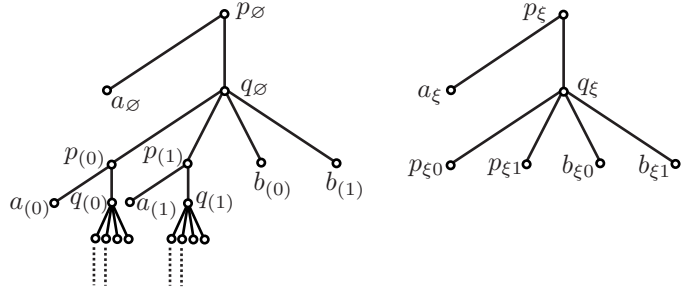}
\caption{The poset $P$ of 
Example \ref{Ex:InfPosClop}}\label{Fig:InfPosClop}
\end{figure}

Now a subset~$\bx$ of~$P$ is closed if
 \begin{align*}
 \set{a_{\xi},p_{\xi i}}\subseteq\bx&\ \Rightarrow\ 
 p_{\xi}\in\bx\,,
 &&\text{for each }\xi\in\SS
 \text{ and each }i\in\set{0,1}\,,\\
 \set{p_{\xi 0},p_{\xi 1}}\subseteq\bx&
 \ \Rightarrow\ 
 q_{\xi}\in\bx\,,
 &&\text{for each }\xi\in\SS\,,\\
 \set{b_{\xi 0},b_{\xi 1}}\subseteq\bx&
 \ \Rightarrow\ 
 q_{\xi}\in\bx\,,
 &&\text{for each }\xi\in\SS\,.
 \end{align*}
We denote by~$\gf$ the corresponding closure operator. We set $\bp=\setm{p_{\xi}}{\xi\in\SS}$, $\bq=\setm{q_{\xi}}{\xi\in\SS}$, and $\ba=\bp\cup\bq$. Then $\cgf(\ba)=\bp$ and $\gf\cgf(\ba)=\ba$, so~$\ba$ is regular closed.

We claim that~$\ba$ has no nonempty clopen subset. For let $\bx\subseteq\ba$ be clopen. If $q_{\xi}\in\bx$ for some~$\xi\in\SS$, then (as $\set{b_{\xi 0},b_{\xi 1}}\in\cM(q_{\xi})$) $\set{b_{\xi 0},b_{\xi 1}}$ meets~$\bx$, thus it meets~$\ba$, \contr; whence $\bx\subseteq\bp$. If $p_{\xi}\in\bx$ for some $\xi\in\SS$, then (as $\set{a_{\xi},p_{\xi i}}\in\cM(p_{\xi})$ for each $i<2$) $\set{p_{\xi 0},p_{\xi 1}}\subseteq\bx$, thus $q_{\xi}\in\bx$, \contr. This proves our claim.

Now the sets $\ba_i=\ba\cup\setm{b_{\xi i}}{\xi\in\SS}$, for
$i\in\set{0,1}$, are both clopen and they intersect in~$\ba$. By the
paragraph above, the meet of $\set{\ba_0,\ba_1}$ in $\Clop(P,\gf)$ is the empty set. However, the meet of $\set{\ba_0,\ba_1}$ in $\Reg(P,\gf)$ is~$\ba$.
\end{proof}

\begin{note}
  For the closure space $(P,\gf)$ of Example~\ref{Ex:InfPosClop}, the
  poset $\Clop(P,\gf)$ is not a lattice. Indeed, for each $i<2$, the
  subset $\bb_i=(P\dnw\set{p_{(0)},p_{(1)}})\cup\set{q_\es,b_{(i)}}$
  is clopen, $\ba_i\subseteq\bb_j$ for all $i,j<2$, and there is no
  clopen subset~$\bc$ of~$P$ such that
  $\ba_i\subseteq\bc\subseteq\bb_j$ for all $i,j<2$.
\end{note}

The following example shows that the lattice of all regular closed subsets of a closure space of semilattice type may not be spatial.

\begin{examplepf}\label{Ex:SemTypeNotSp}
An infinite closure space $(P,\varphi)$ of semilattice type such that
\begin{enumerate}
\item $P$ is a \js\ with largest element.

\item $\Reg(P,\gf)=\Clop(P,\gf)$.

\item The poset $\Clop P$ \pup{cf. Example~\textup{\ref{Ex:FromJSemil}}} is not a lattice.

\item None of the lattices~$\Reg P$ \pup{cf. Example~\textup{\ref{Ex:FromJSemil}}} and ~$\Reg(P,\varphi)$ has any completely \jirr\ element.
\end{enumerate}
\end{examplepf}

\begin{proof}
  We denote by~$P$ the set of all finite sequences of elements of
  $\set{0,1,2}$. For $p,q\in P$, let $p\leq q$ hold if~$q$ is a prefix of $p$. Observe, in particular, that~$P$ is a \js, with largest element the empty sequence~$\es$. The join of any subset~$\bx$ of~$P$ is the longest common prefix for all elements of~$\bx$.

Say that a subset~$\bx$ of~$P$ is closed if
 \[
 \set{pi,pj}\subseteq\bx\ 
 \Rightarrow\ p\in\bx\,,\quad
 \text{for all }p\in P
 \text{ and all distinct }
 i,j\in\set{0,1,2}\,,
 \]
and denote by~$\gf$ the associated closure operator. It is easy to verify that $(P,\gf)$ is a closure space of semilattice type.

\setcounter{claim}{0}
\begin{claim}\label{Cl:UUaiclosed}
Let $p\in P$ and let $\ba_i\subseteq P\dnw pi$ be closed, for each $i\in\set{0,1,2}$. Set
 \[
 X=\setm{i<3}{pi\in\ba_i}\,.
 \]
If $\bigcup_{i<3}\ba_i$ is not closed, then $\card X\geq 2$.
\end{claim}

\begin{cproof}
If~$\ba=\bigcup_{i<3}\ba_i$ is not closed, then there are $q\in P$ and $i\neq j$ such that $\set{qi,qj}\subseteq\ba$ but $q\notin\ba$. If $\set{qi,qj}\subseteq\ba_k$, for some $k<3$, then, as~$\ba_k$ is closed, we get $q\in\ba_k\subseteq\ba$, a contradiction. It follows that there are $i'\neq j'$ such that $qi\in\ba_{i'}$ and $qj\in\ba_{j'}$. {}From $\ba_{i'}\subseteq P\dnw pi'$ it follows that~$qi$ extends~$pi'$, thus, as it also extends~$q$, the finite sequences~$q$ and~$pi'$ are comparable (with respect to~$\leq$). Likewise, $q$ and~$pj'$ are comparable. Since $i'\neq j'$, it follows that~$p$ extends~$q$. Since~$qi$ extends~$pi'$ and~$qj$ extends~$pj'$, it follows that $p=q$, $i=i'$, and $j=j'$. Therefore, $pi\in\ba_i$ and $pj\in\ba_j$, so $\set{i,j}\subseteq X$.
\end{cproof}

\begin{claim}\label{Cl:cfgclosclos}
If~$\ba$ is closed, then $\cgf(\ba)$ is closed, for any $\ba\subseteq P$.
\end{claim}

\begin{cproof}
Let $p\in P$ and let $i\neq j$ in $\set{0,1,2}$ such that $\set{pi,pj}\subseteq\cgf(\ba)$, we must prove that $p\in\cgf(\ba)$. {}From $\set{pi,pj}\subseteq\cgf(\ba)\subseteq\ba$, together with~$\ba$ being closed, it follows that $p\in\ba$. Suppose that $p\notin\cgf(\ba)$. There exists $\bx\in\cM(p)$ such that $\bx\cap\ba=\es$; from $p\in\ba$ it follows that $\bx\subseteq P\ddnw p$, so $\bx=\bigcup_{k<3}(\bx\dnw pk)$. The set~$\ba_k=\gf(\bx\dnw pk)$ is a closed subset of $P\dnw pk$, for each $k<3$. If $\bb=\bigcup_{k<3}\ba_k$ is closed, then, as $\bx\subseteq\bb$ and $p\in\gf(\bx)$, we get $p\in\bb$, so $p\leq pk$ for some $k<3$, a contradiction. Hence~$\bb$ is not closed, so, by Claim~\ref{Cl:UUaiclosed} above, there are distinct $i'\neq j'$ in $\set{0,1,2}$ such that $pi'\in\ba_{i'}$ and $pj'\in\ba_{j'}$. Pick $k\in\set{i,j}\cap\set{i',j'}$. Then $pk\in\cgf(\ba)$ and $pk\in\gf(\bx\dnw pk)$, hence $(\bx\dnw pk)\cap\ba\neq\es$, and hence $\bx\cap\ba\neq\es$, thus completing the proof that $p\in\cgf(\ba)$.
\end{cproof}

It follows from Claim~\ref{Cl:cfgclosclos} that
$\Reg(P,\gf)=\Clop(P,\gf)$.

Suppose that $\Reg P$ (that is, the lattice of all regular closed subsets over the closure space $(P,\tcl)$ defined in Example~\ref{Ex:FromJSemil}) has a completely \jirr\ element~$\ba$. By Corollary~\ref{C:MinNbhdSemil2},
$\ba$ is clopen. Since $(P,\tcl)$ is a closure space of semilattice
type, it follows from Lemma~\ref{L:DescrJirrRegP} that~$\ba$ has a
largest element~$p$, with $\ba_*=\ba\setminus\set{p}$. For all $i\neq
j$ in $\set{0,1,2}$, it follows from $p\in\cgf(\ba)$ and
$p\in\tcl\set{pi,pj}$ that $\set{pi,pj}\cap\ba\neq\es$. Since this
holds for every possible choice of $\set{i,j}$, it follows that there
are distinct $i,j\in\set{0,1,2}$ such that $\set{pi,pj}\subseteq\ba$,
so $\set{pi,pj}\subseteq\ba\setminus\set{p}=\ba_*$, and so, since~$\ba_*$ is closed, $p\in\ba_*$, a contradiction.

An argument similar to the one of the paragraph above shows that~$\Reg(P,\gf)$ has no completely \jirr\ element either.

The subsets $\ba_i=P\dnw\set{i0}$, $\bb_0=(P\dnw\set{0,10,2})\cup\set{\es}$, and $\bb_1=(P\dnw\set{0,1})\cup\set{\es}$ are all clopen, with $\ba_i\subseteq\bb_j$ for all $i,j<2$. However, suppose that there exists a clopen subset~$\bc\subseteq P$ such that $\ba_i\subseteq\bc\subseteq\bb_j$ for all $i,j<2$. Since $\set{00,10}\subseteq\ba_0\cup\ba_1\subseteq\bc$ and~$\bc$ is closed, $\es=00\vee10$ belongs to~$\bc$. Since $\es=1\vee2$ and~$\bc$ is open, it follows that $\set{1,2}$ meets~$\bc$, thus it meets $\bb_0\cap\bb_1$, a contradiction as $1\notin\bb_0$ and $2\notin\bb_1$. Therefore, $\Clop P$ is not a lattice.
\end{proof}

Our next example shows that Theorem~\ref{T:ClopWFLatt} cannot be extended from closure spaces of semilattice type to closure spaces of poset type. That is, in poset type, even if $\Clop(P,\gf)$ is a lattice, it may not be a sublattice of $\Reg(P,\gf)$.

\begin{example}\label{Ex:PosetClEx}
Let $P=\set{a_0,a_1,a,\top}$ be the poset represented on the left hand side of Figure~\ref{Fig:RegNotClop}, and set
 \[
 \gf(\bx)=\begin{cases}
 \bx\cup\set{\top}\,,&\text{if either }
 \set{a_0,a_1}\subseteq\bx
 \text{ or }a\in\bx\,,\\
 \bx\,,&
 \text{otherwise}
 \end{cases}\,,
 \qquad\text{for each }
 \bx\subseteq P\,.
 \]
 
\begin{figure}[htb]
\includegraphics{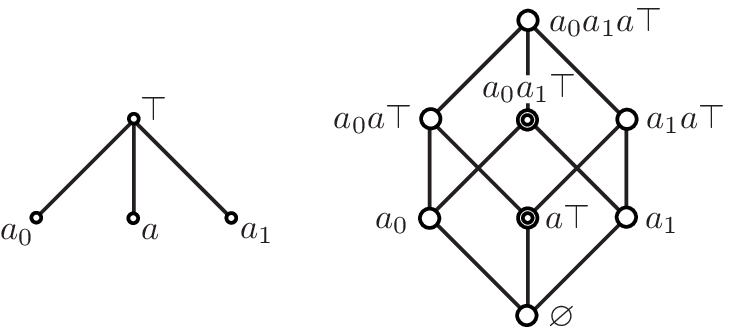}
\caption{The poset~$P$ and the containment $\Clop(P,\gf)\subsetneqq\Reg(P,\gf)$}
\label{Fig:RegNotClop}
\end{figure}

It is straightforward to verify that~$\gf$ is a closure operator on~$P$. Furthermore, the nontrivial minimal coverings in~$(P,\gf)$ are exactly those given by the relations $\top\in\gf(\set{a_0,a_1})$ and $\top\in\gf(\set{a})$. Since~$a_0$, $a_1$, and~$a$ are all smaller than~$\top$, it follows that $(P,\gf)$ is a closure space of poset type.

We represent the eight-element Boolean lattice $L=\Reg(P,\gf)$ on the
right hand side of Figure~\ref{Fig:RegNotClop} (the labeling is given
by $\set{a_0}\mapsto a_0$, $\set{a_0,a,\top}\mapsto a_0a\top$, and so
on). The poset $K=\Clop(P,\gf)$ is the six-element ``benzene
lattice''. The two elements of $L\setminus K$ (viz. $\set{a,\top}$ and
$\set{a_0,a_1,\top}$) are marked by doubled circles on the right hand
side of Figure~\ref{Fig:RegNotClop}.

Observe that~$\set{a_0}$ and~$\set{a_1}$ are both clopen, and that
 \begin{align*}
 \set{a_0}\vee\set{a_1}&=
 \set{a_0,a_1,\top}&&\text{in }
 \Reg(P,\gf)\,,\\
 \set{a_0}\vee\set{a_1}&=
 \set{a_0,a_1,a,\top}&&\text{in }
 \Clop(P,\gf)\,.
 \end{align*}
In particular, $\Clop(P,\gf)$ is not a sublattice of $\Reg(P,\gf)$.
\end{example}

\section{A non-clopen \jirr\ regular open set in a finite graph}\label{S:NonClopNbhd}

The present section will be devoted to the description of a counterexample showing that Corollary~\ref{C:CJIBGorCyc} cannot be extended to arbitrary finite graphs. We shall denote by~$H$ the graph denoted, in the online database \url{http://www.graphclasses.org/}, by $\cK_{3,3}-e$, labeled as on Figure~\ref{Fig:K33-e}.

\begin{figure}[htb]
\includegraphics{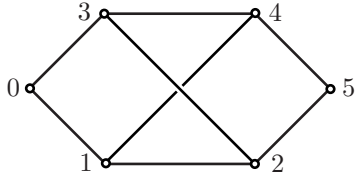}
\caption{The graph~$H$}\label{Fig:K33-e}
\end{figure}

We skip the braces and commas in denoting the subsets of~$H$, and we set
 \begin{multline}\label{Eq:Descrbv}
 \bv=\{1,3,5,01,03,12,34,013,123,125,134,
 145,235,345,\\
 0123,0134,0145,0235,1235,1345,12345,
 01345,01235,012345\}\,. 
 \end{multline}
The elements of~$\bv$ are surrounded by boxes in Table~\ref{Ta:buinK33-e234}.
Not every nonempty subset of~$H$ belongs to~$\so{H}$ (e.g., $02$), in which case we mark it as such (e.g., \OUT{02}). The subset $1234=12\prt 34$ belongs to~$\tcl(\bv)\setminus\bv$.

In order to facilitate the verification of the proof of Theorem~\ref{T:NonClopjirr}, we list the elements of $\bv^\cpl=\so{H}\setminus\bv$:
 \begin{multline}\label{Eq:Descrbvcpl}
 \bv^\cpl=\{0,2,4,14,23,25,45,012,014,023,034,124,234,245,\\
0124,0125,0234,0345,1234,1245,2345,01234,01245,02345\}\,. 
 \end{multline}

\begin{theorem}\label{T:NonClopjirr}
  The set~$\bv$ is a minimal open, and not closed, neighborhood of~$H$, with $\tcl(\bv)=\bv\cup\set{1234}$. Furthermore, $\bv$ is \jirr\ in $\Regop(\so{H},\tcl)$.
\end{theorem}

\begin{proof}
We first verify, by using~\eqref{Eq:Descrbvcpl}, that~$\bv^\cpl$ is closed. In order to do this, it is sufficient to verify that whenever~$X$, $Y$, $Z$ are nonempty connected subsets of~$H$ such that $Z=X\prt Y$, then $\set{X,Y}\subseteq\bv^\cpl$ implies that $Z\in\bv^\cpl$. We thus obtain that~$\bv$ is open.

Likewise, by using~\eqref{Eq:Descrbv}, we verify that $\bv\cup\set{1234}$ is closed. Since $1234=12\prt 34$ belongs to~$\tcl(\bv)$, it follows that $\tcl(\bv)=\bv\cup\set{1234}$. Moreover, $1234=14\prt 23$ in~$\so{H}$ with $14,23\notin\tcl(\bv)$, thus $1234\notin\tin\tcl(\bv)$, and thus, since~$\bv$ is open, it follows that $\tin\tcl(\bv)=\bv$, that is, $\bv$ is regular open.

By definition, $H=012345$ belongs to~$\bv$. We verify, by using
Proposition~\ref{P:MinNbhd}, that~$\bv$ is a minimal neighborhood
of~$H$. For each $X\in\bv$, we need to find $\bx\in\cM(H)$ such that
$\bx\cap\bv=\set{X}$. If $X=H$, take $\bx=\set{H}$. If $X=123$, take
$\bx=\set{123,0,45}$. If $X=134$, take $\bx=\set{134,0,25}$. If
$X=1235$, take $\bx=\set{1235,0,4}$. If $X=1345$, take
$\bx=\set{1345,0,2}$. In all other cases, $H\setminus X$ belongs
to~$\bv^\cpl$, so we can take $\bx=\set{X,H\setminus X}$.

Since~$\bv$ is a minimal neighborhood of~$H$, it is, \emph{a fortiori}, a minimal element of the set of all regular open neighborhoods of~$H$. In order to verify that~$\bv$ is \jirr\ in $\Regop(\so{H},\tcl)$, it suffices to verify that~$H$ is irreducible in~$\bv$, that is, that there is no partition of the form $H=X\prt Y$ with $X,Y\in\tcl(\bv)$. This can be easily checked on Table~\ref{Ta:buinK33-e234}.
\end{proof}

Even without invoking Proposition~\ref{P:MinNbhd}, it is easy to
verify directly that~$\bv$ contains no clopen neighborhood
of~$H$. Suppose, to the contrary, that~$\ba$ is such a clopen
neighborhood. Since $H=12\prt 0345$ with $H\in\ba$ and
$0345\notin\bv$ (thus $0345\notin\ba)$, it follows
from the openness of~$\ba$ that $12\in\ba$. Likewise,
$34\in\ba$. Since~$\ba$ is closed, it follows that $1234=12\prt 34$
belongs to~$\ba$, thus to~$\bv$, a contradiction.

\begin{corollary}\label{C:NonClopjirr}
The extended permutohedron~$\sR(H)$ is not the Dedekind-MacNeille completion of the permutohedron~$\sP(H)$.
\end{corollary}

\begin{table}[htb]
\begin{tabular}{| c | c | c | c | c | c |}
\hline
0 & \BX{1} & 2 & \BX{3} & 4 & \BX{5}\\ \hline
\BX{12345} & 02345 & \BX{01345} & 
01245 & \BX{01235} & 01234\\ \hline
\end{tabular}
\begin{tabular}{| c | c | c | c | c |}
\hline
\BX{01} & \OUT{02} & \BX{03} & \OUT{04} & \OUT{05}\\ \hline
2345 & \BX{1345} & 1245 & 
\BX{1235} & 1234\\ \hline\hline
\BX{12} & \OUT{13} & 14 & \OUT{15} & 23\\
\hline
0345 & \OUT{0245} & \BX{0235} & 0234
& \BX{0145}\\ \hline\hline
\OUT{24} & 25 & \BX{34} & \OUT{35} & 45\\
\hline
\OUT{0135} & \BX{0134} & 0125 & 0124 &
\BX{0123}\\ \hline\hline
012 & \BX{013} & 014 & \OUT{015} & 023\\
\hline
\BX{345} & 245 & \BX{235} & 234 &
\BX{145}\\ \hline\hline
\OUT{024} & \OUT{025} & 034 & \OUT{035}
& \OUT{045}\\ \hline
\OUT{135} & \BX{134} & \BX{125} & 124 &
\BX{123}\\ \hline
\end{tabular}
\caption{\tvi Nonempty proper members of~$\so{H}$; members of~$\bv$ boxed}\label{Ta:buinK33-e234}
\end{table}

\section{A non-clopen minimal regular open neighborhood}\label{S:NonClopReg}

For any positive integer~$n$, the complete graph~$\cK_n$ is a block graph, hence~$\sR(\cK_n)$ is the Dedekind-MacNeille completion of~$\sP(\cK_n)$. (This follows from Corollary~\ref{C:CJIBGorCyc}; however, invoking Lemma~\ref{L:JIalmOpen} is even easier.) While the corresponding results for transitive binary relations (cf. Santocanale and Wehrung~\cite{SaWe12a}) and for \js{s} (cf. Corollary~\ref{C:MinNbhdSemil2}) are obtained \emph{via} the stronger result that every \emph{open} set is a union of clopen sets, we shall prove in this section that for~$n$ large enough (i.e., $n\geq7$), not even every \emph{regular open} subset of~$\so{\cK_n}$ is a union of clopen sets.

In what follows, we shall label the vertices of $G=\cK_7$ from~$0$ to~$6$, and describe the construction of a regular open subset of~$\so{G}$ that is not a set-theoretical union of clopen sets. It can be proved, after quite lengthy calculations, that~$7$ is the smallest integer with that property: for each $n\leq 6$, every regular open subset of~$\so{\cK_n}$ is a set-theoretical union of clopen sets.

\begin{theorem}\label{T:hostile7}
There exists a minimal neighborhood~$\bu$ of~$G=\cK_7$, which is, in addition, regular open, and such that
 \begin{equation}\label{Eq:X-X7}
 X\in\bu\ \Leftrightarrow\ 
 G\setminus X\notin\bu\,,
 \quad\text{for any }X\subseteq G\,, 
 \end{equation}
together with $Q_0,Q_1,Q_2\in\bu$ such that $G=Q_0\prt Q_1\prt Q_2$. In particular, $\bu$ contains no clopen neighborhood of~$G$.
\end{theorem}

\begin{proof}
As in Section~\ref{S:NonClopNbhd}, we skip the braces in denoting subsets of~$G$, so, for instance, $134$ is short for $\set{1,3,4}$. The~$Q_i$ are defined by
 \[
 Q_0=012\,,\quad Q_1=34\,,\quad
 Q_2=56\,.
 \]
\begin{table}
\begin{tabular}{| c | c | c | c | c | c | c |}
\hline
\BX{0} & 1 & 2 & \BX{3} & 4 & \BX{5} & 6\\ \hline
123456 & \BX{023456} & \BX{013456} & 
012456 & \BX{012356} & 012346 & \BX{012345}\\ \hline\hline\hline
\BX{01} & \BX{02} & \BX{03} & \BX{04} & \BX{05} & \BX{06} & 12 \\ \hline
23456 & 13456 & 12456 & 
12356 & 12346 & 12345 & \BX{03456} \\
\hline\hline
13 & 14 & \BX{15} & 16 & \BX{23} & 24 & 25 \\ \hline
\BX{02456} & \BX{02356} & $02346^\star$ & 
\BX{02345} & $01456^\star$ & \BX{01356} & \BX{01346} \\
\hline\hline
26 & \BX{34} & \BX{35} & 36 & 45 & 46 & \BX{56} \\ \hline
\BX{01345} & $01256^\star$ & 01246 & \BX{01245} & 
\BX{01236} & \BX{01235} & $01234^\star$ \\ \hline\hline\hline
\BX{012} & \BX{013} & \BX{014} & \BX{015} & 016 & \BX{023} & 024 \\ \hline
$3456^\star$ & 2456 & $2356^\star$ & 
2346 & \BX{2345} & 1456 & \BX{1356} \\
\hline\hline
\BX{025} & \BX{026} & \BX{034} & \BX{035} & \BX{036} & \BX{045} & \BX{046} \\ \hline
1346 & $1345^\star$ & 1256 & 
1246 & 1245 & 1236 & $1235^\star$ \\
\hline\hline
\BX{056} & 123 &124 & 125 & 126 & 134 & \BX{135} \\ \hline
1234 & \BX{0456} & \BX{0356} & \BX{0346} & 
\BX{0345} & \BX{0256} & 0246 \\
\hline\hline
136 & 145 & 146 & 156 & 234 & \BX{235} & 236 \\ \hline
\BX{0245} & \BX{0236} & \BX{0235} & \BX{0234} & 
\BX{0156} & 0146 & \BX{0145} \\
\hline\hline
245 & 246 & 256 & \BX{345} & 346 & \BX{356} & 456 \\ \hline
\BX{0136} & \BX{0135} & \BX{0134} & 0126 & 
\BX{0125} & 0124 & \BX{0123} \\
\hline
\end{tabular}
\caption{\tvi Proper subsets of~$G$; members of~$\bu$ boxed, members of $\tcl(\bu)\setminus\bu$ marked by asterisks}\label{Ta:PperSubs7}
\end{table}
We group complementary pairs of nonempty subsets of~$G$ on Table~\ref{Ta:PperSubs7}, and we box and boldface, on that table, the elements of~$\bu\setminus\set{0123456}$.

Hence, $\bu=\set{0,023456,013456,3,\dots,356,0123,0123456}$. It has~$64$ elements. It is obvious, on the table, to see that $Q_i\in\bu$ for each $i\in\set{0,1,2}$. It is an elementary, although quite horrendous, task to verify that~$\bu$ is open and that the subset
 \[
 \ba=\bu\cup\set{01234,1235,1345,02346,01256,2356,01456,3456}
 \]
is closed. Each element of $\ba\setminus\bu$ is marked by an asterisk on Table~\ref{Ta:PperSubs7}. Each of the decompositions
 \begin{align*}
 01234&=012\prt34\,,&1235&=15\prt23\,,\\
 1345&=15\prt34\,,&02346&=23\prt046\,,\\
 01256&=012\prt56\,,&2356&=23\prt56\,,\\
 01456&=014\prt56\,,&3456&=34\prt56
 \end{align*}
yields a partition of an element of $\ba\setminus\bu$ into elements of~$\bu$; whence $\ba=\tcl(\bu)$. On the other hand,
each of the decompositions
 \begin{align*}
 01234&=13\prt024\,,&1235&=13\prt25\,,\\
 1345&=13\prt45\,,&02346&=36\prt024\,,\\
 01256&=25\prt016\,,&2356&=25\prt36\,,\\
 01456&=45\prt016\,,&3456&=36\prt45
 \end{align*}
yields a partition of an element of $\ba\setminus\bu$ into members not belonging to~$\tcl(\bu)$; whence $\bu=\tin(\tcl(\bu))$, that is, $\bu$ is regular open.

Since there are no complementary pairs of elements of~$\bu$, it follows from Proposition~\ref{P:MinNbhd} that~$\bu$ is a minimal neighborhood of~$G$ in~$\so{G}$.

Since $\tcl(\bu)=\ba\neq\bu$, the subset~$\bu$ is not clopen. Since it is a minimal neighborhood of~$G$, it contains no clopen neighborhood of~$G$.
\end{proof}

The last statement of Theorem~\ref{T:hostile7} can be proved directly, as follows. Let $\ba\subseteq\bu$ be clopen and suppose that $G\in\ba$. {}From $Q_1\prt Q_2\notin\bu$ it follows that $Q_1\prt Q_2\notin\ba$, thus, as $G=Q_0\prt Q_1\prt Q_2$ and~$\ba$ is open, we get $Q_0\in\ba$. Likewise, $Q_1\in\ba$, so, as~$\ba$ is closed, $Q_0\prt Q_1\in\ba$, that is, $01234\in\ba$, so $01234\in\bu$, \contr.

\begin{remark}\label{Rk:K3union}
It is much easier to find, even in~$\cK_3$, an open set which is not a set-theoretical union of clopen sets: just take $\bu=\set{0,1,2,012}$.
\end{remark}

\section{Open problems}\label{S:Pbs}

\begin{problem}\label{Pb:IdPHB}
  Is there a nontrivial lattice-theoretical identity (resp.,
  quasi-identity) that holds in the Dedekind-MacNeille completion of
  the poset of regions of any central hyperplane arrangement? How
  about hyperplane arrangements in fixed dimension?
\end{problem}

Any identity solving the first part of Problem~\ref{Pb:IdPHB} would,
in particular, hold in any permutohedron~$\sP(n)$, which would solve
another problem in Santocanale and Wehrung~\cite{SaWe11}. There is a
nontrivial quasi-identity, in the language of lattices \emph{with
  zero}, holding in the Dedekind-MacNeille completion~$L$ of any
central hyperplane arrangement, namely pseudocomplementedness
(cf. Corollary~\ref{C:DMcNPosHB1}). However, we do not know about
quasi-identities only in the language $(\vee,\wedge)$---we do not even
know whether the quasi-identity~\RSD{1}, introduced in
Section~\ref{S:SDReg}, holds in~$L$. For a related example, see
Example~\ref{Ex:AlmostSDReg2}.

Our next problem asks for converses to Theorems~\ref{T:bsPGBded} and~\ref{T:FinSemilBded}.

\begin{problem}\label{Pb:EmbPbBdedOrth}
Can every finite ortholattice, which is a bounded homomorphic image of a free lattice, be embedded into~$\sR(G)$ for some finite graph~$G$ \pup{resp., into~$\Reg{S}$ for some finite \js~$S$}?
\end{problem}

In Example~\ref{Ex:AlmostSDReg2}, we find a finite convex geometry whose lattice of regular closed subsets contains a copy of~$\sL_4$ (cf. Figure~\ref{Fig:NonSD}), thus fails the quasi-identity~\RSD{1} introduced in Section~\ref{S:SDReg}. However, this example also contains a copy of~$\sL_1$. This suggests the following problem.

\begin{problem}\label{Pb:CGL1}
Let $(P,\gf)$ be a finite convex geometry. If $\Reg(P,\gf)$ fails semidistributivity, does it necessarily contain a copy of~$\sL_1$?
\end{problem}

By Theorem~\ref{T:AlmostSDReg2}, Problem~\ref{Pb:CGL1} has a positive answer for $(P,\gf)$ of poset type.

Several results of the present paper state the boundedness of lattices of regular closed subsets of certain closure spaces. Permutohedra (on finite chains) are such lattices (cf. Caspard~\cite{Casp00}). The latter result was extended in Caspard, Le Conte de Poly-Barbut, and Morvan~\cite{CLM04} to all finite \emph{Coxeter lattices} (i.e., finite Coxeter groups with weak Bruhat ordering). Now to every finite Coxeter group is associated a so-called \emph{Dynkin diagram}, which is a \emph{tree}.

\begin{problem}\label{Pb:CoxLatt}
Relate an arbitrary finite Coxeter lattice~$L$ to the permutohedron on the corresponding Dynkin diagram (or a related graph). Can~$L$ be described as $\Reg(P,\gf)$, for a suitable closure system $(P,\gf)$ of semilattice type?
\end{problem}

For type~A the answer to Problem~\ref{Pb:CoxLatt} is well-known, as we
just get the usual permutohedra. For other types, the situation looks
noticeably more complicated. For example, let
$\mathcal{D}_{4}$ be the graph arising
from the Dynkin diagram of type $D_{4}$. Thus~$\mathcal{D}_{4}$ is a
star with three leaves (and one center). The lattice
$\sP(\mathcal{D}_{4})$, represented on the left hand side of Figure~\ref{fig:Pstar}, has~$160$ elements, while the Coxeter group arising from that diagram, whose weak Bruhat ordering is represented on the right hand side of Figure~\ref{fig:Pstar}, has~$192$ elements. It can be shown that the smaller lattice is a homomorphic image of the larger one, but that the smaller lattice does not embed into the larger one.

\begin{figure}[htbp]
  \centering
  \includegraphics[scale=0.18]{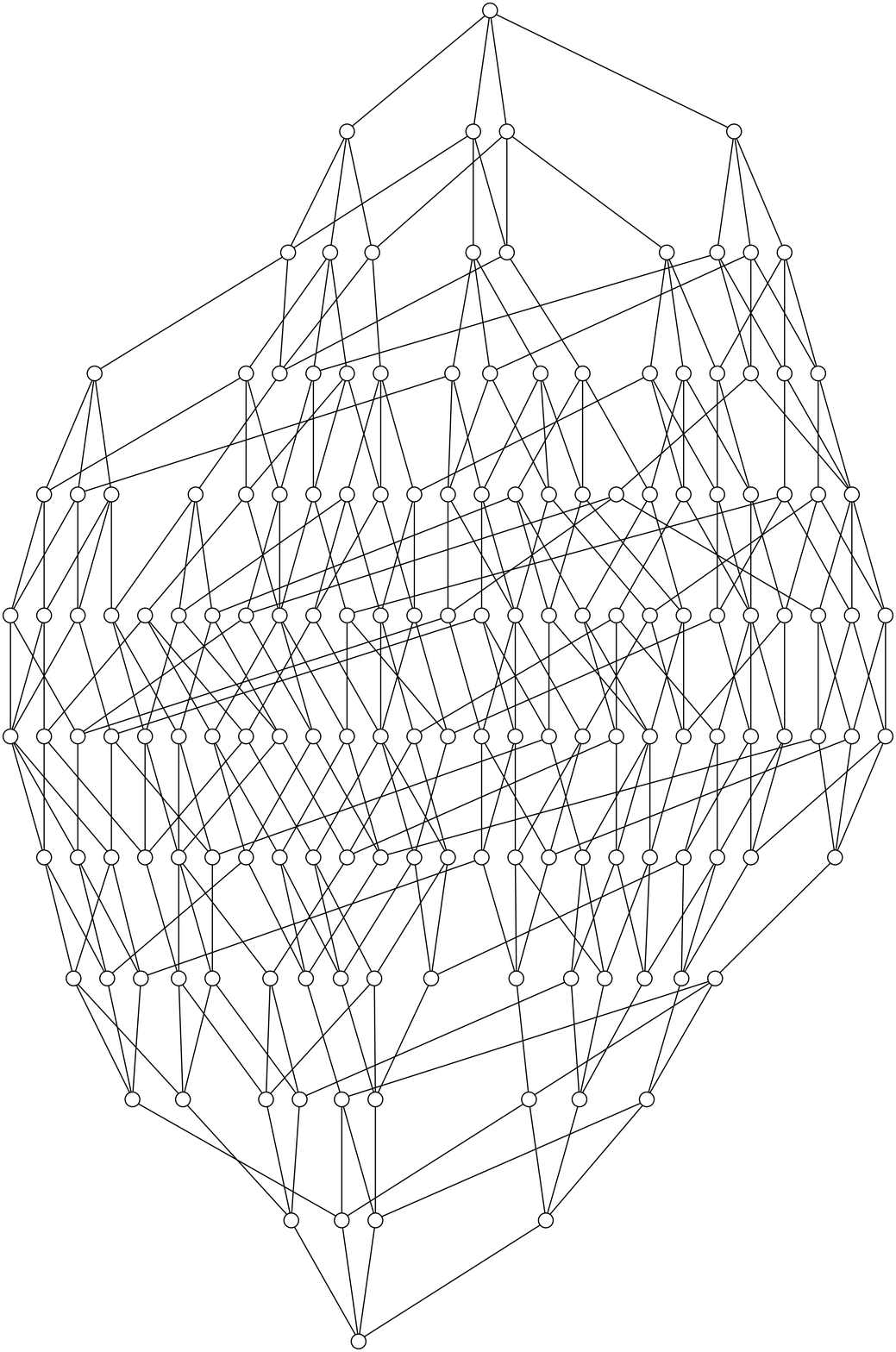}
  \quad
  \includegraphics[scale=0.18]{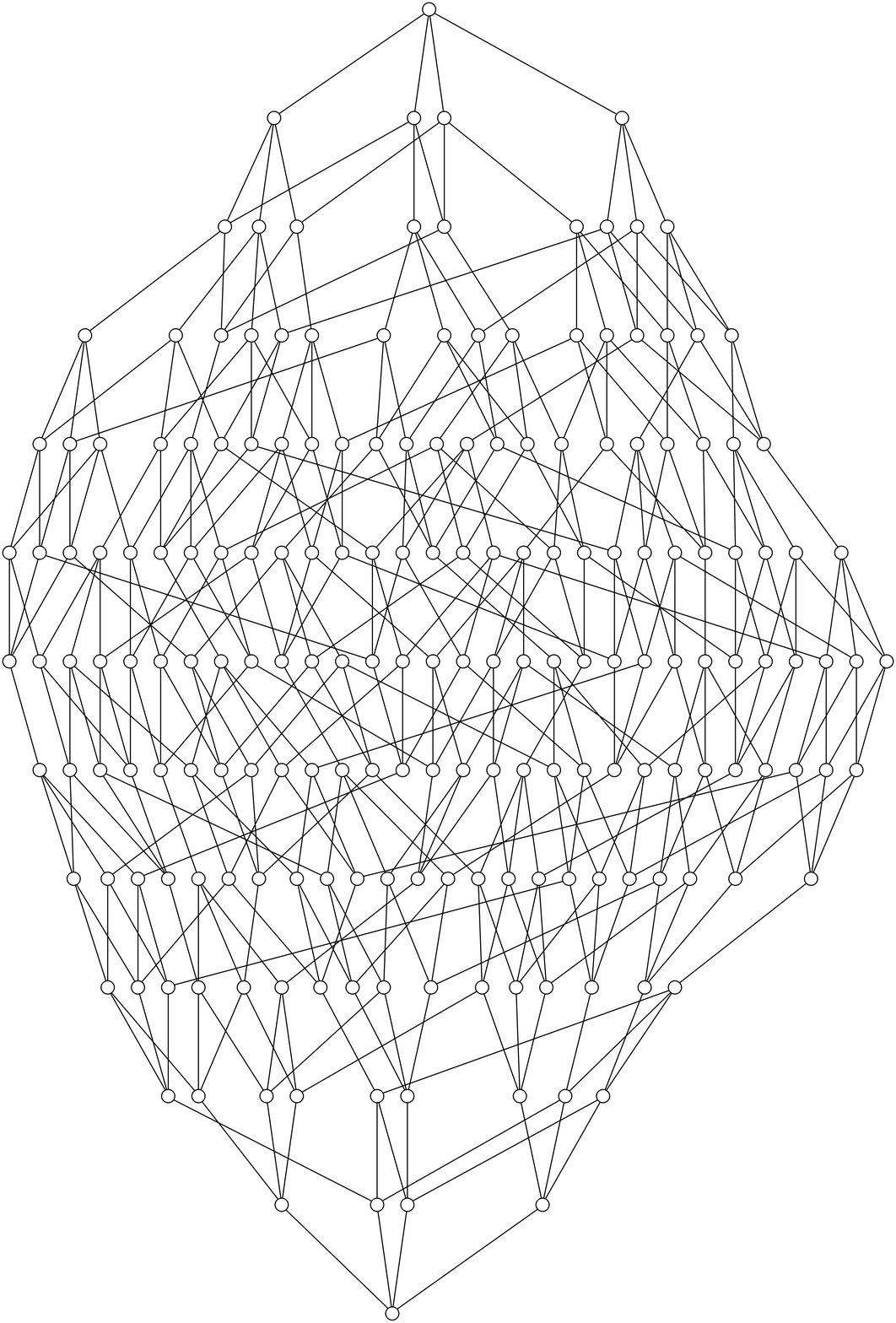}
  \caption{The lattice $\sP(\mathcal{D}_{4})$ and the Coxeter lattice of type~$D_{4}$}
  \label{fig:Pstar}
\end{figure}


\begin{problem}\label{Pb:Spatial}
Let~$G$ be an infinite graph. Is every element of~$\sR(G)$ a join of completely \jirr\ (resp., clopen) elements of~$\sR(G)$?
\end{problem}

A counterexample to the analogue of Problem~\ref{Pb:Spatial}
for~$\Reg{S}$, for a \js~$S$, is given by
Example~\ref{Ex:SemTypeNotSp}. On the other hand, the analogue of
Problem~\ref{Pb:Spatial} for regular closed subsets of transitive
binary relations has a positive answer (cf.  Santocanale and
Wehrung~\cite{SaWe12a}). We do not even know the answer to
Problem~\ref{Pb:Spatial} for $G=\cK_\omega$, the complete graph on a
countably infinite vertex set. As evidence towards the negative, see
Theorem~\ref{T:hostile7}.

\begin{problem}\label{Pb:GraphTight}
  Let~$G$ be a graph. If a set $\setm{\ba_i}{i\in I}$ of clopen
  subsets of~$\so{G}$ has a join in~$\sP(G)$, is this join necessarily
  equal to $\tcl\pI{\bigcup\vecm{\ba_i}{i\in I}}$?
\end{problem}

The finite case of Problem~\ref{Pb:GraphTight} is settled by Theorem~\ref{T:AbundClop}.

\begin{problem}\label{Pb:AbundClop}
Can one remove the well-foundedness assumption from the statement of Theorem~\ref{T:ClopWFLatt}? That is, is $\Clop(P,\gf)$ a lattice if{f} $\Clop(P,\gf)=\Reg(P,\gf)$, for any closure space $(P,\gf)$ of semilattice type?
\end{problem}

Example~\ref{Ex:InfPosClop} suggests a negative answer to Problem~\ref{Pb:AbundClop}, while Corollary~\ref{C:MinNbhdSemil3} (dealing with \js{s}) and Theorem~\ref{T:PGLatt} (dealing with graphs) both suggest a positive answer to Problem~\ref{Pb:AbundClop}.

\begin{problem}\label{Pb:EqThPG}
Is there a nontrivial lattice identity that holds in $\sR(G)$ for every finite graph~$G$ (resp., in $\Reg S$ for every finite \js\ $S$)?
\end{problem}

Some ideas about Problem~\ref{Pb:EqThPG} may be found in Santocanale and Wehrung~\cite{SaWe11}.

\begin{problem}\label{Pb:SubgraphDMcN}
Let~$G$ be an induced subgraph of a graph~$H$. If~$\sR(H)$ is the Dedekind-MacNeille completion of~$\sP(H)$, is~$\sR(G)$ the Dedekind-MacNeille completion of~$\sP(G)$?
\end{problem}

\section{Acknowledgment}\label{S:Acknow}
The authors are grateful to William McCune for his
\texttt{Prover9-Mace4} program~\cite{McCune}, to Ralph Freese for his lattice drawing program~\cite{LatDraw}, and to the authors of the Graphviz framework, available online at \url{http://www.graphviz.org/}. Part of this work was completed while the second author was visiting the CIRM in March 2012. Excellent conditions provided by the host institution are greatly appreciated.

\end{document}